\documentclass[reqno,12pt,twoside]{amsart}
\usepackage{amsmath}
\usepackage{amssymb}
\usepackage[dvips]{graphicx}
\numberwithin{equation}{section}

\begingroup

\newtheorem{thm}{Theorem}[section]

\newtheorem{cor}{Corollary}
\newtheorem{lemma}{Lemma}[section]

\newtheorem{prop}{Proposition}[section]

\endgroup

\newenvironment{pf}{\par \medskip \noindent{\it Proof:}}
{\nopagebreak \hfill \qed \par \medskip}
\newcommand{\gl}{\mathop{GL}}

\newcommand{\bC}{{\mathbb C}}
\newcommand{\bA}{{\mathbb A}}
\newcommand{\bR}{{\mathbb R}}

\newcommand{\bZ}{{\mathbb Z}}

\newcommand{\bF}{{\mathbb F}}
\newcommand{\bE}{{\mathbb E}}

\newcommand{\cW}{\mathcal W}

\newcommand{\cM}{\mathcal M}

\newcommand{\bs}{\backslash}
\newcommand{\G}{\mathrm{\mathbf G}}
\renewcommand{\H}{\mathrm{\mathbf H}}

\newcommand{\Z}{\mathrm{\mathbf Z}}
\renewcommand{\P}{\mathrm{\mathbf P}}
\newcommand{\N}{\mathrm{\mathbf N}}
\newcommand{\M}{\mathrm{\mathbf M}}
\newcommand{\U}{\mathrm{\mathbf U}}
\newcommand{\B}{\mathrm{\mathbf B}}
\newcommand{\A}{\mathrm{\mathbf A}}

\newcommand{\Gal}{\mathrm{Gal}}
\newcommand{\Ind}{\mathrm{Ind}}
\newcommand{\Hom}{\mathrm{Hom}}

\newcommand{\veps}{\varepsilon}
\newcommand{\la}{\langle}
\newcommand{\ra}{\rangle}

\renewcommand{\Im}{\mathop{Im}}

\newcommand{\Sym}{\mathrm{Sym}}
\newcommand{\diag}{\mathrm{diag}}

\newcommand{\be}{\begin{enumerate}}
\newcommand{\ee}{\end{enumerate}}
\newcommand{\bpm}{\begin{pmatrix}}
\newcommand{\epm}{\end{pmatrix}}
\newcommand{\beq}{\begin{equation}}
\newcommand{\eeq}{\end{equation}}
\newcommand{\wt}{\widetilde}
\newcommand{\dw}{{\dot{w}}}
\newcommand{\tw}{\tilde{w}}
\newcommand{\tJ}{{^t\hspace{-.35ex}J}}
\newcommand{\vphi}{\varphi}
\newcommand{\fo}{\mathfrak o}
\newcommand{\fp}{\mathfrak p}
\newcommand{\cL}{\mathcal L}

\newcommand{\sddots}{\mathinner{\mkern1mu\raise1pt\hbox{.}\mkern2mu
\raise4pt\hbox{.}\mkern2mu\raise7pt\hbox{.}\mkern1mu}}

\begin{document}

\title[Exterior and symmetric square $\varepsilon$--factors]{Local Langlands correspondence for $GL_n$ and the exterior and symmetric square $\varepsilon$--factors}

\author{J.W. Cogdell}
\address{Department of Mathematics, The Ohio State University, Columbus, OH 43210, USA}
\email{cogdell@math.ohio-state.edu}
\author{F. Shahidi} 
\address{Department of Mathematics, Purdue University, West Lafayette,IN 47907, USA}
\email{shahidi@math.purdue.edu}
\author{T-L. Tsai}
\address{Taida Institute for Mathematical Science, 
National Taiwan University, 
No.1, Sec. 4, Roosevelt Road, Taipei, Taiwan 106 }
\email{ttsai2004@gmail.com}
\dedicatory{To the memory of Joe Shalika}

\thanks{JWC was partially supported by NSF grant DMS--0968505. FS was partially supported by NSF grants DMS--0700280 and DMS--1162299.  TLT was partially supported by grant NSC-101-2628-M-002-009 from the National Science Council of Taiwan.}

\begin{abstract}

Let $F$ be a $p$--adic field, i.e., a finite extension of $\mathbb Q_p$ for some prime $p$. The local Langlands correspondence attaches to each
continuous $n$--dimensional $\Phi$-semisimple representation $\rho$ of $W'_F$, the
Weil--Deligne group for $\overline F/F$,  an
 irreducible admissible representation  $\pi(\rho)$ of $GL_n(F)$ such that, among other things, the local $L$- and $\veps$-factors of pairs are preserved.  This correspondence should be robust and preserve various parallel operations on the arithmetic and analytic sides, such as taking the exterior square or symmetric square. In this paper, we  show that  this is the case for the local arithmetic and analytic symmetric square and exterior square $\varepsilon$--factors, that is, that $\veps(s,\Lambda^2\rho,\psi)=\veps(s,\pi(\rho),\Lambda^2,\psi)$ and $\veps(s,\Sym^2\rho,\psi)=\veps(s,\pi(\rho),\Sym^2,\psi)$. The agreement of the $L$-functions also follows by our methods, but this was already known by Henniart.
The proof is a robust deformation argument, combined with local/global techniques, which reduces the problem to the stability of the analytic $\gamma$-factor $\gamma(s,\pi,\Lambda^2,\psi)$  under highly ramified twists when $\pi$ is supercuspidal.
This last step is achieved
by relating the $\gamma$-factor to a Mellin transform of a partial Bessel function attached to the representation and then analyzing the asymptotics of the partial Bessel function, inspired in part by the theory of Shalika germs for Bessel integrals. 
The stability for every irreducible admissible representation $\pi$ then follows from those of the corresponding arithmetic $\gamma$--factors as a corollary.
\end{abstract}

\maketitle

\section{Introduction}

Artin $L$-functions  were introduced by Artin \cite{Ar23,Ar30} to  generalize Weber's factorization of the Dedekind zeta function of an abelian extension of a number field, and as part of his general musings on a non-abelian class field theory. When he then compared this with the abelian case it  led him to formulate   the Artin reciprocity law \cite{Ar23}. For an arbitrary finite dimensional representation of the  Galois group or Weil group they are expected to equal automorphic $L$-functions. In fact, so far this equality has been the only general way to prove Artin's conjecture that the $L$-functions attached to non-trivial irreducible representations are entire. Specifically this is the strategy proposed by Langlands in order to prove the Artin conjecture with considerable success in the case of two dimensional representations \cite{La80,Tu, BDSBT,Ta}.

The Artin $L$-functions satisfy a functional equation. Artin's attempts to understand this functional equation led to his definition of the Artin conductor $\mathfrak f(\rho)$ and the Artin root number $W(\rho)$ which enter into the $\veps$-factor that appears in the functional equation
\[
\veps(s,\rho)=W(\rho)[|d_{k/\mathbb Q}|^{n} N_{k/\mathbb Q}(\mathfrak f(\rho))]^{-(s-\tfrac{1}{2})}
\]
where $\rho$ is an $n$-dimensional  representation of $Gal(\bar{k}/k)$ .  Note that the Artin conductor $\mathfrak f(\rho)$ factors as a product of local conductors $\mathfrak f(\rho_v)$ where $\rho_v$ is the restriction of $\rho$ to the decomposition group at $v$ and $\mathfrak f(\rho_v)$ depends on the further restriction of $\rho_v$ to the higher ramification groups \cite{Ar31}.  

The Artin root number $W(\rho)$ withstood a local factorization  until the  work of Dwork \cite{Dw}, Langlands \cite{La, La70} and finally Deligne \cite{De} gave a factorization of the  $\veps$-factor into local factors
\[
\veps(s, \rho)=\prod_v \veps(s, \rho_v,\psi_v).
\]
Consequently the factor $\veps(s,\rho)$ appearing in the functional equation was now related to intrinsic local data attached to $\rho$ rather than to the global factors coming from abelian $\veps$-factors through the factorization given by Brauer's theorem \cite{Br}, which is essentially how Artin defined the root number $W(\rho)$.

Deligne's proof of the factorization, and in particular his definition of the local factors, itself comes about indirectly as an existence and uniqueness result for his local factors. He postulates a number of desired properties of these local factors and then show that factors satisfying these conditions exist and are unique.  It is of interest that his proof used both local and global techniques and the global functional equation.  Among Deligne's axioms is that of stability: if $\rho_1$ and $\rho_2$ are a pair of local Galois representations with $\det(\rho_1)=\det(\rho_2)$ then for every sufficiently ramified characters (one-dimensional representations of the Weil group) 
\[
\veps(s, \rho_1\otimes\chi,\psi)=\veps(s,\rho_2\otimes\chi,\psi).
\]
 
 These $\veps$-factors play a crucial role in the local Langlands correspondence (LLC) for $GL_n$ \cite{HT, He00}. In fact if $\rho_1$ and $\rho_2$ are two Frobenius semisimple representations of the Weil-Deligne group $W'_F$, for $F$ a non-archimedean local field of characteristic zero, of dimension $n_1$ and $n_2$ which correspond under the LLC to representations $\pi(\rho_1)$ and $\pi(\rho_2)$ of $GL_{n_1}(F)$ and $GL_{n_2}(F)$ then
 \[
 \veps(s,\rho_1\otimes\rho_2,\psi)=\veps(s, \pi(\rho_1)\times \pi(\rho_2), \psi)
 \]
   where the factors on the right are the local Rankin-Selberg $\veps$-factors \cite{JPSS, Sh84}. 
   
   While the LLC is uniquely determined by a number of conditions, including the equality of twisted $\veps$-factors as above, one expects that the correspondence is very robust with respect to all Langlands $\veps$-factors associated to representations $R$ of $GL_n(\mathbb C)$, i.e., we should have
  \[
  \veps(s, R\cdot\rho,\psi)=\veps(s,\pi(\rho),R, \psi)
  \]
  whenever the factors on the right hand side can be defined. When $R=\Lambda^2$ or $R=\Sym^2$ these factors were attached to irreducible admissible representations $\pi$ of $GL_n$ in \cite{Sh90}. As in Deligne, these factors are proved to satisfy a number of axioms that determine them uniquely.
  
    The main result of this paper can be precisely stated as follows. Let $F$ be a $p$--adic field, by which we mean a non-archimedean local field  of characteristic zero, i.e., a finite extension of $\mathbb Q_p$ for some prime $p$, with a fixed algebraic closure $\overline{F}$. Fix a non-trivial additive character $\psi$ of $F$. 
Fix a positive integer $n$.  
Let  $\rho$ be a continuous $n$--dimensional $\Phi$-semisimple representation of $W'_F$, the
Weil--Deligne group for $\overline F/F$.
Let $\pi=\pi(\rho)$ be the irreducible admissible representation of $GL_n(F)$ associated to $\rho$ by the local Langlands correspondence. 

\begin{thm} \label{eps} Let $\Lambda^2$ and $\Sym^2$ denote the exterior and symmetric
square representations of $GL_n(\bC)$.
Let $\veps(s,\pi,\Lambda^2,\psi)$ and $\veps(s,\pi,\Sym^2,\psi)$ be the
$\veps$--factors attached to $\pi$,  $\psi$ and $\Lambda^2$ and $\Sym^2$, as in \cite{Sh90}.
Denote by $\veps(s,\Lambda^2\rho,\psi)$ and $\gamma(s,\Sym^2\rho,\psi)$ the
corresponding Artin factors as in \cite{De}.
Then
\be
\item $\veps(s,\Lambda^2\rho,\psi)=\veps(s,\pi(\rho),\Lambda^2,\psi)$\newline
\noindent and
\item$\veps(s,\Sym^2\rho,\psi)=\veps(s,\pi(\rho),\Sym^2,\psi)$.
\ee
A similar identity holds for the $L$-functions.
\end{thm}

Our proof is based on three basic techniques. The first is additivity of the local factors (usually called multiplicativity in the analytic context). The next is the stability of the local factors on both sides under highly ramified twists. Finally we embed the local situation into a global context and use the global functional equation on both sides. More specifically, we begin with three reductions. Using structure theory we can reduce to proving the equality of local factors when $\rho$ is an irreducible representation of the Weil group and thus the corresponding representation $\pi(\rho)$ is supercuspidal. We then use the LLC to reduce the theorem to the statement about the exterior square  $\veps$-factors. Finally we replace the $\veps$-factors by the related $\gamma$-factors, since this is more convenient for our analytic ingredients. After these reductions, we first prove a stable equality of the local factors under highly ramified twists. The proof of this is inductive on the dimension of $\rho$. The first step is to establish the equality at a base point, i.e., for one specific $\rho_0$ and its twists. This is established by our first globalization, the global functional equation and additivity. The second step is a deformation argument applied to these factors; once we have equality at a base point we can use stability of the individual local factors under highly ramified twists to obtain the stable equality. Once we have this, we use our second globalization, along with the global functional equation, additivity and now our stable equality, to prove the equality of local factors for monomial representations. Then we obtain equality in general by using Brauer's theorem  as in \cite{He00} and additivity once again. As a consequence of the equality of $\gamma$-factors,  we can recover the  equality of $L$-functions, which was originally proved in \cite{He10} using a base change argument

Embedded in the proof above is the use of the stability of both the arithmetic and analytic local factors. The arithmetic stability is due to Deligne as mentioned above. However the crucial analytic stability for the exterior square $\gamma$-factors is newly established in this paper. Let us give a precise statement here.

\medskip

\noindent{\bf Analytic stability for supercuspidals} (Proposition \ref{scstab})  {\it  Let $\pi_1$ and $\pi_2$ be
  two irreducible supercuspidal representations of $GL_n(F)$ 
with the same central characters.
Then for all the suitably highly ramified characters $\chi$ of $F^\times$,
identified as characters of $GL_n(F)$ through the determinant, we have 
$$
\gamma(s,\pi_1\otimes\chi,\Lambda^2,\psi)=\gamma(s,\pi_2\otimes\chi,\Lambda^2,\psi).
$$
The required degree of ramification depends only on $\pi_1$ and $\pi_2$.}

\medskip

 The proof of this result is rather lengthy and occupies the last two sections  of this paper. The inverse of the local coefficient $C_\psi(s,\pi)$, by means of which $\gamma(s,\tilde{\pi}, \Lambda^2,\psi^{-1})$ is defined  through equation \eqref{lcg}, has an integral representation as a Mellin transform of a ``partial'' Bessel function.  This partial Bessel function can be expressed as a partial  Bessel integral $B^G_\vphi(g,f)$ on $G=GL_n(F)$ defined by a  matrix coefficient $f\in C^\infty_c(G,\omega_\pi)$ of $\pi$ and a suitable cutoff function $\vphi$, giving the ``partial'' nature of the integral, as dictated by the main result of \cite{Sh02}. To prove the stability for  supercuspidals one needs to determine the asymptotic behavior of $B^G_\vphi(g,f)$, much as we did in \cite{CPSS05, CPSS08}, and here is where the bulk of the paper lies. Inspired by the germ expansion of Jacquet and Ye  \cite{JY96,J12} for certain orbital integrals, which are in fact Bessel integrals when $\pi$ is supercuspidal, we establish the asymptotics in Proposition \ref{unifsmooth}. Our arguments for establishing the uniform smoothness of the asymptotics given in Proposition \ref{unifsmooth} are modeled on those of \cite{J12}. Unfortunately, our integral representation for $\gamma(s,\tilde{\pi}, \Lambda^2,\psi^{-1})$ involves not the full Bessel integrals as analyzed in \cite{J12} but rather partial versions forced on us by \cite{Sh02}.  We are not able to get a germ expansion for our $B^G_\vphi(g,f)$, and as far as we can tell none exists,  but we are able to retain sufficient uniform smoothness of the expansion to establish the needed stability.
 
  Attempts have been made to show that the factors defined in \cite{Sh90, ShB} are stable with success in a number of cases \cite{CPSS08, Ts}.  In this paper we prove the analytic stability for supercuspidal representations for the exterior square $\gamma$-factor ``by hand'', but as a consequence of Theorem 1.1, we can then deduce the stability of  $\veps(s,\pi, R,\psi)$ for $R=\Lambda^2$ or $\Sym^2$  for arbitrary $\pi$ from the equality
  \[
  \veps(s,R\rho,\psi)=\veps(s,\pi(\rho),R,\psi)
  \]
  from Theorem \ref{eps} and the general arithmetic stability of Deligne \cite{De}.   We can similarly deduce the general analytic stability of $\gamma(s,\pi, R, \psi)$ for $R=\Lambda^2$ or $\Sym^2$  and arbitrary $\pi$.

 The fact that local Artin $\varepsilon$-factors appear as the local factors in the functional equations satisfied by all the cusp forms on $GL_n$  \cite{Sh90}
 is quite significant,  since most cusp forms are not of Galois type, coming from an irreducible representation of the global Weil group, where one can use Artin's global functional equation.  
  
  We would like to note that our proof of Theorem \ref{eps} is a robust argument. It can be applied to any local factor $\veps(s,\pi, R,\psi)$ as long as they satisfy the local and global conditions of those defined in \cite{Sh90, ShB, Sh-Rog}. The problem is reduced to proving the stability of $\gamma(s,\pi,R,\psi)$ in the supercuspidal case which will then need to be treated separately,  as we do here.
  As for higher exterior powers for $GL_n$, it may be possible to tackle the equality for the cases $R=\Lambda^3$  for $n=6,7,8$, which are among the cases appearing in our method \cite{LaEP, ShB}. In view of our general approach here,  one will mainly need to prove the stability for supercuspidals. This would require analysis on the simply connected exceptional groups $E_n$ for $n=6,7,8$, along the lines explored in \cite{DH, SK09}.

\noindent{\bf Acknowledgements}. This paper has been long in coming and we have several debts of gratitude to acknowledge. We thank Guy Henniart for providing for us the globalization arguments we needed in Lemma \ref{hen} and for offering both support and critical comments on several preliminary versions of this paper. We owe a special debt to Herv\'e Jacquet. In an earlier version of this paper we had believed that the supercuspidal stability would follow directly from the results in \cite{JY96} and we essentially commissioned the paper \cite{J12} as an appendix to this paper. In the end we had to retool the arguments of \cite{J12} to work in the context of our partial Bessel integrals, but we follow the method of \cite{J12} quite closely. We  thank Joachim Schwermer and the Erwin Schr\"odinger Institute of the University of Vienna for providing us   place to come together and a supportive atmosphere in which to work on these results over the years.  The ESI is a wonderful asset for mathematics and we hope it has a long and prosperous future within the University. Finally, we dedicate this paper to the memory of Joe Shalika, to whom we  owe a great debt for his guidance, both formal and informal; we hope this effort pays back a small bit of that debt.

\section{Reduction to exterior square $\gamma$-factors}

While the $\veps$-factors are the most arithmetically interesting factors, on the analytic side, the factor that one can analyze directly is the $\gamma$-factor, since these arise in various local functional equations. Key for us is the relation of the 
$\gamma$-factor with the theory of local coefficients as in \cite{Sh90}, which requires that $\pi$ be generic. However, using the Langlands classification,  the definition of the local factors, the $L$--, $\gamma$--, and $\veps$--factors, can be extended to all irreducible admissible representations of $GL_n(F)$  (see page 322 in \cite{Sh90}). For any representation $R$ of $GL_n(\mathbb C)$, the $L$-group of $GL_n$, the relation between the $\veps$ factors of Theorem \ref{eps}  and the local $\gamma$-factor is 
\[
\veps(s,\pi,R,\psi)=\frac{\gamma(s,\pi,R,\psi)L(s,\pi,R)}{L(1-s,\tilde\pi,R)}.
\]
On the arithmetic side, for $\rho$ an $n$-dimensional representation  of $W'_F$ as above, we can simply define a $\gamma$-factor by a
\[
\gamma(s,R\rho,\psi)=\frac{\veps(s,R\rho,\psi)L(1-s,R\rho^\vee)}{L(s,R\rho)},
\]
where we have adopted the convention that juxtaposition of maps indicates composition $R\rho=R\circ\rho$, so that we have the analogous relation on the arithmetic side
\[
\veps(s,R\rho,\psi)=\frac{\gamma(s,R\rho,\psi)L(s,R\rho)}{L(1-s,R\rho^\vee)}.
\]

In the context of Theorem \ref{eps}, so for $R=\Lambda^2$ or $R=\Sym^2$,  Theorem \ref{eps} would imply the equality of $\gamma$-factors as well, i.e.,
\[
\gamma(s, R\rho,\psi)=\gamma(s, \pi(\rho), R,\psi).
\]
However, as Henniart shows in Sections 3.4 and 3.5 of \cite{He10}, if one has the equality of $\gamma$-factors, even only up to a root of unity, then one can deduce the equality of $L$-factors
\[
L(s, R\rho)=L(s,\pi(\rho),R)
\]
and in fact he utilized  this technique in establishing  the equality of the exterior and symmetric square $L$-factors in \cite{He10}.  So if one has the exact equality of $\gamma$-factors then one can deduce first the equality of the $L$-functions and then the equality of the $\veps$-factors. Hence Theorem \ref{eps} is equivalent to the following theorem.

\begin{thm} \label{gamma}Let $\Lambda^2$ and $\Sym^2$ denote the exterior and symmetric
square representation of $GL_n(\bC)$, respectively.
Let $\gamma(s,\pi,\Lambda^2,\psi)$ and $\gamma(s,\pi,\Sym^2,\psi)$ be the
$\gamma$--factors attached to a $\pi$, $\psi$  and $\Lambda^2$ and $\Sym^2$, as in \cite{Sh90,ShB}.
Denote by $\gamma(s,\Lambda^2\rho,\psi)$ and $\gamma(s,\Sym^2\rho,\psi)$ the
corresponding Artin factors.
Then
\be
\item $\gamma(s,\Lambda^2\rho,\psi)=\gamma(s,\pi(\rho),\Lambda^2,\psi)$\newline
\noindent and
\item$\gamma(s,\Sym^2\rho,\psi)=\gamma(s,\pi(\rho),\Sym^2,\psi)$.  
\ee
\end{thm}

In view of the equalities
\[
\begin{aligned}
\gamma(s,\pi\times\pi,\psi)&=\gamma(s,\pi,\Lambda^2,\psi)\gamma(s,\pi,\Sym^2,\psi)\\
\gamma(s,\rho\otimes\rho,\psi)&=\gamma(s,\Lambda^2\rho,\psi)\gamma(s,\Sym^2\rho,\psi)
\end{aligned}
\]
since the local Langlands correspondence as established  in \cite{HT} and \cite{He00} preserves the $L$--factors and $\veps$--factors of pairs,  we see that if we establish Theorem \ref{gamma} for the exterior square $\gamma$-factor, namely statement (1) of the theorem, then the result for the symmetric square, namely statement (2),  will follow.

\noindent{\it Remark on the choice of an additive character:}  For $a\in F^\times$ and $\psi$ a non-trivial additive character of $F$, we let $\psi^a$ denote the character given by $\psi^a(x)=\psi(ax)$. Through the relation with the theory of local coefficients mentioned above, the work in \cite{Sh90} allows one to investigate how the $\gamma$-factors $\gamma(s, \pi,r,\psi^a)$ vary as a function of $a$, for those representations of the $L$-group $r$ that arise in this method. This variation was made quite explicit in the work of Henniart \cite{He10}, Section 2. On the other hand, the variation  of the arithmetic $\gamma(s, r\rho,\psi^a)$  in $a$ can be derived from Delgine \cite{De}.  As was observed by Henniart \cite{He10}, these variations are compatible with local class field theory, in the sense that the arithmetic and analytic $\gamma$-factors vary the same way under the substitution $\psi\mapsto\psi^a$. Therefore it suffices to prove Theorem \ref{gamma} for any fixed $\psi$ to conclude the statement for all $\psi$. In particular, if $\mathbb F$ is a global field and $v_0$ a place of $\mathbb F$ such that $\mathbb F_{v_0}\simeq F$, then we may always assume that the local additive character $\psi$ of $F$ is the local component at $v_0$ of a global additive character $\Psi$ of $\mathbb F\backslash \mathbb A_{\mathbb F}$. We will do this in what follows without further mention.

\section{Proof of the Theorem for the exterior square}

\subsection{Stable equality}

We will begin by proving the following stable version of the Theorem \ref{gamma} (1).

\begin{prop}[Stable Equality]  \label{stabthm}
Let $\rho$ be a $n$--dimensional continuous irreducible representation of $W_F$.
Then for every suitably highly ramified character $\chi$ of $F^*$
\[
\gamma(s,\Lambda^2(\rho\otimes\chi),\psi)=\gamma(s,\pi(\rho)\otimes\chi,\Lambda^2,\psi).
\]
\end{prop}

We will prove Proposition \ref{stabthm} by induction on $n$. For $n=1$ the statement is $1=1$, since the exterior square of a one dimensional representation is zero. For $n=2$ we know that $\Lambda^2\rho=\det(\rho)$ and $\gamma(s,\pi(\rho),\Lambda^2,\psi)=\gamma(s,\omega_{\pi(\rho)},\psi)$. Since $\det(\rho)$ and $\omega_{\pi(\rho)}$ correspond under the local Langlands correspondence, and the local Langlands correspondence is compatible with twists by characters, we see that Proposition \ref{stabthm} is true for $n=2$  even without the highly ramified assumption. 

To proceed, we will make the following induction hypothesis:

\noindent{\bf Induction Hypothesis}: {\it For every $p$-adic local field  $F$ of characteristic zero and every irreducible $m$-dimensional continuous representation $\rho$ of $W_F$ with $m<n$,  and for every suitably highly ramified character $\chi$ of $F^\times$, with the necessary degree of ramification depending on $\rho$,   we have 
\[
\gamma(s,\Lambda^2(\rho\otimes\chi),\psi)=\gamma(s,\pi(\rho)\otimes\chi,\Lambda^2,\psi).
\]}

Under this induction hypothesis we will prove Proposition \ref{stabthm} for $n$--dimensional irreducible continuous representations $\rho$ of $W_F$ for any $p$-adic local field $F$.

Note that for $\rho$ irreducible as in the proposition, $\pi(\rho)$ will be supercuspidal and hence generic, so we may use the full strength of the methods of \cite{Sh90,ShB} including the functional equation. We will prove this proposition  in several steps. In the next section we will establish such an equality for a fixed representation $\rho_0$ and for any character $\chi$.  This will be a global-to-local argument and will make crucial use of the Induction Hypothesis. We will then use an argument based on the stability of the $\gamma$-factors on the two sides to deform the equality at the  base point to obtain the stable equality for all $n$-dimensional $\rho$.

\subsection{Equality at a base point}

To produce the equality between the analytic and arithmetic exterior square $\gamma$-factors for a single representation, we will employ a global-to-local argument. It is based on the following lemma whose proof was communicated to us by Henniart.

\begin{lemma}\label{hen}
Let $F$ be a $p$-adic field and  $\omega_0$ a character of $F^\times$. There exist a number field $\mathbb F$
and an irreducible continuous $n$--dimensional complex representation $\Sigma$ of $W_{\mathbb F}$ such
that if $\Sigma_v=\Sigma|W_{\bF_v}$, then:\
\be
\item There exists a place $v_0$ of $\mathbb F$ such that $\mathbb F_{v_0}=F$, 
$\det\Sigma_{v_0}=\omega_0$, and $\Sigma_{v_0}$ 
is irreducible.
\item For every $v<\infty$ with $\ v\neq v_0$ , the local component  $ \Sigma_v$ is reducible.
\item  $\Pi=\pi(\Sigma)\colon=\otimes_v \pi(\Sigma_v)$ is a
  cuspidal automorphic representation of $GL_n (\bA_\bF)$. 
\ee
\end{lemma}

\begin{pf} Let $E$ be the unramified
extension of $F$ of degree $n$, necessarily cyclic, contained in $\overline{F}$. There  
exists a character $\eta$ of $E^\times$ which restricts to
$\omega_0$ on the units $\mathfrak o_F^\times$ of $F$, and which is moreover regular with
respect to the action of $Gal(E/F)$, in the sense that all its Galois conjugates are distinct. 
Choosing the value of $\eta$ correctly on a uniformizer of $F$, we can assume that the determinant
of the degree $n$ representation $\rho_0$ of $W_F$ induced from $\eta$ is $\omega_0$. By
the regularity condition $\rho_0$  is an irreducible representation of $W_F$.
So the corresponding representation $\pi_0=\pi(\rho_0)$ of $GL_n( F)$ is irreducible and supercuspidal.

Let us now globalize the situation. We can  choose an extension $\bE/\bF$ of number fields,
cyclic of degree $n$, giving $E/F$ at some finite place $v_0$ of $\bF$.
We now look for a unitary character $\alpha$ of the
idele class group of $\bE$ which restricts to 
$\eta$ on the units of $E=\bE_{w_0}$, where $w_0$ is the unique place of $\bE$ above $v_0$, and is trivial on the units of $\bE_w$
when $w$ is a finite place not above $v_0$. 
Such a character $\alpha$ certainly exists.
Indeed we start with a finite order character of the product $U_{\bE,f}=\prod_w \mathfrak o_{\bE_w}^\times$ of
the unit groups of $\bE_w$ over finite places 
$w$ of $\bE$; as $\bE^\times$ intersects $U_{\bE,f}$ trivially we can extend that
character to a group homomomorphism of  $\bE^\times U_{\bE,f}$ 
into $\mathbb C^1$,  complex numbers of modulus $1$, by extending it trivially  on $\bE^\times$. In this way
we get a finite order character of $\bE^\times U_{\bE,f}$, 
where $\bE^\times U_{\bE,f}$ has the idele topology. As $\bE^\times U_{\bE,f}$ is closed
in the idele group $\mathbb A_\bE^\times$ of $\bE$, we can extend further 
to a character of $\mathbb A_\bE^\times$, which, being trivial on
$\bE^\times$, descends to the desired character $\alpha$ of the 
idele class group $\bE^\times\backslash \mathbb A^\times_\bE$.

From $\alpha$ we obtain, by induction, a degree $n$ representation $\Sigma'$
of the Weil group $W_\bF$ of $\bF$, which at the place $v_0$ gives 
$\rho_0$ up to an unramified twist. By automorphic induction \cite{AC} we also obtain
a cuspidal automorphic representation $\Pi'$ of
$GL_n(\mathbb A_\bF)$. At each place $v$ of $\bF$, $\Pi'_v$ corresponds
to $\Sigma'_v$ under the Langlands correspondence,  
because the local Langlands correspondence is compatible with (cyclic)
automorphic induction (or base change) \cite{HT, He01}.  So $\Pi'=\pi(\Sigma')=\otimes_v\pi(\Sigma'_v)$. 
At a finite place $w$ of $\bE$ other than $w_0$, the local character of
$\bE^\times_w$ obtained by restriction of $\alpha$ is trivial on the local units, so at a finite place
$v$ of $\bF$ other than $v_0$, the local representation $\Sigma'_v$ is reducible.

Finally, to get $\rho_0$ and $\pi_0$ with the proper determinant and central character from $\Sigma'_{v_0}$ and $\Pi'_{v_0}$, we only have to twist $\Sigma'$, and hence $\Pi'$, by a suitable power of the absolute value to obtain $\Sigma$ and $\Pi$.  Since the resulting automorphic representation $\Pi=\pi(\Sigma)$ is the supercuspidal $\rho_0$ at the place $v_0$, it is itself cuspidal.
\end{pf}

If we combine this lemma with our induction hypothesis, we can obtain the  equality of local factors at a base  point.

\begin{prop}[Equality at a Base Point] \label{bp} Let $F$ be a p-adic field and $\omega_0$ a character of $F^\times$. Then there exists an irreducible $n$-dimensional representation $\rho_0$ of $W_F$ with $\det\rho_0$ corresponding to $\omega_0$  by local class field theory, such that for all characters $\chi$ of $F^\times$ we have
\[
\gamma(s,\Lambda^2 (\rho_0\otimes\chi),\psi)=\gamma(s,\pi(\rho_0)\otimes\chi,\Lambda^2,\psi).
\]
\end{prop}

\begin{pf}  Let $F$ and $\omega_0$ be as in the statement of Lemma  \ref{hen} and fix a character $\chi$ of $F^\times$.
By Lemma \ref{hen} we can find a global field $\bF$ and  a representation $\Sigma$ of $W_\bF$ such that $\bF_{v_0}=F$ for some
place $v_0$ of $\bF$ and $\rho_0=\Sigma_{v_0}=\Sigma|W_{\bF_{v_0}}$ is irreducible, but
$\Sigma_v=\Sigma|W_{\bF_v}$ is reducible for all $v<\infty$, $v\neq
v_0$. 
Moreover, again by Lemma \ref{hen}, we may take $\rho_0$ so that $\det\rho_0$  corresponds to $\omega_0$ by local class field theory. 
Then $\Pi=\pi(\Sigma)=\otimes_v \pi(\Sigma_v)$ is a cuspidal
automorphic representation of $GL_n(\bA_\bF)$, and so all of its local components $\Pi_v=\pi(\Sigma_v)$ are generic.
Let $\Psi=\otimes\Psi_v$ be a non-trivial additive character of $\bF\backslash \mathbb A_\bF$ such that $\Psi_{v_0}=\psi$.

Let $S$ be a finite set of places of $\bF$ containing $v_0$ such that for $v\notin S$ we have that $v$ is non-archimedean and $\Sigma_v$  and $\Psi_v$ are unramified. Then
$\Pi_v=\pi(\Sigma_v)$ will also be unramified. Let $S_\infty$ be the archimedean places of $\bF$ and let 
$T=S\setminus(S_\infty\cup\{v_0\})$. At the places $v\in T$ we have that $\Sigma_v$ is reducible. Let $\Sigma_v^{ss}$ 
denote its semi-simplification and write  $\Sigma_v^{ss}=\Sigma_{v,1}\oplus\cdots\oplus\Sigma_{v,r_v}$. Then for each such 
place $\Pi_v=\pi(\Sigma_v)$ will be a constituteent of $\Xi_v=\Ind(\Pi_{v,1}\otimes\cdots\otimes\Pi_{v,r_v})$ with $\Pi_{v,j}=\pi(\Sigma_{v,j})$. For each place $v\in T$ choose a sufficiently highly ramified character $\chi_v$ so that our induction hypothesis holds for each pair $(\Sigma_{v,j},\pi(\Sigma_{v,j}))$ and the character $\chi_v$.  

Now take $\eta=\otimes\eta_v$ to be an idele class character of $\bF$ such that $\eta_{v_0}=\chi$ and for each $v\in T$, 
$\eta_v=\chi_v$.  Since the local Langlands correspondence is compatible with twisting by characters, we know that for each place $v$ of $\bF$ we have $\pi(\Sigma_v\otimes\eta_v)=\pi(\Sigma_v)\otimes\eta_v=\Pi_v\otimes\eta_v$. Hence globally $\Pi\otimes\eta=\pi(\Sigma\otimes\eta)$. 

We now employ the global functional equations for the exterior square $L$-functions, as given in \cite{De} and \cite{Sh90,ShB},
 \[
 \begin{aligned}
 L(s,\Lambda^2(\Sigma\otimes\eta))&=\veps(s,\Lambda^2(\Sigma\otimes\eta))L(1-s,\Lambda^2(\Sigma^\vee\otimes\eta^{-1}) )\\
 L(s,\Pi\otimes\eta,\Lambda^2)&=\veps(s,\Pi\otimes\eta,\Lambda^2)L(1-s,\widetilde\Pi\otimes\eta^{-1},\Lambda^2).
 \end{aligned}
 \]
For an unramified place $v\notin S$, we know that the unramified $\Sigma_v$ is a direct sum of unramified characters and the corresponding $\Pi_v$ is full induced from the corresponding unramified characters. So if we write $\Sigma_v=\nu_{1,v}\oplus\cdots \oplus\nu_{i,m}$ then $\Pi_v=\Ind(\nu_{1,v}\otimes\cdots\otimes\nu_{m,v})$ and we have 
\[
\begin{aligned}
L(s,\Lambda^2(\Sigma_v\otimes\eta_v))&=\prod_{i<j}L(s,\nu_{i,v}\nu_{j,v}\eta_v^2)=L(s,\Pi_v\otimes\eta_v,\Lambda^2)\\
L(1-s,\Lambda^2(\Sigma_v^\vee\otimes\eta_v^{-1}))&=\prod_{i<j}L(1-s,\nu_{i,v}^{-1}\nu_{j,v}^{-1}\eta_v^{-2})=L(1-s,\widetilde{\Pi}_v\otimes\eta^{-1}_v,\Lambda^2)\\
\veps(s,\Lambda^2(\Sigma_v\otimes\eta_v),\Psi_v)&=\prod_{i<j}\veps(s,\nu_{i,v}\nu_{j,v}\eta_v^2,\Psi_v)=\veps(s,\Pi_v\otimes\eta_v,\Lambda^2,\Psi_v)\equiv 1
\end{aligned}
\]
so that 
\[
\begin{aligned}
L^S(s,\Lambda^2(\Sigma\otimes\eta))&=L^S(s,\Pi\otimes\eta,\Lambda^2)\\
L^S(1-s,\Lambda^2(\Sigma^\vee\otimes\eta^{-1}))&=L^S(1-s,\widetilde{\Pi}\otimes\eta^{-1})\\
\veps^S(s,\Lambda^2(\Sigma\otimes\eta),\Psi)&=\veps^S(s,\Pi\otimes\eta,\Lambda^2,\Psi).
\end{aligned}
\]
Thus, from the global functional equations we  have
\[
\prod_{v\in S}\gamma(s,\Lambda^2(\Sigma_v\otimes\eta_v),\Psi_v)=\prod_{v\in S} \gamma(s,\Pi_v\otimes\eta_v,\Lambda^2,\Psi_v).
\]

For $v\in S_\infty$, the set of archimedean places of $\mathbb F$, we know that 
\[
\gamma(s,\Lambda^2(\Sigma_v\otimes\eta_v),\Psi_v)=\gamma(s,\Pi_v\otimes\eta_v,\Lambda^2,\Psi_v)
\]
by the results of \cite{Sh85}, since we know that the arithmetic factors and the analytic factors defined by the Langlands-Shahidi method always agree at archimedean places. 

Consider now a place $v\in T$. Then $\Sigma_v$ is reducible as above and we have
\[
\begin{aligned}
\gamma(s,\Lambda^2(\Sigma_v\otimes\eta_v),\Psi_v)&=\gamma(s,\Lambda^2(\Sigma_v^{ss}\otimes\eta_v),\Psi_v)\\
&=\gamma(s,\Lambda^2((\Sigma_{v,1}\oplus\cdots\oplus\Sigma_{v,r_v})\otimes\eta_v),\Psi_v).
\end{aligned}
\]
Similarly, for $\Pi_v=\pi(\Sigma_v)$  we have
\[
\begin{aligned}
\gamma(s,\Pi_v\otimes\eta_v,\Lambda^2,\Psi_v)&=\gamma(s,\Xi_v\otimes\eta_v,\Lambda^2,\Psi_v)\\&
=\gamma(s,\Ind((\Pi_{v,1}\otimes\cdots\otimes\Pi_{v,r_v})\otimes\eta_v),\Lambda^2,\Psi_v).
\end{aligned}
\]
Now by induction on the number of factors one shows that 
\[
\gamma(s,\Lambda^2((\Sigma_{v,1}\oplus\cdots\oplus\Sigma_{v,r_v})\otimes\eta_v),\Psi_v)=
\gamma(s,\Ind((\Pi_{v,1}\otimes\cdots\otimes\Pi_{v,r_v})\otimes\eta_v),\Lambda^2,\Psi_v)
\]
with the base case being $r_v=2$, in which case we have
\[
\begin{aligned}
\gamma(s,&\Lambda^2((\Sigma_{v,1}\oplus\Sigma_{v,2})\otimes\eta_v),\Psi_v)\\&=\gamma(s,\Lambda^2(\Sigma_{1,v}\otimes\eta_v),\Psi_v)\gamma(s,\Lambda^2(\sigma_{v,2}\otimes\eta_v),\Psi_v)\gamma(s,(\sigma_{v,1}\otimes\eta_v)\otimes(\Sigma_{v,2}\otimes\eta_v),\Psi_v)\\
&=\gamma(s,\Pi_{1,v}\otimes\eta_v,\Lambda^2,\Psi_v)\gamma(s,\Pi_{v,2}\otimes\eta_v,\Lambda^2,\Psi_v)\gamma(s,(\Pi_{v,1}\otimes\eta_v)\otimes(\Pi_{v,2}\otimes\eta_v),\Psi_v)\\
&=\gamma(s,\Ind(\Pi_{1v,1}\otimes\Pi_{v,2})\otimes\eta_v,\Psi_v)
\end{aligned}
\]
where the first equality follows from the additivity of the arithmetic $\gamma$-factor \cite{De}, the second equality is a consequence of the induction hypothesis and the fact that the local Langands correspondence preserves local factors of pairs \cite{HT, He00}, and the final equality is the multiplicativity of the analytic $\gamma$-factor \cite{ShPS,ShB}.

Thus we are left with 
\[
\gamma(s,\Lambda^2(\Sigma_{v_0}\otimes\eta_{v_0}),\Psi_{v_0})=\gamma(s,\Pi_{v_0}\otimes\eta_{v_0},\Lambda^2,\Psi_{v_0})
\]
which by our construction  is precisely
\[
\gamma(s,\Lambda^2(\rho_0\otimes\chi),\psi)=\gamma(s,\pi_0\otimes\chi,\Lambda^2,\psi)
\]
as desired.
\end{pf}

\subsection{Deformations: stability of $\gamma$-factors} 
To pass from the equality of $\gamma$-factors at a base point for all characters to equality of $\gamma$ for all representations, but only for suitably highly ramified characters, we must be able to stably deform both sides of our equality.

In the arithmetic context, it was one of the basic ingredients of Deligne's proof of the existence of the local $\veps$--factors 
that  if $V$ is any finite dimensional complex representation of $W_F$ then for every sufficiently highly ramified character $\chi$ of $F^\times$, where the degree of ramification depends on $\rho$, there is a $y=y(\chi,\psi)\in F$, depending on $\chi$ and $\psi$, such that
\[
\veps(s,V\otimes\chi,\psi)=\det(V)^{-1}(y)\veps(s,\chi,\psi)^{\dim(V)},
\]
that is, for sufficiently ramified characters $\chi$, the arithmetic $\veps$--factor only depends on the determinant $\det(V)$ and the dimension $\dim(V)$ of $V$. (See Lemma 4.16 of \cite{De} or the introduction to \cite{DH}.)
If we apply this to $V=\Lambda^2\rho$ for $\rho$ an irreducible representation of $W_F$, and use the fact that for sufficiently ramified $\chi$ we always have $L(s,V\otimes\chi)=1$ then we arrive at the following proposition.

\begin{prop}[Arithmetic Stability] \label{arithstab}Let $\rho_1$ and $\rho_2$ be
  two continuous $n$-dimensional representations of $W_F$ 
with $\det(\rho_1)=\det(\rho_2)$.
Then for all the suitably highly ramified characters $\chi$ of $F^\times$ we have 
\[
\gamma(s,\Lambda^2(\rho_1\otimes\chi),\psi)=\gamma(s,\Lambda^2(\rho_2\otimes\chi),\psi).
\]
The required degree of ramification depends only on $\rho_1$ and $\rho_2$.
\end{prop}
 
 The analogous result on the analytic side is the following.

\begin{prop}[Analytic Stability for Supercuspidals] \label{scstab} Let $\pi_1$ and $\pi_2$ be
  two irreducible supercuspidal representations of $GL_n(F)$ 
with the same central characters.
Then for all the suitably highly ramified characters $\chi$ of $F^\times$,
identified as characters of $GL_n(F)$ through the determinant, we have 
$$
\gamma(s,\pi_1\otimes\chi,\Lambda^2,\psi)=\gamma(s,\pi_2\otimes\chi,\Lambda^2,\psi).
$$
The required degree of ramification depends only on $\pi_1$ and $\pi_2$
\end{prop}

The proof of this is completely local. It uses the relation between the local $\gamma$--factor and the local coefficient as in \cite{Sh90}, combined with the integral representation for the local coefficient in \cite{Sh02}, and the Jacquet-Ye germ expansion for $GL_n$ \cite{JY96} as used in \cite{Ts}. It will be given below in the lengthy Section 4.

\subsection{Proof of  stable equality and some corollaries} We are now in a position to complete the proof of Proposition \ref{stabthm}, the stable version of the equality of the arithmetic and analytic exterior square $\gamma$--factors. Let $\rho$ be an irreducible $n$-dimensional representation of $W_F$ and let $\pi=\pi(\rho)$ be the corresponding supercuspidal representation of $GL_n(F)$. Taking $\omega_0=\omega_{\pi}$ in Proposition \ref{bp} we have an irreducible $n$-dimensional representation of $W_F$ and corresponding supercuspidal representation $\pi_0=\pi(\rho_0)$ of $GL_n(F)$ such that
\begin{enumerate}
\item $\omega_\pi=\omega_0=\omega_{\pi_{_0}}$
\item $\det(\rho)=\det(\rho_0)$
\item $\gamma(s,\Lambda^2(\rho_0\otimes \chi),\psi)=\gamma(s,\pi_0\otimes\chi, \Lambda^2,\psi)$ for all characters $\chi$ of $F^\times$.
\end{enumerate}

Now take $\chi$ sufficiently ramified so that both Proposition \ref{arithstab} holds for the pair $(\rho,\rho_0)$ and Proposition \ref{scstab} holds for the pair $(\pi,\pi_0$).  Then for such suitably ramified $\chi$ we have
\[
\gamma(s,\Lambda^2(\rho\otimes \chi),\psi)=\gamma(s,\Lambda^2(\rho_0\otimes\chi),\psi)=\gamma(s,\pi_0\otimes \chi,\Lambda^2,\psi)=\gamma(s,\pi\otimes\chi,\Lambda^2\psi).
\]
This completes the induction step and the proof of Proposition \ref{stabthm}. \hfill $\Box$

Note that the degree of ramification necessary depends not just on the representations $\rho$ and $\pi$, but also on the choice of the pair $(\rho_0,\pi_0)$ from Proposition \ref{bp}. So one must fix such a pair for every character $\omega_0$, but can easily reduce to a character $\omega_0$ in each inertial class since deformations by unramified characters can be absorbed into the complex parameter of the $\gamma$--factors.

Before we continue with the proof of Theorem \ref{gamma} (1), let us give two corollaries of our stable version of the theorem.

\begin{cor} [Extension to the Weil-Deligne Group] \label{wdstab} Let $\rho$ be a continuous $n$--dimensional $\Phi$-semisimple complex representation of the Weil-Deligne group $W'_F$.
Then for all suitably highly ramified characters $\chi$ of $F^\times$
\[
\gamma(s,\Lambda^2(\rho\otimes\chi),\psi)=\gamma(s,\pi(\rho)\otimes\chi,\Lambda^2,\psi)
\]
\end{cor}

\begin{pf} This follows from Proposition \ref{stabthm} and the following facts:
\begin{enumerate}
\item the compatibility of the construction of $\Phi$-semisimple representations of $W'_F$ from irreducible representations of $W_F$ and the Zelevinsky construction of irreducible representations of $GL_n(F)$ from supercuspidals \cite{Ze},
\item for a representation of the Weil-Deligne group, the  $\veps$-factor $\veps(s,\rho,\psi)=\veps(s,\rho^{ss},\psi)$ depends only on the simplification of $\rho$ as a representation of $W_F$ \cite{De},
\item the local Langlands correspondence preserves $L$-factors of pairs and of exterior square $L$-factors, and for highly ramified twists these stabilize to $1$ \cite{He00, He10, De, Sh00},
\item the resulting additivity of the arithmetic  $\gamma$-factor and the multiplicativity of the analytic $\gamma$-factor \cite{De,ShPS}. 
\end{enumerate}
\end{pf}

\begin{cor} [General Analytic Stability] Let  $\pi_1$ and $\pi_2$ be two 
  irreducible admissible representations of $GL_n(F)$ 
with the same central character. Then for any suitably ramified character of $F^\times$ we have
\[
\gamma(s,\pi_1\otimes\chi,\Lambda^2,\psi)=\gamma(s,\pi_2\otimes\chi,\Lambda^2,\psi).
\]
\end{cor}

\begin{pf} Let $\rho_1$ and $\rho_2$ be continuous $n$-dimensional $\Phi$-semisimple representations of the Weil-Deligne group $W'_F$ so that $\pi_i=\pi(\rho_i)$. Then by Colollary \ref{wdstab} for every sufficiently ramified character $\chi$ 
we have
\[
\gamma(s,\Lambda^2(\rho_i\otimes\chi),\psi)=\gamma(s,\pi_i\otimes\chi,\Lambda^2,\psi).
\]
So, just as in the final step of the proof of Proposition \ref{stabthm} above,  our result would follow if we knew the analogue of arithmetic stability, Proposition \ref{arithstab}, for representations of $W'_F$. But, as in Corollary \ref{wdstab} this follows from Proposition \ref{arithstab} and the fact that for a representation of the Weil-Deligne group, the  $\veps$-factor $\veps(s,\rho,\psi)=\veps(s,\rho^{ss},\psi)$ depends only on the simplification of $\rho$ as a representation of $W_F$ and that the semi simplification does not change the determinant of the representation, thus giving the arithmetic stability of the $\gamma$--factors for Weil-Deligne representations. \end{pf}

\subsection{Proof of Theorem \ref{gamma} (1)}

We begin by proving Theorem \ref{gamma} (1) for a special class of representations of $W_F$, namely those induced from finite order characters of $W_L$ for $L/F$ a finite Galois extension.
Here we need to use ideas of Harris \cite{H98} and Henniart  \cite{He00}. It is established by another global-to-local argument combined with stability.

\begin{lemma}[Equality for Monomial Representations]  \label{ind}Let $E/F$ be a finite Galois extension of degree $n$ contained in $\overline{F}$ and set $G=\Gal(E/F)$. Let $F\subset L\subset E$ be an intermediate
extension. Let $\chi$ be a finite order character of $H=\Gal(E/L)$ Let $\rho=\Ind_{H}^{G}(\chi)$. Then
\[
\gamma(s,\Lambda^2\rho,\psi)=\gamma(s,\pi(\rho),\Lambda^2,\psi).
\]
\end{lemma}

\begin{pf} As in  Henniart's proof of the Local Langlands Correspondence in \cite{He00}, utilizing a theorem of Harris \cite{H98} , we know there exists a global extension of number fields $\mathbb F\subset \mathbb L\subset \mathbb  E$ and a character $\mathfrak X:\mathbb L^\times\backslash \mathbb A_{\mathbb L}^\times\rightarrow \mathbb C^\times$ satisfying
\begin{enumerate}
\item[(i)] there is a place $v_0$ of $\mathbb F$, having a unique place $w_0$ of $\mathbb E$ above it, such that $\mathbb F_{v_0}=F$, $\mathbb E_{w_0}=E$ and $\Gal(\mathbb E/\mathbb F)=\Gal(E/F)$;
\item[(ii)] if $v_0'$ is the unique place of $\mathbb L$ over $v_0$ then $\mathbb L_{v_0'}=L$ and $\mathfrak X_{v_0'}=\chi$,
 \end{enumerate}
 such that, if we let $\Sigma=\Ind_{\Gal(\mathbb E/\mathbb L)}^{\Gal(\mathbb E/\mathbb F)}(\mathfrak X)$, then there is a cuspidal automorphic representation $\Pi$ of $GL_m(\mathbb A_\mathbb F)$, where $m=(\mathbb L:\mathbb F)$, satisfying condition
 \begin{enumerate}
 \item[(iii)] $\Sigma$ and $\Pi$ are associate in the sense that for all finite places where $\Sigma$ and $\Pi$ are unramified, $\Pi_v=\pi(\Sigma_v)$.
 \end{enumerate}
 Moreover, as Henniart later shows in \cite{He00} and/or \cite{He01}, we also know
 \begin{enumerate}
 \item[(iv)] for all places $\Pi_v=\pi(\Sigma_v)$,
 \item[(v)] at the place $v_0$ we have $\Sigma_{v_0}=\Ind_{W_L}^{W_F}(\chi)=\rho$ and $\pi(\rho)=\Pi_{v_0}=\pi_{1,v_0}\boxplus\cdots\boxplus \pi_{r,v_0}$ is a full induced representation from supercuspidal representations. 
 \end{enumerate}
 
Let $\Psi=\otimes\Psi_v$ be a non-trivial additive character of $\mathbb F\backslash \mathbb A_{\mathbb F}$ such that $\Psi_{v_0}=\psi$ from the statement. Let $S$ be the set of places of $\mathbb F$ such that for $v\notin S$ we have that $v$ is non-archimedean and $\Sigma_v$, $\Pi_v$ and $\Psi_v$ are all unramified and let $S'=S\setminus \{v_0\}$.  Take $\eta=\otimes\eta_v$ to be a idele class character for $\mathbb F$ such that 
 $\eta_{v_0}$ is trivial, and
  for all $v\in S'$,  $\eta_v$ is sufficiently highly ramified that Proposition \ref{stabthm} holds for the pairs $(\Sigma_v,\Pi_v)$.
Since the local Langlands correspondence is compatible with twists by characters, we know that $\Sigma\otimes\eta$ and $\Pi\otimes\eta$ are still associated and in fact $\pi(\Sigma_v\otimes\eta_v)=\Pi_v\otimes\eta_v$ for all places $v$ of $\mathbb F$.  
 
 We employ the global functional equations for the exterior square $L$-functions from \cite{De} and \cite{Sh90,ShB}, namely 
 \[
 \begin{aligned}
 L(s,\Lambda^2(\Sigma\otimes\eta))&=\veps(s,\Lambda^2(\Sigma\otimes\eta))L(1-s,\Lambda^2(\Sigma^\vee\otimes\eta^{-1})) \\
 L(s,\Pi\otimes\eta,\Lambda^2)&=\veps(s,\Pi\otimes\eta,\Lambda^2)L(1-s,\widetilde\Pi\otimes\eta^{-1},\Lambda^2),
 \end{aligned}
 \]
as in the proof of Proposition \ref{bp}. The unramified local Langlands correspondence again implies
\[
\begin{aligned}
L^S(s,\Lambda^2(\Sigma\otimes\eta))&=L^S(s,\Pi\otimes\eta,\Lambda^2)\\
L^S(1-s,\Lambda^2(\Sigma^\vee\otimes\eta^{-1}))&=L^S(1-s,\widetilde{\Pi}\otimes\eta^{-1}, \Lambda^2)\\
\veps^S(s,\Lambda^2(\Sigma\otimes\eta),\Psi)&=\veps^S(s,\Pi\otimes\eta,\Lambda^2,\Psi)
\end{aligned}
\]
and so from the global functional equations we we have
\[
\prod_{v\in S}\gamma(s,\Lambda^2(\Sigma_v\otimes\eta_v),\Psi_v)=\prod_{v\in S} \gamma(s,\Pi_v\otimes\eta_v,\Lambda^2,\Psi_v).
\]

For $v\in S_\infty$, the set of archimedean places of $\mathbb F$, we know that 
\[
\gamma(s,\Lambda^2(\Sigma_v\otimes\eta_v),\Psi_v)=\gamma(s,\Pi_v\otimes\eta_v,\Lambda^2,\Psi_v)
\]
by the results of \cite{Sh85}, since we know that the arithmetic factors and the analytic factors defined by the Langlands-Shahidi method always agree at archimedean places. For the places $v\in S'$ we also know
\[
\gamma(s,\Lambda^2(\Sigma_v\otimes\eta_v),\Psi_v)=\gamma(s,\Pi_v\otimes\eta_v,\Lambda^2,\Psi_v)
\]
by our choice of $\eta_v$ and Proposition \ref{stabthm}. Thus we are left with 
\[
\gamma(s,\Lambda^2(\Sigma_{v_0}\otimes\eta_{v_0}),\Psi_{v_0})=\gamma(s,\Pi_{v_0}\otimes\eta_{v_0},\Lambda^2,\Psi_{v_0})
\]
which by our construction, and since $\eta_{v_0}$ is trivial,  is precisely
\[
\gamma(s,\Lambda^2\rho,\psi)=\gamma(s,\pi(\rho),\Lambda^2,\psi)
\]
as desired.
\end{pf}

Still following the lead of Henniart, we next extend this to irreducible continuous representations $\rho$ of $W_F$ having determinant of finite order.

\begin{lemma}[Equality for Galois Representations]   \label{fcc}Let $\rho$ be an irreducible continuous representation of $W_F$ of degree $n$ with finite order determinant. Then 
\[
\gamma(s,\Lambda^2\rho,\psi)=\gamma(s,\pi(\rho),\Lambda^2,\psi).
\]
\end{lemma}

\begin{pf}  We will work in the Grothendieck groups of representations. Following Henniart, we let $\mathcal G_F^0(m)$ denote the set of isomorphism classes of irreducible continuous complex representations of $W_F$ of degree $m$, $\mathcal G_F^0$ the disjoint union of the $\mathcal G_F^0(m)$ and $R_{\mathcal G}(F)$ the Grothendieck group of $W_F$, the free abelian group with basis $\mathcal G_F^0$.  Similarly, we let $\mathcal A_F^0(m)$  be the set of isomorphism classes of irreducible admissible supercuspidal representations of $GL_m(F)$, $\mathcal A_F^0$ the disjoint union of the $\mathcal A_F^0(m)$, and $R_{\mathcal A}(F)$ the Grothendieck group of $GL(F)$, the free abelain group with basis $\mathcal A^0_F$. 

Now let $\rho$ be an irreducible continuous representation of $W_F$ of degree $n$ with finite order determinant. As in Henniart's proof of the local Langlands correspondence in \cite{He00}, there is a finite Galois extension $E/F$ and a $n$ dimensional representation of $Gal(E/F)$ which gives $\rho$ by inflation to $W_F$; call it again $\rho$. Then by Brauer's Theorem, in the group  $ R_{\mathcal G}(F)$ we can write 
\[
\rho=\sum_i n_i\Ind_{H_i}^G(\chi_i)
\]
where $H_i\subset G$ corresponds to a finite extension $F\subset L_i\subset E$ with $H_i=\Gal(E/L_i)$ and $\chi_i$ a character of $H_i$.  If we let $\rho_i=\Ind_{H_i}^G(\chi_i)$,  then we know by Henniart, quoted in (v) in the proof of Lemma
\ref{ind}, that $\pi(\rho_i)=\pi_i$ is a full induced from supercuspidal representations, $\pi_i=\pi_{i,1}\boxplus\cdots\boxplus\pi_{i,r_i}$. Then by Henniart's proof of the local Langlands correspondence we know that
\[
\pi(\rho)=\sum_i n_i(\pi_i)^{ss}
\]
where $(\pi_i)^{ss}$ indicates the supercuspidal support, so in our case $(\pi_i)^{ss}=\sum_j \pi_{i,j}$ in $R_{\mathcal A}(F)$. 

We now want to compare the exterior square $\gamma$-factors for $\rho$ and $\pi(\rho)$.  For this it is easiest to work with 
representations of $W_F$ or $GL_m(F)$ rather than virtual representations. If $\rho_1$ and $\rho_2$ are representations of 
$W_F$, then their sum in $R_{\mathcal G}(F)$ corresponds to the direct sum $\rho_1\oplus\rho_2$ of representations. 
Under the local Langlands correspondence, if $\pi(\rho_1)=\pi_1$ and $\pi(\rho_2)=\pi_2$, then $\rho_1\oplus\rho_2$ 
corresponds to the induced representation $\pi_1\boxplus\pi_2$ and so in $R_{\mathcal A}(F)$ the sum corresponds to 
parabolic induction and this is consistent with the local Langlands correspondence. By additivity of the arithmetic $\gamma
$-factors and the multipliciaivity of the analytic $\gamma$-factors of Shahidi, we know
\[
\begin{aligned}
\gamma(s,\Lambda^2(\rho_1+\rho_2),\psi=\gamma(s\Lambda^2(\rho_1\oplus \rho_2),\psi)&=\gamma(s,
\Lambda^2\rho_1,\psi)\gamma(s,\Lambda^2\rho_2,\psi)\gamma(s,\rho_1\otimes\rho_2,s) \\
\gamma(s,\pi_1+\pi_2,\Lambda^2,\psi)=\gamma(s,\pi_1\boxplus \pi_2,\Lambda^2,\psi)&=\gamma(s,\pi_1,\Lambda^2,\psi)
\gamma(s,\pi_2,\Lambda^2\psi)\gamma(s,\pi_1\times\pi_2,s). 
 \end{aligned}
 \]
 So the exterior square $\gamma$ factors satisfy the same formalism on the arithmetic and analytic side for {\it sums} of 
 representations. In particular, on the analytic side, this means that multiplicativity gives
 \[
 \gamma(s, (\pi_i)^{ss},\Lambda^2,\psi)=\gamma(s,\pi_i,\Lambda^2,\psi).
 \]
 
Coming back to our expression
\[
\rho=\sum_i n_i\Ind_{H_i}^G(\chi_i)=\sum_{i}n_i\rho_i
\]
renumbering the term if necessary, we can assume that $n_i>0$ for $i=1,\dots,r$ and $n_i<0$ for $i=r+1,\dots r+s$. Then we can rewrite this relation as
\[
\rho+\sum_{i=r+1}^{r+s} (-n_i)\rho_i=\sum_{i=1}^r n_i\rho_i.
\]
Now both sides are sums of true representations. Let us write $m_i=-n_i$ for $i=r+1,\dots,r+s$.
Applying the formalism of the exterior square $\gamma$-function, from the expression we obtain
\[
\begin{aligned}
\gamma(s,\Lambda^2(\sum_{i=1}^r n_i\rho_i),\psi)&=
\gamma(s,\Lambda^2(\rho+\sum_{i=r+1}^{r+s}m_i\rho_i),\psi)\\&=\gamma(s,\Lambda^2\rho,\psi)\gamma(s,\Lambda^2(\sum m_i\rho_i),\psi)\gamma(s,\rho\otimes(\sum m_i\rho_i),\psi)\\
&=\gamma(s,\Lambda^2\rho,\psi)\gamma(s,\Lambda^2(\sum m_i\rho_i),\psi)\prod\gamma(s,\rho\otimes\rho_i,\psi)^{m_i}
\end{aligned}
\]
and so
\[
\gamma(s,\Lambda^2\rho,\psi)=\frac{\gamma(s,\Lambda^2(\sum_{i=1}^r n_i\rho_i),\psi)}{\gamma(s,\Lambda^2(\sum_{i=r+1}^{r+s} m_i\rho_i),\psi)\prod_{i=r+1}^{r+s}\gamma(s,\rho\otimes\rho_i,\psi)^{m_i}}.
\]

On the analytic side we have 
\[
\pi(\rho)=\pi=\sum_i n_i(\pi_i)^{ss}
\]
or
\[
\pi+\sum_{i=r+1}^{r+s} m_i(\pi_i)^{ss}=\sum_{i=1}^r n_i(\pi_i)^{ss}.
\]
Performing the same steps on the analytic side, and using the fact that the formalism of the exterior square $\gamma$-factors is the same for the arithmetic and analytic $\gamma$-factors, we arrive at
\[
\gamma(s,\pi,\Lambda^2,\psi)=\frac{\gamma(s,\Lambda^2(\sum_{i=1}^r n_i(\pi_i)^{ss}),\psi)}{\gamma(s,\Lambda^2(\sum_{i=r+1}^{r+s} m_i(\pi_i)^{ss}),\psi)\prod_{i=r+1}^{r+s}\gamma(s,\pi\times(\pi_i)^{ss},\psi)^{m_i}}.
\]

By the multiplicativity of the analytic local Rankin-Selberg $\gamma$-factors, for each $i=r+1,\dots, r+s$ we have
\[
\gamma(s,\pi\times(\pi_i)^{ss},\psi)=\gamma(s,\pi\times\pi_i,\psi)
\]
and  since the local Langlands correspondence preserves local factors of pairs we can further
conclude
\[
\gamma(s,\pi\times\pi_i,\psi)=\gamma(s,\rho\otimes\rho_i,\psi).
\]
Hence we have
\[
\prod_{i=r+1}^{r+s}\gamma(s,\pi\times(\pi_i)^{ss},\psi)^{m_i}=\prod_{i=r+1}^{r+s}\gamma(s,\rho\otimes\rho_i,\psi)^{m_i}.
\]
Then, by an induction on the number of factors, and utilizing that the formalism of the exterior square $\gamma$-factors is the same on the analytic and arithmetic side, we have
\[
\gamma(s,\sum_{i=1}^r n_i(\pi_i)^{ss},\Lambda^2,\psi)=\gamma(s,\Lambda^2(\sum_{i=1}^r n_i\rho_i), \psi)
\]
with the base case being one summand where we know
\[
\gamma(s,(\pi_i)^{ss},\Lambda^2,\psi)=\gamma(s,\pi_i,\Lambda^2,\psi)=\gamma(s,\Lambda^2\rho_i,\psi)
\]
with the first equality coming from the multiplicativity of the analytic $\gamma$-factor, as noted above, and the second equality by Lemma \ref{ind}. Similarly we have
\[
\gamma(s,\sum_{i=r+1}^{r+s} m_i(\pi_i)^{ss},\Lambda^2,\psi)=\gamma(s,\Lambda^2(\sum_{i=r+1}^{r+s} m_i\rho_i),\psi).
\]
Hence we can conclude
\[
\gamma(s,\Lambda^2\rho,\psi)=\gamma(s,\pi(\rho),\Lambda^2,\psi)
\]
and our lemma is proven.
\end{pf}

To complete the proof of Theorem \ref{gamma} (1),  we  begin with Lemma \ref{fcc}. We can extend this to arbitrary irreducible continuous $n$-dimensional representations of $W_F$ by tensoring with an unramified character. Both the local Langlands correspondence and the formalism of the exterior square $\gamma$-factors are compatible with twisting by characters.  We then can extend to arbitrary continuous $n$-dimensional representations of $W_F$ since the local Langlands correspondence extends over direct sums and then again applying additivity/multiplicativity of the exterior square $\gamma$-factors. Finally, we extend to all continuous $\Phi$-semisimple $n$-dimensional representations of the Weil-Deligne group $W'_F$ by the compatibility of the reduction of the local Langlands correspondence to the irreducible case on the Weil-Deligne side and the Bernstein-Zelelvinsky reduction to the supercuspidal case on the analytic side, again combined with additivity/multiplicativity of the exterior square $\gamma$-factors. \hfill\qed

\section{Analytic stability  for supercuspidal
representations I: reduction to partial Bessel integrals }

The purpose of this  section and the next is to provide the proof of Proposition \ref{scstab}. We will follow the 
general theory of \cite{CPSS08}  and \cite{Sh02}, specialized to $GSp_{2n}$, combined with the ideas from \cite{JY96, J12, Ts}.
To ease the notational burden, we will use bold face letters $\mathrm{\mathbf H}$ to denote algebraic groups, all defined over $F$, and then $H=\mathrm{\mathbf H}(F)$ to denote their $F$-points.

\subsection{Preliminaries}\label{prelim}  Let us recall that $F$ is a non-archimedean
local field of characteristic zero with  $\mathfrak{o}$  the ring of
integers of $F$, and $\mathfrak{p}$  its unique maximal ideal with uniformizer $\varpi$.  We set
$q=[\mathfrak{o}: \mathfrak{p}]$, and fix a valuation $|\cdot|=
|\cdot|_{F}$ normalized by $|\varpi|=q^{-1}$.

Let $\G=GL_n$ as an algebraic group over $F$.  Let $\B=\A\U$ be a  Borel subgroup of $\G$ over $F$, where $\A$ is a maximal split  torus of $\G$ and $\U$ the unipotent radical of $\B$.  We take the standard matrix realization of $G=GL_n(F)$, in which $B$ becomes the standard upper triangular Borel subgroup, $A$ the diagonal torus and $U$ the upper triangular unipotent matrices. Let
$\Phi=\Phi(\A,\G)$ be the roots of $\A$ in $\G$, $\Phi^{+}$ to be the set of 
positive roots in $\G$, and $\Delta=\Delta(\A,\G)$ the simple roots. If we define rational characters $e_i$ of $\A$ by $e_i(a)=a_i$ for $a=diag(a_1, \dots, a_n)\in A$. Then  $\Delta=\{\alpha_1,\ldots,\alpha_{n-1}\}$ where
$\alpha_i=e_i-e_{i+1}$
in the (additive) group $X(\A)_F$ of $F$-rational characters of $\A$.

Let $W=W(\A,\G)$ denote the Weyl group of $\G$. If we use the standard splitting which gives the usual matrix representation of $GL_n$, and let $x_{\alpha_i}(t)$ be the one parameter subgroup $x_{\alpha_i}(t)=I_n+tE_{i,i+1}$ with $E_{i,j}$ the standard elementary matrix with a $1$ in the $(i,j)$--position and $0$ everywhere else, and similarly $x_{-\alpha_i}(t)=I_n+tE_{i+1,i}$, then, following Chevalley/Steinberg/Bruhat--Tits, we choose as our representative for the simple reflection $s_i$ associated to $\alpha_i$ the matrix
\[
\dot{s_i}=x_{\alpha_i}(1)x_{-\alpha_i}(-1)x_{\alpha_i}(1).
\]
This is a matrix $\dot{s_i}$ with $1's$ on the diagonal other than a $2\times 2$ block of the form $\bpm 0 & 1 \\ -1 & 0\epm$ in the $(i,i+1)$ block. 
For the general Weyl group element $w$ we choose a reduced expression for $w$ in terms of the simple reflections and take $\dot{w}$ to be the corresponding product of matrix representatives $\dot{s_i}$.
This representative is independent of the choice of reduced decomposition by  Proposition  8.3.3 of Springer \cite{Spr}  or  Lemma 56 of Steinberg \cite{St}. 
This choice has the advantage that for all $w\in W$ we have $\det({\dw})=1$, i.e., we take representatives that live in the derived group, in our case $SL_n$. 
If we let $w_\ell$ denote the longest element, then as the representative of the long element in our matrix realization we have
\[
\dw_\ell=\bpm & & & 1\\ & & -1 \\ & \sddots \\ (-1)^{n-1}\epm.
\]
Moreover, if $M\subset GL_n$ is a standard Levi subgroup in block form, so $M\simeq GL_{n_1}\times\cdots\times GL_{n_t}$ then the long Weyl element of $M$ will be of the form
\[
\dot{w}_\ell^M=\bpm \dot{w}_{\ell,1} \\ & \dot{w}_{\ell,2} \\ & & \ddots\\ & & & \dot{w}_{\ell,_t}\epm
\]
with each $ \dot{w}_{\ell,i}$ the representative of the long Weyl element of $GL_{n_i}$ as above.

We will let $\H=GSp_{2n}$ be the symplectic similitude group over $F$ for a positive integer $n$.  
If we set
\[
J=J_n=\bpm & & & 1\\ & & -1 \\ & \sddots \\ (-1)^{n-1}\epm
\quad\quad\text{and}\quad\quad
J'=J'_{2n}=\bpm & J\\-{\tJ}\epm
\]
then we can take a matrix realization of $H$ as
\[
H=GSp_{2n}(F)=\{h\in GL_{2n}(F)\mid {^th} J' h=\eta(h) J' \text{ for some }\eta(h)\in F^\times\}.
\]
$\eta:H\rightarrow F^\times$ is the similitude character of $H$. Note that ${\tJ}=J^{-1}=(-1)^{n-1}J$. The center $\Z_\H$ of $\H$ is isomorphic to $\G_m$ and is thus cohomologically trivial. The cohomological triviality of the center is necessary for the use of \cite{Sh02} and is responsible for our chioce of $\H=GSp_{2n}$ rather than (the seemingly simpler) $Sp_{2n}$.

Let $\B_\H=\A_\H\U_\H$ be the standard (upper triangular) 
Borel subgroup of $\H$ over $F$, where $\A_\H$ is a maximal split torus
and $\U_\H$ is the unipotent radical of $\B_\H$. In the matrix realization, we can represent the elements of $A_H$ by
\[
a'=\bpm a \\ & a_0Ja^{-1}J^{-1}\epm=\diag(a_1,\dots,a_n;a_0a_n^{-1},\dots,a_0a_1^{-1})\quad\text{with}\quad a=\bpm a_1\\&\ddots\\& & a_n\epm\in A
\]
which has similitude character $\eta(a')=a_0$.

Denote by $\P_\H=\M_\H\N_\H$ the standard Siegel parabolic subgroup of $\H$. This
 is the self--associate maximal parabolic subgroup of $\H$ over $F$ with the
Levi factor $\M_\H\simeq\G\times\G_m=\gl_{n} \times \gl_1$ containing $\A_\H$, and let $\N_\H$ be its unipotent
radical. 
The matrices  in $M_H$ are of the form
\beq\label{tildem}
m=m(g,a_0)=\bpm g \\ & a_0J {^tg^{-1}}J^{-1}\epm
\eeq
with $g\in G$ and $a_0\in GL_1(F)$ while  those in $N_H$ are of the form
\beq\label{n}
n=n(y)=\bpm I_n & y \\ & I_n\epm
\eeq
with $y\in Mat_n(F)$ with $y=J{^ty}J$. We will use the isomorphism $g\mapsto m(g,e)$ to identify $G$ with a subgroup of $M_H$ throughout this section. 
Then $\U\simeq\U_\H\cap \M_\H$, which in the matrix realization corresponds to the standard upper triangular maximal unipotent subgroup of $G$. 

Let $\Phi_\H=\Phi(\A_\H,\H)$ and $\Delta_\H=\Delta(\A_\H,\H)$ be the roots and
simple roots of $\A_\H$ in $\H$. Notice here $\Delta_\H=\{\alpha_{1},\ldots,
\alpha_{n} \}$ in numbering from Bourbaki \cite{Bourbaki}.   
If we define rational characters $e_i$ of $\A_\H$ by $e_i(a')=a_i$, including $i=0$, then 
\[
\alpha_i=e_i-e_{i+1}\in\Delta \quad\text{for }i=1\ldots,n-1\quad \text{and }\alpha_n=2e_n-e_0
\]
in the (additive) group $X(\A_\H)_F$ of $F$-rational characters of $\A_\H$.
 Denote by $\Phi_\H^{+}$   the set of 
positive roots in $\H$.

Let $\{x_{\alpha}\,|\, \alpha\in\Delta_\H\}$ be an $F$-splitting of $\H$
as defined in \cite{CPSS05, CPSS08} extending our splitting of $\G$.  To do this, let $E_{i,j}$ represent the elementary matrix in $Mat_{2n}(F)$ having a single $1$ in the $(i,j)$ position and zeroes elsewhere. Then we can set
\[
x_{\alpha_i}(u)=\begin{cases} I_{2n}+u(E_{i,i+1}-E_{2n-i,2n-i+1})  & i=1,\dots,n-1 \\  
I_{2n}+uE_{n,n+1} & i=n 
\end{cases}.
\] 
The simple root subgroups  $U_{\alpha}=\U_\alpha(F)=\{x_{\alpha}(u_{\alpha})\,|\, u_{\alpha}\in F\}$ correspond to the one parameter subgroups of $U_H$ lying just above the diagonal. Note that  $\alpha_{n}$ is the unique simple root in $\Delta_\H$ whose
root subgroup sits in $\N_\H$. In fact, $N_H$ is spanned by the following root subgroups
\[
N_H=\langle U_\alpha\mid \alpha=e_i+e_j\text{ with }1\leq i<j\leq n\text{ or }\alpha=2e_i\text{ with }1\leq i \leq n\rangle.
\]
For any $ n\in N_H$, we can use the splitting to write $n=\prod x_{\alpha}(u_{\alpha})$ where $
\alpha$ runs over the roots occurring in ${N_H}$ as above. Since ${N_H}$ is abelian, this is independent of the order taken in the product. In particular, for ${\alpha}_n$, the coordinate $u_{{\alpha}_n}$ will be 
independent of the order we take. We will denote this by $u_{\alpha_n}( n)$. If we use the matrix 
realization of our group, $u_{{\alpha}_n}( n)$ will be the $(n,n+1)$-entry of $n$.
We let 
\[
\rho=\rho_{P_H}=\frac{n+1}{2}(e_1+\cdots + e_n)-\frac{n(n+1)}{4}e_0
\]
 denote half the sum of the positive roots occurring in ${N_H}$.

Let $W_\H=W(\A_\H,\H)$ denote the Weyl group of $\A_\H$ in $\H$.  Note that we have $W_\G=W_{\M_\H} \subset W_\H$ via $w\mapsto m(w,e)$ and we will use this identification.   
Let $w^H_\ell$ denote the long Weyl element in $W_\H$. Let $w_0=w^H_\ell \cdot  w_\ell^{-1}$.   As a representative for these elements we have 
\beq
\dw^H_\ell=\bpm & J\\ -\tJ\epm,\quad \dw_\ell=m(\dw_\ell,1)=\bpm J\\&{J}\epm, \quad\text{and}\quad \dw_0=\bpm & I_n\\(-1)^nI_n\epm. \label{welts}
\eeq
 Then $w_0(\Delta)=\Delta$, and $w_0(\alpha_n)<0$.
Since $\P_\H$ is self-associate, then
$\overline{\N}_H=w^H_\ell \N_\H (w_\ell^H)^{-1}=\N_\H^-$, where $\N_\H^-$ is the opposite
subgroup to $\N_\H$ with respect to $\M_\H$.

Given a non-trivial character
$\psi$ of $F$, we define a generic character of $U_H$ by
\[
\psi(u)=\psi(\sum_{\alpha\in\Delta_H} u_{\alpha})=\sum_{i=1}^n(u_{i, i+1}).
\]
The condition of compatibility of $\psi$ and $w_0$ within $H$ necessary for the use of \cite{Sh02} is that $\psi(w_0uw_0^{-1})=\psi(u)$ for $u\in U_H\cap M_H\simeq U$ and one checks that this holds for our $\psi$ and $w_0$.

For $i=1, \ldots, n$, we define cocharacters
\[
e_{i}^{\ast}(t)=\text{diag} (1, \ldots,1 ,t, 1, \ldots, 1; 1,
\ldots,1 ,t^{-1}, 1, \ldots, 1),
\]
which means $e_i^\ast(t)$ has $t$ on the $i$-th component, $t^{-1}$ on the
$(2n+1-i)$-th component, and $1$'s elsewhere.  Additionally, we
define
\[
e_{0}^{\ast}(t)=\text{diag} (1, \ldots,1 ; t, \ldots, t).
\]
In terms of these cocharacters, we see that the element $a'\in A_H$ above is given by
\[
a'=\prod_{i=0}^n e_i^*(a_i)
\]
so that $\prod e_i^*: (F^\times)^{n+1}\rightarrow A_H$ gives a splitting of $A_H$. 
Another convenient splitting of $A_H$ is given by the co-characters that are duals to the simple roots. To this end, for $i=1,\dots,n-1$, we define
\[
\alpha_i^*=\prod_{j=1}^i e_j^*
\]
so that
\[
\alpha_i^*(t)=\bpm tI_i \\ & I_{2n-i}\\& & t^{-1}I_i\epm
\]
with $t$ in the first $i$ entries, and for $i=n$ we set
\[
\alpha_n^*=\prod_{i=0}^n e_i^*
\]
so that 
\[
\alpha_n^*(t)=\bpm tI_n \\ & I_n\epm
\]
Then we see that  $\alpha_i(\alpha_j^*(t))=\begin{cases} t & i=j \\ 1 & i\neq j\end{cases}$. 
We can also use these to split the torus via
$e_0^*\prod \alpha_i^*:(F^\times)^{n+1}\rightarrow A_H$. 

In terms of the second splitting, we have
\[
Z_{H}
=\{e_0^*(t)\alpha_n^*(t)\mid t\in F^\times\}=\left\{\bpm tI_n\\& tI_n\epm\mid t\in F^\times\right\}
\]
and
\[
Z_{{M_H}}
=\{e_{0}^{\ast}(t_2) \alpha_n^*(t_1)
\mid t_2,t_1\in F^{\times}\}
=\left\{\bpm t_1 I_n\\& t_2I_n\epm\mid\ t_1,t_2\in F^\times\right\}.
\]
We have a short exact sequence
\[
1\rightarrow Z_{H}\rightarrow Z_{{M_H}}\rightarrow F^\times\rightarrow 1
\]
which we can split using $\alpha_n^*$. We let ${Z}_{{M_H}}^0$ denote the image of $\alpha_n^*$ then we get a factorization $Z_{{M_H}}=Z_H{Z}^0_{{M_H}}$. In keeping with the notation in \cite{Sh02, CPSS08} we will also use $\alpha^\vee$ to denote $\alpha^*_n$, particularly when thinking about this splitting of the center of ${M_H}$,  to wit
\[
\alpha^\vee(t)=\bpm tI_n \\ & I_n\epm.
\]

Let $X( \M_\H)_F$ be the group of $F$-rational characters of $\M_\H$. Since  $\M_\H\simeq \G\times \G_m$, this is  the  free rank two $\mathbb Z$-module spanned by the determinant and the similitude character. Let
${\mathfrak a}_H=\Hom(X(\M_\H)_F, \bR)$, ${\mathfrak a}_H^*=X( \M_\H)_F\otimes_{\bZ}
\bR$ as in \cite{Sh90}, and ${\mathfrak a}_{H,\bC}^*={\mathfrak
a}_H^*\otimes_{\bR} \bC$. Also, define $H_{ P_H}=H_{M_H} :  M_H\to {\mathfrak a}_H$ by
\[
q^{\la\chi,H_{M_H}( m)\ra}=|\chi( m)|_F\quad\text{for all }\quad \chi\in X(\M_\H)_F.
\]
(We hope the standard use of $H$ for the Harish-Chandra ``logarithm'' and our use of $H=GSp_{2n}(F)$ will not cause confusion. The distinction should be clear from context.)

\subsection{Reduction to local coefficients}\label{rlc}
Let $\pi$ be an irreducible admissible $\psi$-generic representation
of $G$.  Since $G= \gl_n(F)$, this is independent of the choice of the character and the splitting chosen.
 We extend this to a representation of $M_H\simeq G\times\gl_1(F)$ by making it trivial on the $\gl_1$ factor, i.e., as  $\pi\otimes\mathbf 1$. The choice of the extension does not matter, since any extension will lead to the same local coefficient. We will continue to denote this representation by $\pi$ and hope that there is no confusion. We note that in particular, with the notation of \eqref{tildem},  $\omega_\pi( m)=\omega_{\pi\otimes\mathbf 1}({m})=\omega_\pi(g)$.

 Given $\nu\in{\mathfrak  a}_{H,\bC}^*$, let 
\[
I(\nu,\pi)=\Ind_{M_HN_H}^{{H}}(\pi\otimes q^{\la \nu,H_{{M_H}}(-)\ra}\otimes 1)
\]
 be the induced
representation, and denote its space by $V(\nu,\pi)$. 
Let 
\[
\widehat{\alpha}=\la{\rho},{\alpha}_n\ra^{-1}{\rho}=\frac{1}{2}(e_1+\cdots+ e_n)-\frac{n}{4}e_0
\]
as in
\cite{CPSS08}.
Now for $s\in
\bC$, we define $I(s, \pi)=I(s\widehat{\alpha}, \pi)$ and  let $V(s,\pi)$ be
its space. Note that  $q^{\la s\widehat{ \alpha}, H_{M_H}(m)\ra}=|\det(g)|^{s/2}|a_0|^{-ns/4}$.
The standard intertwining operator $A(s,\pi):I(s,\pi)\rightarrow
I(-s,w_0(\pi))$, as defined in equation (2.6) of \cite{Sh02}, is
given by
\[
A(s,\pi)f(h)=\int_{{N_H}} f(\dw_0^{-1} n h) dn
\]
for all $h\in{H}$, and $f\in V(s,\pi)$.

Let $\lambda$ be a Whittaker functional for $\pi$.  If
$\lambda_\psi(s,\pi)$ is the Whittaker functional (see
\cite{Sh81, Sh90}) for $I(s, \pi)$ canonically  attached to $\lambda$ defined by
\[
\lambda_\psi(s,\pi)(f)=\int_{{N_H}} \la f(\dw_0^{-1} n),\lambda\ra\cdot {\psi^{-1}(n)} dn,
\]
which is the equation (2.6) of \cite{Sh02}, then the local
coefficient $C_\psi(s,\pi)$ is defined by
\[
\lambda_\psi(s,\pi)=C_\psi(s,\pi)\cdot \lambda_\psi(-s,w_0(\pi))\cdot A(s,\pi).
\]

Next, we denote by $^L{{\mathfrak n}_\H}$, the Lie algebra of the $L$-group
of ${\N_\H}$.  Let $r$ be the adjoint action of
$^L{{\M_\H}}$ on $^L{{\mathfrak n}_\H}$.  For $\H=GSp_{2n}$, $^L{{\H}}$ can be taken to be $GSpin_{2n+1}(\mathbb C)$, $^L{{\P_\H}}$ the Siegel parabolic subgroup of $GSpin_{2n+1}(\mathbb C)$ fixing a maximal isotropic subspace, $^L{{\N_\H}}$ its unipotent radical, and $^L{{\M_\H}}\simeq \gl_n(\mathbb C)\times \gl_1(\mathbb C)$.  The representation $r$ of $\gl_n(\mathbb C)\times\gl_1(\mathbb C)$ on $^L{{\mathfrak n}_\H}$ then decomposes as $r=r_1\oplus r_2$ where $r_1=St$ is the standard $n$-dimensional representation of $\gl_n(\mathbb C)$ and $r_2=\Lambda^2$ is the exterior square representation of $\gl_n(\mathbb C)$.  The copy of $\gl_1(\mathbb C)$ lies in the center, and so acts trivially on $^L{{\mathfrak n}_H}$. Then by \cite{Sh90, ShB} we have
\beq\label{lcg}
C_\psi(s,\pi)=\gamma(s,\tilde \pi,\psi^{-1})\gamma(2s,\tilde\pi,\Lambda^2,\psi^{-1})
\eeq
where $\gamma(s,\tilde\pi,\psi^{-1})=\gamma(s,\tilde\pi,St,\psi^{-1})$ is the standard $\gamma$-factor for $\tilde\pi$, the contragredient of $\pi$,  and $\gamma(2s,\tilde\pi,\Lambda^2,\psi^{-1})$ is the factor we want to prove stability for.

The standard $\gamma$-factor for $\gl_n$ is known to be stable under highly ramified twists, with the stable form depending only on the  central character, by a result of Jacquet and Shalika \cite{JS85}. So the above equality reduces Proposition \ref{scstab} to proving stability for the associated local coefficients for supercuspidal representations of $G$.

\subsection{An integral representation for local coeficients}\label{irlc}

We now specialize the results in \cite{Sh02} and \cite{CPSS05} to our situation. As in \cite{Sh02, CPSS05}, we start with a 
Bruhat decomposition. The big Bruhat cell in $H$ relative to $P_H$, $M_HN_Hw_0N_H$,  can be translated to  $M_HN_Hw_0N_Hw_0^{-1}=M_HN_H\overline{N}_H$.  Then for $n$ in an open set in $N$ we have $\dw_0^{-1}n\in {M_H}N_H\overline{N}_H$ 
and we can write
\beq
\dw_0^{-1} n=mn'\overline n, \label{bruhat}
\eeq 
where $m\in M_H$, $n'\in
N_H$, and $\overline n\in \overline N_H$. Indeed if we write
\[
n=\bpm I_n & y \\ & I_n\epm
\]
as in \eqref{n}, then under the open condition $\det(y)\neq 0$ we have
\beq\label{decomp}
\dw_0^{-1}n=\bpm & (-1)^nI_n \\ I_n\epm\bpm I_n & y \\ & I_n\epm=\bpm (-1)^{n-1}y^{-1} \\ & y\epm\bpm I_n & -y \\ & I_n\epm\bpm I_n \\ y^{-1} & I_n \epm
\eeq
and one can check from the condition  $y=J{^ty}J$ that the components of this decomposition indeed live in $G\subset M_H$, 
${N_H}$, and $\overline{{N}_H}$ as claimed.

Let $\mu$ be the densely defined map from $N_H$ to $G\subset M_H$ sending $n$ to its $m$ component  in \eqref{bruhat}; in the notation of \eqref{decomp}, when $\det (y)\neq 0$ $\mu\bpm I_n & (-1)^{n-1}y^{-1} \\ & I_n\epm=\bpm y \\ & (-1)^{n-1}y^{-1}\epm$.

For $ m$ occurring in the decomposition (\ref{bruhat}) 
we may apply the
Bruhat decomposition in $G$ relative to its Borel subgroup $B$  to write
\beq
 m=u_1 \dw a  u_2 \label{Mbruhat}
\eeq
where $u_1, u_2\in U$, $ a\in {A}$ and $\dw$ represents an element in the Weyl group of $G$.

If we let $B'_H=\dw_0{B_H}\dw_0^{-1}=A_H\overline{N}_HU$ then there is a unique $B_H'$ Bruhat cell 
$C'(\overline{w})=B_H'\overline{w}B_H'$ which intersects $N_H$ in an open set $N_H(\overline{w})=N_H\cap C'(\overline{w})$, 
and for  $n\in N_H(\overline{w})$ will have such a decomposition. 
For $n\in N_H(\overline{w})$ we have $\ m=\mu(n)$ lies in a unique $B$-Bruhat cell in $G\subset M_H$, call it 
$C(\tilde w)=B\tilde{w}B$ and this is the Bruhat cell in $G$ that intersects $\Im(\mu)$ in a 
subset of highest possible dimension. From Proposition 3.2 in \cite{CPSS08} (see also Remark 1.11 in \cite{SK09})  we have 
the relation $\overline{w}=w_0\tilde{w}$. 
In our situation  W. Kuo, in the case 2(b) in his appendix to \cite{SK09}, has used the analysis of \cite{Sund} to compute $
\tilde{w}$ and $\overline{w}$ in our case.
We have  $\tilde{w}=w_\ell$  and then $\overline{w}=w^H_\ell$ as in \eqref{welts}.   

Note that in order to use Theorem 6.2 of \cite{Sh02} to establish the
stability, the $\dot{w}$ appearing here must support a Bessel
function on $G$ in the sense of Section 2.2 in \cite{CPSS05}, recalled in Section \ref{bfio} below. For us $\tilde{w}=w_\ell$ and  this supports a Bessel function by the criterion in \cite{CPSS05}. 
Also, 
Assumption 3.6 in \cite{CPSS05} holds by the work of 
Sundaravaradhan \cite{Sund}.

 The Bessel function associated to $\pi$ and $w_\ell$ is a function on the Bruhat cell $C_{M_H}(w_\ell)=U w_\ell A_H  U$  of the form 
 \beq
j_{\pi, {w_\ell}}(m)
=\int_{U_{M}}W({m}u)\psi^{-1}(u)du
\eeq
with $W \in \mathcal W(\pi,\psi)$ with $W(e)=1$.   
In our setting, we have taken the representation $\pi$ of $G$ and extended it to $M_H\simeq G\times GL_1(F)$ by making it trivial on the $GL_1(F)$ factor. The Whittaker function, which goes into the definition of the Bessel function, is then a Whittaker function on $G$ multiplied by the trivial character on $GL_1(F)$.  So we can write the Bessel function as a function on the Bruhat cell $C(w_\ell)=U w_\ell A  U$ of $G$, namely
\[
j_{\pi,{w_\ell}}({m})=\int_{U} W(gu)\psi^{-1}(u)du=j_{\pi,w_\ell}(g)
\]
with the notation in  \eqref{tildem}.   We will delay the discussion of  the convergence of this integral, which can be delicate in general, to Section \ref{bfio}.

The function that appears in the integral representation for the local coefficient in \cite{Sh02} is not this Bessel function itself, but rather a partial Bessel function. To define this, we take a cutoff function $\varphi_{\overline{N}_0}$ on $\overline{N}_H$, which is the characteristic function of a sufficiently large open compact  subgroup $\overline{N}_0$ of  $\overline{N}_H$. We can make this very explicit in our case. Fix  $\kappa\in\mathbb N$ and let $\varphi_\kappa$ be the characteristic function of a neighborhood of $0$ in $Mat_n(F)$ defined by the condition that the $(i,j)$--entry of the matrix are bounded in absolute value by $q^{(i+j-1)\kappa}$, i.e.
 \[
\varphi_\kappa(X)=\begin{cases}  1 &   |X_{i,j}|\leq q^{(i+j-1)\kappa} \\ 0 & \text{otherwise}\end{cases}.
\]
The group $\overline{N}_H$ is given by
\[
\overline{N}_H=\left\{ \bar{n}(y)=\bpm I_n \\ y & I_n \epm \mid y=J {^ty}J\right\}.
\]
We set  
\[
\overline{N}_0=\overline{N}_{0,\kappa}=\left\{ \bar{n}(y) \mid \varphi_\kappa(\varpi^{d+f}\dw_\ell y)=1\right\}=\{ \bar{n}(y) \mid |y_{n-i+1,j}|\leq q^{(i+j-1)\kappa-(d+f)}\},
\]
where $d$ is the conductor of $\psi$ and $f$ is the conductor of $\omega_\pi^{-1}(w_0\omega_\pi)$,
and let  $\varphi_{\overline{N}_{0,\kappa}}$  be the characteristic function of $\overline{N}_{0,\kappa}$. As $\kappa$ grows, these subgroups exhaust $\overline{N}_H$. Then the partial Bessel function $j_{\pi,w_\ell,\kappa}({m},z)$ associated to $\pi$, ${w_\ell}$ and the cutoff parameter $\kappa$ (or equivalently $\overline{N}_{0,\kappa}$) is the function on ${M_H}\times Z^0_{{M_H}}$ given by
\[
j_{\pi,{w_\ell},\kappa}({m},z)=j_{\pi,{w_\ell},\overline{N}_{0,\kappa}}({m},z)=\int_{U}W_v(gu)\varphi_{\overline{N}_{0,\kappa}}(zu^{-1}\bar{n}uz^{-1})\psi^{-1}(u)\ du
\]
(compare with (6.21) of \cite{Sh02}) where ${m}$ and $\bar{n}$ are related to $n$ through \eqref{bruhat} or \eqref{decomp}. We finally let $z=\alpha^\vee(\varpi^{d+f}u_{{\alpha}_n}(\dw_0\bar{n}\dw_0^{-1}))$, where $d$ is the conductor of $\psi$ and $f$ is the conductor of $\omega_\pi^{-1}(w_0\omega_\pi)$, and set
\[
j_{\pi,{w_\ell},\kappa}(g)=j_{\pi,{w_\ell},\overline{N}_{0,\kappa}}({m},\alpha^\vee(\varpi^{d+f}u_{{\alpha}_n}(\dw_0\bar{n}\dw_0^{-1})))
\]
(compare with (6.24) and (6.39) of \cite{Sh02}).

 We let $\pi_s$ denote the representation $\pi\otimes q^{\la s\hat{\alpha},H_{{M_H}}(-)\ra}$. This will have central character  $\omega_{\pi_s}( m)=\omega_\pi(g)|\det(g)|^{s/2}|a_0|^{-ns/4}$.  With this notation we can restate Theorem 6.2 of \cite{Sh02} in our case.

\begin{prop}\label{integrep1}
Let $\pi$ be an irreducible admissible generic representation
of $G$.  Suppose $\omega_\pi(w_0\omega_\pi^{-1})$ is ramified as
a character of $F^\times$.  Then for all sufficiently large  $\kappa$
\beq\label{ir1}
\begin{aligned}
C_\psi(s,\pi)^{-1}&=\gamma(2\la\hat\alpha,\alpha^\vee\ra s,\ \omega_\pi(w_0\omega_\pi^{-1})\circ\alpha^\vee,\ \psi)^{-1}\\
&\cdot\int_{Z^0_{{M_H}} U\backslash N_H} j_{\pi,w_\ell,\kappa}(g) \omega^{-1}_{\pi_s} (\alpha^\vee(u_n))(\dw_0\omega_{\pi_s})(\alpha^\vee(u_n))
q^{\la s\hat\alpha+\rho, H_{M_H}(m)\ra }dn,
\end{aligned}
\eeq
where, off a set of measure zero, we decompose $\dw_0^{-1}n=mn'\bar{n}$ as in \eqref{bruhat} with $m=\mu(n)$,  $g$ related to ${m}$ as in \eqref{tildem},  $u_n=u_{\tilde{\alpha}_n}(\dw_0\bar{n}\dw_0^{-1})$, and  the integration is over the set of $Z^0_{M_H}U$ orbits in ${N_H}$.  
\end{prop}

The factor $\gamma(2\la\hat\alpha,\alpha^\vee\ra
s,\ \omega_\pi(w_0\omega_\pi^{-1})\circ\alpha^\vee,\ \psi)$ is a simple abelian
$\gamma$-factor depending only on the central character of $\pi$.

We can simplify this in our situation as follows. First, from the formulas above for $\hat{{\alpha}}$ and $\alpha^\vee$ we have 
\[
\la\hat{{\alpha}},\alpha^\vee\ra=\frac{n}{2}-\frac{n}{4}=\frac{n}{4}.
\]
Next,
\[
\omega_\pi(w_0\omega_\pi^{-1})\circ \alpha^\vee=\omega_\pi \quad\text{and}\quad \omega_{\pi_s}^{-1}(w_0\omega_{\pi_s})\circ \alpha^\vee=\omega_\pi^{-1}|\cdot|^{-ns/2}
\]
as a characters of $F^\times$. Finally
\[
q^{\la s\hat{{\alpha}}+\tilde{\rho}, H_{M_H}({m})\ra}= |\det(g)|^{(s+n+1)/2}|a_0|^{-(n(s+n+1))/4}.
\]

Then we can restate Proposition \ref{integrep1} as follows.

\begin{prop}\label{integrep2}
Let $\pi$ be an irreducible admissible generic representation
of $G$.  Suppose $\omega_\pi$ is ramified as a character of $F^\times$.  Then for all sufficiently large  $\kappa$
\beq\label{ir2}
\begin{aligned}
C_\psi(s,\pi)^{-1}&=\gamma(\tfrac{ns}{2},\ \omega_\pi,\ \psi)^{-1}\\
&\cdot \int_{Z^0_{M_H} U\backslash N_H} j_{\pi,w_\ell,\kappa}(g) \omega^{-1}_{\pi}(u_n) |u_n|^{-ns/2}
|\det(g)|^{(s+n+1)/2}|a_0|^{-n(s+n+1)/4}dn,
\end{aligned}
\eeq
where, off a set of measure zero, we decompose $\dw_0^{-1}n=mn'\bar{n}$ as in \eqref{bruhat} with ${m}=\mu(n)$ and $u_n=u_{\tilde{\alpha}_n}(\dw_0\bar{n}\dw_0^{-1})$; $g$ and $a_0$ are related to ${m}$ as in  \eqref{tildem}, and  the integration is over the set of $Z^0_{{M_H}}U$ orbits in ${N_H}$.
\end{prop}

For our proof of stability, we will need to consider this integral representation for local coefficients $C_\psi(s,\pi\otimes\chi)^{-1}$ for sufficiently ramified characters $\chi$ of $F^\times$. It will be important for us to be able to choose $\kappa$, or equivalently $\overline{N}_0\subset \overline{N}$, to be independent of $\chi$. 

To establish this uniformity, we must recall how the subgroup $\overline{N}_0$ comes about in the proof of Theorem 6.2 of \cite{Sh02}. If we fix an irreducible generic representation $\pi'$ such that $\omega_{\pi'}$ is ramified, then the subgroup $\overline{N}_0$ in in Theorem 6.2 of \cite{Sh02} is chosen to satisfy the following two conditions:
\begin{enumerate}
\item There exists a section $f\in V(s,\pi')$ of the induced representation such that $f$ is supported in $P_H\overline{N}_0$.
\item $\overline{N}_0$ is sufficiently large so that $\alpha^\vee(t)\overline{N}_0\alpha^\vee(t)^{-1}$  depends only on $|t|$ for all $t\in F^\times$.
\end{enumerate}
The second condition is independent of $\pi'$. As for the first condition, it is known (see the proof of Theorem 6.2 of \cite{Sh02}) that there exist  sections $f$ in $V(s,\pi')$ compactly supported in $P_H\overline{N}$ modulo $P_H$.  We fix such a section $f$ and  then choose  $\overline{N}_0$ large enough to contain the support, that is, so that $f$ is supported in $P_H\overline{N}_0$.

So let us fix a $\chi_0$ such that $\omega_\pi\otimes \chi^n_0$, which is the central character of $\pi\otimes\chi_0$, is ramified. Then let us choose $\kappa_0$ such that (1) and (2) hold for $\overline{N}_{0,\kappa_0}$ and an appropriate section $f_{\chi_0}$ in $V(s,\pi\otimes\chi_0)$. If $\kappa\geq \kappa_0$ then $\overline{N}_{0,\kappa}\supset \overline{N}_{0,\kappa_0}$ so that (1) and (2) also hold for any $\kappa\geq \kappa_0$. So the formula of Proposition \ref{integrep2} holds for $\pi\otimes\chi_0$ and any $\kappa\geq \kappa_0$.

Now let $\chi$ be any other character of $F^\times$ such that $\omega_\pi\chi^n$ is ramified. Then, as noted above, there is a section $f_\chi\in V(s,\pi\otimes\chi)$ of the type needed in Theorem 6.2 of \cite{Sh02} which is supported in $P_H\overline{N}_{0,\chi}$ for some compact open $\overline{N}_{0,\chi}\subset \overline{N}$. If $\overline{N}_{0,\chi}\subset\overline{N}_{0,\kappa_0}$, then Proposition \ref{integrep2} will hold for $\pi\otimes\chi$ and any $\kappa\geq\kappa_0$. 

On the other hand, if $\overline{N}_{0,\chi}\not\subset\overline{N}_{0,\kappa_0}$, then we note that if we right translate $f_\chi$ by $e^*_0(t)\in{M_H}$ then $R(e_0^*(t))f_\chi$ will be supported in $P_H[e_0^*(t)\overline{N}_{0,\chi}e_0^*(t)^{-1}]$. Note that for $\bar{n}\in\overline{N}_H$ we have
\[
e_0^*(t)\overline{n}(y)e_0^*(t)^{-1}=\bpm I_n  \\ & tI_n\epm\bpm I_n \\ y  & I_n \epm\bpm I_n \\ & t^{-1}I_n\epm=\bpm I_n \\ ty & I_n\epm,
\]
thus,  for $|t|<1$ this action contracts on $\overline{N}_H$. So for $|t|$ sufficiently small,  $e_0^*(t)\overline{N}_{0,\chi}e_0^*(t)^{-1}\subset \overline{N}_{0,\kappa_0}$. Hence, fixing such a $t$,  if we use $f'_\chi=R(e_0^*(t))f_\chi$ in place of $f_\chi$, then $f'_\chi$ has support contained in $P_H\overline{N}_{0,\kappa_0}$ and Proposition \ref{integrep2} holds for $\pi\otimes\chi$ and any $ \kappa\geq\kappa_0$. 

Taken together, this establishes the following strengthening of Proposition \ref{integrep2}. 

\begin{prop}\label{uniformintegrep} Let $\pi$ be an irreducible generic  representation of $G$. Then there exists a $\kappa_0$ such that for all $\kappa\geq \kappa_0$ and all $\chi$ such that $\omega_\pi\chi^n$ is ramified we have
\beq\label{uir}
\begin{aligned}
C_\psi(s,\pi\otimes&\chi)^{-1}=\gamma(\tfrac{ns}{2},\ \omega_{\pi}\chi^n,\ \psi)^{-1}\\
&\cdot \int_{Z^0_{{M_H}} U\backslash N_H} j_{\pi\otimes\chi,w_\ell,\kappa}(g) (\omega_{\pi}\chi^n)^{-1}(u_n) |u_n|^{-ns/2}
|\det(g)|^{(s+n+1)/2}|a_0|^{-n(s+n+1)/4}dn,
\end{aligned}
\eeq
where, off a set of measure zero, we decompose $\dw_0^{-1}n=m n'\bar{n}$ as in \eqref{bruhat} with ${m}=\mu(n)$ and $u_n=u_{\tilde{\alpha}_n}(\dw_0\bar{n}\dw_0^{-1})$; $g$ and $a_0$ are related to $\tilde{m}$ as in  \eqref{tildem}, and  the integration is over the set of $Z^0_{{M_H}}U$ orbits in ${N_H}$.
\end{prop}

Now we have the reciprocal of the local coefficient, up to
abelian $\gamma$-factors,  as an integral  transform of a partial Bessel function. We will identify this as a multi-variable Mellin transform in the next section.

\subsection{Change of variables}
By Propositions \ref{integrep1} and \ref{integrep2}, we have an integral representation for the
inverse of the local coefficient $C_\psi(s,\pi)^{-1}$.   In this
section,  we will replace the domain $Z^0_{{M_H}}U\backslash N_H$ by a
suitable torus inside $Z\backslash A$ and  will show that
the integral representation given above can be written as a Mellin
transform of a Bessel function. This  is similar to Proposition 2.1
in \cite{CPS2} and Theorem 4.22 in \cite{CPSS08}.

We begin with  the following description of $U\backslash N_H$.

\begin{prop}\label{R}
  Set 
\[
R
=\left\{ \bpm I_n& a\dw_\ell \\ & I_n\epm\mid a
=\diag(a_1,\dots,a_n)\in A\right\}.
\]
Then, on a open dense subset of $N_H$, $R$ is a set of representatives of $U\backslash N_H$.
\end{prop}

We begin by recalling  the following elementary lemma on symmetric matrices.

\begin{lemma}\label{sym1} There exists an open dense subset $Sym'_n(F)$ of the space of symmetric $n\times n$ matrices with the property that for each matrix $S$ in this subset one has an upper triangular unipotent matrix
 $u$ and a non-singular diagonal matrix $t$ such that
$S=ut{^tu}$.
\end{lemma}

The Zariski open subset $Sym'_n(F)$ is characterized by the non-vanishing of the  principal minors, beginning from the lower right corner. It is now easy to prove the proposition.

\begin{pf} Suppose $n=\bpm I_n &y \\ & I_n\epm \in N_H$. Then $y$ satisfies $y=J{^ty}J$. Note that as matrices $J=\dw_\ell$. If we let $S=y\dw_\ell$, then $S$ is symmetric and, by the lemma, on the open dense set $Sym'_n(F)$  we can write $S=ua{^tu}$ with $a\in A$ and $u\in U$. Then, on this open set, $y=S\dw_\ell=ua\dw_\ell(\dw_\ell^{-1}{(^tu)}\dw_\ell)$. As the action of $U$ on $N_H$ translates into this twisted conjugacy on $y$, we see that on the open dense set  of those $n$ such that $y\dw_\ell\in Sym'_n(F)$, $R$  gives a complete set of representatives for $U\backslash N_H$. 
\end{pf}

If we take $r=\bpm I_n & a\dw_\ell\\& I_n\epm\in R$ and apply the decomposition \eqref{decomp} then we see 
\[
\mu(r)={m}=\bpm (-1)^{n-1}(a\dw_\ell)^{-1} \\ &a\dw_\ell\epm= \bpm \dw_\ell a^{-1} \\ &(-1)^{n-1}J{^t(\dot{w}a^{-1})^{-1}}J^{-1}\epm
\]
so that  $g=\dw_\ell a^{-1}$ and $a_0=(-1)^{n-1}$.  So in \eqref{ir2} we have $|\det(g)|=|\det(a)|^{-1}$ and $|a_0|=1$. Then
\[
\bar{r}=\bpm I_n \\ (-1)^{n-1}\dw_\ell a^{-1} & I_n\epm \quad\text{and}\quad w_0\bar{r}w_0^{-1}=\bpm I_n &-\dw_\ell a^{-1}\\ & I_n\epm
\]
so that in \eqref{ir2} we have $u_n$ is the lower left entry of $-\dot{w}a^{-1}$ which is $(-1)^na_1^{-1}$. 
Thus, as a first step, we can write the integral in Proposition \ref{integrep2} as
\beq\label{ir3}
\omega_\pi(-1)^n\int_{Z^0_{M_H}\backslash R} j_{\pi,\dot{w},\kappa}(\dot{w}a^{-1})\omega_\pi(a_1)|a_1|^{ns/2}|\det(a)|^{-(n+s+1)/2}\left|\frac{dn}{dr}\right|\ dr
\eeq
and we are left with computing the Jacobian  $\left|\frac{dn}{dr}\right|$. 

Under the isomorphism $R\simeq A=\{\diag(a_1,\dots,a_n)\mid a_i\in F^\times\}$ we take the invariant measure $dr$ on $R$ to be $dr=\prod_id^\times a_i=\prod_i\frac{da_i}{|a_i|}$. Then we can easily compute this Jacobian.

\begin{prop} $\displaystyle \left|\frac{dn}{dr}\right|=\prod_i |a_i|^i$.
\end{prop}

\begin{pf} As we have noted in  the proof of Proposition \ref{R}, the action of $U_{{M}}$ on $N$ by conjugation is equivalent, under a multiplication by $\dot{w}$, to the standard action of $U_M$ on $Sym_n(F)$.  So we can compute this Jacobian there.

We will prove the formula by induction on $n$, utilizing Lemma \ref{sym1}. The case $n=1$ is immediate.

In general, for $S=(s_{i,j})\in Sym'_n(F)$ we let $dS=\prod_i ds_{i,i}\prod_{i<j}ds_{i,j}$. Utilizing Lemma \ref{sym1} we write
\[
S=\bpm I_{n-1} &  x \\ & 1\epm \bpm S' & \\ & a_n \epm \bpm I_{n-1} \\ {^tx} & 1\epm=\bpm S'+a_n x{^tx} & a_nx\\a_n{^tx} & a_n\epm
\]
with $S'\in Sym'_{n-1}(F)$ and $x\in F^{n-1}$. We then compute
\begin{align*}
dS&=|\wedge_ids_{i,i}\wedge_{i<j}ds_{i,j}|=|(d(S'+a_nx{^tx}))\wedge(\wedge_i(a_ndx_i+x_ida_n))\wedge da_n|\\
&=|(dS'+d(a_nx{^tx}))\wedge(\wedge_i (a_ndx_i))\wedge da_n|=|a_n|^{n-1}|dS'\wedge(\wedge_idx_i)\wedge da_n|\\
\intertext{and}
&=|a_n|^n dS'\prod_i(dx_i)d^\times a_n.
\end{align*}
By induction, we have $dS'=\prod_{i=1}^{n-1}(|a_i|^i d^\times a_i) du_{n-1}$ where we have used $du_{n-1}$ to denote the invariant measure on the upper triangular subgroup $U_{n-1}\subset GL_{n-1}(F)$.  Substituting into the result of the previous computation then gives
\[
dS=\prod_{i=1}^n(|a_i|^id^\times a_i)du_n.
\]
As $dn=\frac{dS}{d{u_n}}$ and $dr=\prod_id^\times a_i$ we arrive at $\left|\frac{dn}{dr}\right|=\prod_i|a_i|^i$ as claimed.
\end{pf}

Observe that $R\simeq A\simeq Z_H\backslash {A_H}$ is a $n$-dimensional torus as desired. Furthermore, it is easy to see that under this isomorphism of $R$ with $ A$, the action of $Z^0_{M_H}$ on $R$ by left translation becomes the action of $Z$ on $A$, again  by left translation. Thus $Z^0_{M_H}\backslash R\simeq Z\backslash A$. If we combine this with the Jacobian calculation, then the integral  in \eqref{ir3}  becomes
\beq\label{ir4}
\omega_\pi(-1)^n
\int_{Z\backslash A} j_{\pi,\dw_\ell,\kappa}(\dw_\ell a^{-1}) \omega_\pi(a_1)|a_1|^{ns/2}|\det(a)|^{-(n+s+1)/2}\prod_i|a_i|^i\ da
\eeq
and if we then send $a\mapsto a^{-1}$ and simplify, we arrive at the following proposition, which identifies the local coefficient as a Mellin transform of the partial Bessel function.

\begin{prop}\label{integrep3}
Let $\pi$ be an irreducible admissible generic representation
of $G$.  Suppose $\omega_\pi$ is ramified as a character of $F^\times$.  Then for all sufficiently large $\kappa$,
\beq\label{ir5}
\begin{aligned}
C_\psi(s,\pi)^{-1}&=\gamma(\tfrac{ns}{2},\ \omega_\pi,\ \psi)^{-1}\omega_\pi(-1)^n\\
&\cdot\int_{Z\backslash A} j_{\pi,\dw_\ell,\kappa}(\dw_\ell a) \omega_\pi(a_1)^{-1}|a_1|^{-(n-1)(s-1)/2}\prod_{i=2}^n|a_i|^{(n+s+1-2i)/2}\ da.
\end{aligned}
\eeq
Moreover,  for $\pi$  any irreducible generic  representation of $G$, there exists a $\kappa_0$ such that for all $\kappa\geq \kappa_0$ and all $\chi$ for which $\omega_\pi\chi^n$ is ramified we have
\beq\label{uir2}
\begin{aligned}
C_\psi(s,\pi\otimes\chi)^{-1}&=\gamma(\tfrac{ns}{2},\ \omega_\pi\chi^n,\ \psi)^{-1}\omega_\pi\chi^n(-1)^n\\
&\cdot\int_{Z\backslash A} j_{\pi\otimes\chi,\dw_\ell,\kappa}(\dw_\ell a) \omega_{\pi}\chi^n(a_1)^{-1}|a_1|^{-(n-1)(s-1)/2}\prod_{i=2}^n|a_i|^{(n+s+1-2i)/2}\ da.
\end{aligned}
\eeq

\end{prop}

Let us revisit our partial Bessel function $j_{\pi,\dw_\ell,\kappa}(\dw_\ell a)$ in terms of these change of variables. By definition we have
\[
j_{\pi,\dw_\ell,\kappa}(\dw_\ell a)=\int_{U}W_v(\dw_\ell au)\varphi_{\overline{N}_0}(zu^{-1}\bar{n}uz^{-1})\psi^{-1}(u)\ du
\]
where $z=\alpha^\vee(\varpi^{d+f}u_{{\alpha}_n}(w_0\bar{n}w_0^{-1}))$. There is a relation between $g=\dw_\ell a$ and $\bar{n}$ mediated by \eqref{decomp}. Taking $y^{-1}=\dw_\ell a$ in \eqref{decomp} gives
\beq
\bpm & (-1)^nI_n \\ I_n\epm\bpm I_n & (\dw_\ell a)^{-1} \\ & I_n\epm=\bpm (-1)^{n-1}\dw_\ell a  \\ & (\dot{w}a)^{-1}\epm\bpm I_n & -(\dot{w}a)^{-1}  \\ & I_n\epm\bpm I_n \\ \dot{w}a & I_n \epm
\eeq
so we see that in the above formula $\bar{n}$ is given by
\[
\bar{n}=\bpm I_n \\ \dot{w}a & I_n \epm.
\]
If we write 
\[
u=\bpm u_0 \\ & J{^tu_0^{-1}} J^{-1}\epm \quad \text{and} \quad z=\bpm z_0I_n\\& I_n\epm=\alpha^\vee(z_0)
\]
then
\[
zu^{-1}\bar{n}uz^{-1}=\bpm I_n \\ z_0^{-1} J {^tu_0}J^{-1} \dw_\ell a u_0 & I_n \epm.
\]
Thus
\[
\varphi_{\overline{N}_{0,\kappa}}(zu^{-1}\bar{n}uz^{-1})=\varphi_\kappa(\varpi^{d+f}\dw_\ell z_0^{-1} J {^tu_0}J^{-1} \dw_\ell a u_0).
\]
As matrices, $J_n=\dw_\ell$, so $z_0^{-1} J {^tu_0}J^{-1} \dw_\ell a u_0=z_0^{-1}\dw_\ell{^tu_0}au_0$. Next, for $\bar{n}$ as above, $u_{\tilde{\alpha}_n}(w_0\bar{n}w_0^{-1})$ is the lower left entry of $\dw_\ell a$, which is $(-1)^{n-1}a_1$. So $z_0=\varpi^{d+f}u_{{\alpha}_n}(w_0\bar{n}w_0^{-1})=\varpi^{d+f}(-1)^{n-1}a_1$. If we write $a=ta'$ with $t=\diag(a_1,\dots,a_1)$ and $a'=\diag(1,a_2',\dots,a_n')$ then 
\[
\varphi_\kappa(\varpi^{d+f}\dw_\ell z_0^{-1} J {^tu_0}J^{-1} \dw_\ell a u_0)=\varphi_{\kappa}(\dw_\ell \dw_\ell{^tu_0}a'u_0)=\varphi_{\kappa}({^tu_0}a'u_0)
\]
and hence
\beq
\label{pbf}
j_{\pi,\dw_\ell,\kappa}(\dw_\ell a)=j_{\pi,\dw_\ell,\kappa}(\dw_\ell ta')=\omega_\pi(t)\int_{U}W_v(\dw_\ell a'u_0)\varphi_{\kappa}({^tu_0}a'u_0)\psi^{-1}(u_0)\ du_0
\eeq
where we  now view all the variables $\dw_\ell, t, a', u_0$ as matrices in $G$.

\subsection{Bessel functions and Bessel integrals} \label{Bfbi}

Conceptually, a Bessel function on $G$ is a function satisfying $j(u_1gu_2)=\psi(u_1)j(g)\psi(u_2)$ for $u_1,u_2\in U$. If we decompose $G$ into Bruhat cells, so $G=\coprod C(w)$ with $C(w)=U\dw AU$ and restrict a Bessel function to a Bruhat cell $C(w)$, the compatibility of the right and left transformation laws under $U$ put restrictions on which cells $C(w)$ can support Bessel functions and within such cells which parts of $A$ are allowed.

Let $B(G)\subset W$ consist of those Weyl group elements that support Bessel functions in the sense of \cite{CPSS05}, that is, $w\in W$ such that for every simple root $\alpha\in \Delta$ we have that $w\alpha>0$ implies $w\alpha\in \Delta$, or equivalently, $w_\ell w$ is the long element of a standard Levi subgroup of $G$. In particular, 
from Section 2.2 of \cite{CPSS05}, to a $w\in B(G)$ we associate the set of simple roots
\[
\theta^+_w=\{\alpha\in \Delta \mid w\alpha>0\}\subset \Delta
\]
which then determines a standard parabolic subgroup $P_{w}$  with Levi $M=M_{w}$ and long Weyl element $w_\ell^{M}\in M_{w}$ such that $w_\ell^{M}=w_\ell w$. Then we also have
\[
\theta^+_w=\theta^-_{w_\ell^M}=\Delta_M\subset \Delta.
\]
We have a torus associated to  $w$ 
\[
A_{w}=\{a\in A\mid \alpha(a)=1 \text{ for all }\alpha\in \theta^+_w\}\subset A
\]
which is the center $Z_{M_w}$ of $M_w$.  

(This notation differs slightly from that in Section 2.2 of \cite{CPSS05}. First, there is  a conjugation by $w$. The reason for this can be found in \eqref{ir5}. The Bessel function is evaluated at $\dw a$ and is considered as a function on the Bruhat cell $C(w)$ written as $C(w)=U{\dw}AU$, while in \cite{CPSS05} we evaluated the Bessel functions at $a\dw$ and viewed them on the cell $C(w)$ written as $C(w)=UA\dw U$. 
Also note that  in \cite{J12} Jacquet indexes these tori by the corresponding element $\tilde{w}=w_\ell^M$ of his $R(G)$.) 

We now collect some more useful facts about $B(G)$.  We first recall from \cite{CPSS05} the following lemma.

\begin{lemma} Let $w,w'\in B(G)$. Then $w'\leq w$ iff $M_{w}\subset M_{w'}$.
\end{lemma} 

Thus the Bruhat order on $B(G)$ reverses the containment order on the Levi subgroups. Since the containment on the centers is also reversed, we have that if $w,w'\in B(G)$ then $w'\leq w$ iff $M_{w}\subset M_{w'}$ iff $A_{w'}\subset A_w$.

For $w\in
W$, we have the decomposition $U=U_{w}^{+}U_{w}^{-}$, where
\[
U_{ w}^{+}=U\cap w^{-1}Uw,
\mbox{ and }
U_{w}^{-}=U\cap w^{-1}U^{-}w
\]
with $U^{-}$ being the opposite of $U$ with respect to the diagonal torus $A$.
 Let $w\in B(G)$. By definition 
  \[
  U_w^-=\{u\in U|wuw^{-1}\in U^-\}\quad\text{and}\quad U_w^+=\{u\in U\mid wuw^{-1}\in U\}.
  \]
  Since $w\in B(G)$ we have an associated Levi subgroup $M=M_w\subset G$ such that $w=w_\ell w_\ell^M$.  Let  $U_M=M\cap U$ denote the standard maximal unipotent subgroup of $M$ and $N=N_M$ the unipotent radical of the associated standard parabolic $P$. So $U=U_M\cdot N$. For $w=w_\ell^M$ we have
  \[
  U_{w_\ell^M}^-=U_M\quad \text{and}\quad U_{w_\ell^M}^+=N_M,
  \]
 while for $w=w_\ell$ we have
 \[
 U_{w_\ell}^-=U\quad\text{and}\quad U_{w_\ell}^+=\{e\}.
 \]
 Hence, if $w=w_\ell w_\ell^M$ we have
 \[
 U_w^-=U_{w_\ell^M}^+=N_M\quad\text{and}\quad U_w^+=U_M.
 \]

An important observation is the following. Because of the alternating signs on $\dw_\ell$, an elementary computation gives that for any simple root $\alpha$ we have
\beq\label{ecompat}
\dw_\ell x_\alpha(t)\dw_\ell^{-1}=x_{w_\ell\alpha}(-t)
\eeq
where if $x_\alpha(t)=I+tE_{i,i+1}$ then $x_{w_\ell\alpha}(-t)=I-tE_{n-i+1,n-i}$. Note that the same is true for $\dw_\ell^{-1}$
\[
\dw_\ell^{-1} x_\alpha(t)\dw_\ell=x_{w_\ell\alpha}(-t)
\]
Also, using the block form for $\dw^M_\ell$ given above, the same formulas hold for $\dw^M_{\ell}$ and simple roots $\alpha$ that occur in $M$.
This  then yields  compatibility in the following sense.

\begin{prop}\label{pcompat}
Let $w=w_\ell w^M_\ell\in B(G)$. Then for any $u\in U_w^+=U_M$ we have $\psi(\dw u \dw^{-1})=\psi(u)$.
\end{prop}

Since a Bessel function on $G$ is a function satisfying $j(u_1gu_2)=\psi(u_1)j(g)\psi(u_2)$ for $u_1,u_2\in U$, if we decompose $G$ into Bruhat cells, so $G=\coprod C(w)$ with $C(w)=U\dw AU$, then it is now an easy matter of checking compatibility to see that $j(g)\neq 0$ would imply that  $g\in C_r(w)=U\dw A_wU=U\dw A_wU_w^-$ for $w\in B(G)$ \cite{CPSS05}.  

 In order to construct Bessel functions, let $C_c^{\infty}(G)$ 
be the space of smooth functions of compact support on $G$ and, for any continuous character $\omega$ of the center $Z$, we let    $C^\infty_c(G; \omega)$  be the space of functions on $G$ that are smooth and of compact support modulo the center $Z$ and satisfy $f(zg)=\omega(z) f(g)$
for all $ g\in G $ and $z\in Z$.
We have a projection from $C_c^\infty(G)$ to $C_c^\infty(G;\omega)$ given by
\beq\label{mcproj}
\phi\mapsto f_\phi(g)=\int_{Z} \phi(zg) \omega^{-1}(z) dz.
\eeq
The integral converges since for fixed $g$ the orbit $Zg$ is closed in $G$ and $\phi$ is locally constant of compact support. 
This map is known to be surjective (see Lemma 5.1.1.4 of \cite{War} for example).

Let $\pi$ be an irreducible supercuspidal representation  of $G$ with ramified central character $\omega_\pi$ so that Proposition \ref{integrep3} holds.
Let $\mathcal M(\pi)\subset C^\infty_c(G;\omega_\pi)$ be its space of matrix coefficients. For every $f\in \cM(\pi)$  there exists $\phi_f\in C_c^\infty(G)$ which projects to $f$ by \eqref{mcproj}.
For each $f\in \cM(\pi)$  we may define a
Whittaker function $W^f$ in the associated Whittaker model $\cW(\pi,\psi)$  of $\pi$ by
\[
W^f(g)=\int_{U} f(u_1 g) \psi^{-1}(u_1) du_1=\int_{U}\int_{Z} \phi_f(u_1zg)\omega_\pi(z)\psi^{-1}(u_1)\  dzdu_1
\]
which again converges since the orbit $UZg$ is closed in $G$.
For an appropriate choice of $f\in \mathcal M(\pi)$ we will have $W^f(e)=1$.

For $w\in B(G)$ we can define the Bessel function $j_{\pi, w}$
 on $C_r(w)=U\dw A_wU_{w}^-$ by
\beq\label{bf1}
\begin{aligned}
j_{\pi,w}(g)=B^G(g,f)&=\int_{U^{-}_{ w}} W^f(gu_2) \psi^{-1}(u_2) du_2\\
&=\int_{U^{-}_{ w}} \int_{U}\int_{Z} \phi_f(u_1zgu_2)\omega_\pi(z)\psi^{-1}(u_1)\psi^{-1}(u_2)\  dzdu_1 du_2
\end{aligned}
\eeq
as long as $W^f(e)=1$. Again, since for fixed $g$ the orbit $UZgU^{-}_{w}$ is closed in $G$, the integral converges absolutely.   
We refer to $B^G(g,f)$ as a {\it Bessel integral}. As the convergence of these Bessel functions can be delicate in general, we state this convergence formally for future reference.

\begin{prop}\label{convergence}
Let $\pi$ be an irreducible  supercuspidal representation of $G$. Then for $f\in\mathcal M(\pi)$ with $W^f(e)=1$ and $w\in B(G)$ the Bessel integral  $B^G(g,f)$ converges absolutely for fixed $g\in C_r(w)$.
\end{prop}

The definition of the Bessel function above  is consistent with \cite{CPSS05} since the Bessel function is independent of the Whittaker function used to define it as long as $W(e)=1$. 

For our analysis of stability of the local coefficients, and hence our $\gamma$-factors, we will need the analogue of our partial Bessel function, namely {\it partial Bessel integrals}. To define them, we first need the definition of twisted centralizer.
For $g\in G$ we define the twisted centralizer of $g$ in $U$ by
\[
U_g=\{ u\in U\mid {^tu} \dw_\ell^{-1} gu=\dw_\ell^{-1} g\}. 
\]
Note that the condition that $u\in U_g$ is equivalent to the twisted centralizer condition $\dw_\ell {^tu}\dw_\ell^{-1} gu=g$ and to the commutation relation $gu=\dw_\ell {^tu^{-1}}\dw_\ell^{-1}g$.

We now define the partial Bessel integral $B^G_\vphi(g,f)$ as follows. Let $g\in G$. Then $g\in C(w)$ for some $w\in W$. Write $C(w)=U\dw A U_w^-$  with uniqueness of expression. We decompose $A=ZA'$ as above and so write $C(w)=ZU\dw A'U_w^-=ZC'(w)$, again with uniqueness of expression. We can then decompose $g$ accordingly, that is, write $g=zg'$ with $g'\in U\dw A'U_w^-$. Then the partial Bessel integral is defined by
  \beq\label{pbi}
  \begin{aligned}
  B_\vphi^G(g,f)&=\int_{U_g\backslash U} W^f(gu)\vphi({^tu}\dw_\ell^{-1}g'u)\psi^{-1}(u)\ du\\\
 &=\int_{U_g\bs U}\int_U f(xgu)\vphi({^tu}\dw_\ell^{-1}g'u)\psi^{-1}(x)\psi^{-1}(u)\ dxdu\\
  &=\omega_\pi(z) B^G_\vphi(g', f).
  \end{aligned}
  \eeq
 If we change $u$ by left translating by an element $u'$ of $U_g$ on the left, then by the definition of the twisted centralizer we have
${^tu'}\dw_\ell^{-1}gu'=\dw_\ell^{-1}g$ and $gu'=\dw_\ell{^tu'}^{-1}\dw_\ell^{-1}g$ ; since $\psi(\dw_\ell {^tu'}^{-1}\dw_\ell^{-1})=\psi({^tu'}^{-1})=\psi(u')$ we see that the integral is well defined. Since $UgU$ is closed in $G$ and $f$ is compactly supported  mod $Z$, we see that the integral converges for any $g\in G$.  

To see that this really behaves partially like a Bessel function, we need the following lemma. Let $\varphi=\varphi_N$   be as in Section \ref{irlc}.  Then $\vphi$ is the characteristic function of matrices of the form
\[
\bpm \fp^{-N} & \fp^{-2N} & \fp^{-3N} & \cdots \\
\fp^{-2N} & \fp^{-3N} & \fp^{-4N} & \dots \\
\fp^{-3N} & \fp^{-4N}& \fp^{-5N}& \cdots\\
\vdots & \vdots & \vdots &\ddots \\
\epm
\]
We can characterize this set as 
\[
X(N)=\{x=(x_{i,j})\in Mat_n(F) \mid x_{i,j}\in \fp^{-(i+j-1)N}\}.
\]
 Note that if we consider the large compact open subgroup $U(N)$ of $U$ defined by 
 \[
 U(N)=\{ u=(u_{i,j}) \in U \mid u_{i,j}\in \fp^{-N}\text{ for } i,j>1\},
 \]
  then $X(N)$ is stable under the action of $U(N)$ on the right and $^tU(N)$ on the left.  Hence we have the following.
  
\begin{lemma}\label{phiinv} If $\vphi=\vphi_N$ then  $\vphi({^tu}gu)=\vphi(g)$ for all $u\in U(N)$. 
\end{lemma}

  In the arguments that follow, we will need the invariance of Lemma \ref{phiinv} to hold for increasingly large compact open subgroups of the type $U(N)$. We will use the phrase ``for sufficiently large $\vphi$'' to mean ``for $\vphi=\vphi_N$ for sufficiently large $N$''.

Now, a simple change of variables gives the following proposition.

\begin{prop} Let $\vphi=\vphi_N$. Then $B^G_\vphi(u_1gu_2,f)=\psi(u_1)B^G_\vphi(g,f)\psi(u_2)$ for all $u_1\in U$ and $u_2\in U(N)$.
\end{prop}

If we now take $g=\dw_\ell a$ with $a\in  A_{w_\ell}=A$ then  $U_{\dw_\ell a}=U_{\dw_\ell}=\{e\}$.  Then 
\[
B^G_\vphi(\dw_\ell a,f)=\int_U W^f(\dw_\ell a u)\vphi({^tu}a'u)\psi^{-1}(u)\ du.
\]
This then gives our partial Bessel function in \eqref{pbf} as one of our family of partial Bessel integrals.

\begin{prop} Let $f\in \cM(\pi)$ such that $W^f(e)=1$ and let $\vphi=\vphi_{\kappa}$. Then
\[
j_{\pi,\dw_\ell,\kappa}(\dw_\ell a)=B^G_\vphi(\dw_\ell a,f).
\]
\end{prop}

\section{Analytic stability  for supercuspidal
representations II: analysis of partial Bessel integrals }

\subsection{Bessel integrals and orbital integrals}\label{Bfio}

We will now introduce a class of orbital integrals as defined in
\cite{JY96, J12}. For $U\times U$
acting on the right of $G$ by $g\cdot (u_1, u_2)={}^{t}u_1 g u_2$,
we define the stabilizer $U^{g}$ of $g$ in $U\times U$
by the equation $^{t}u_1 g u_2=g$. Then for any function $\phi\in
C_c^{\infty}(G)$, we define the orbital integral $I(g, \phi)$ by
\beq\label{oi1}
I(g, \phi)=\int_{U^{g} \backslash U\times U} \phi(^{t}u_1 g u_2) \psi^{-1}(u_1)\psi^{-1}(u_2) du_1 du_2.
\eeq
For this to be well defined we must have that $\psi(u_1u_2)=1$ if ${^tu_1}gu_2=g$ and following Jacquet we call these $g$ {\it relevant}.

By the Bruhat decomposition,  the elements of the form $\dw a$ with $w\in W$, $a\in
A$ form a set of representatives for the orbits of
$U\times U$ on $G$. 
Following the terminology in \cite{J12}, we say an element $w\in W$ is relevant if $w^2=e$, and,  for all $\alpha \in \Delta$ with
$w(\alpha)<0$ we have  $w(\alpha) \in -\Delta$. Let $R(G)$  denote
the set of relevant elements in $W$. 
From \cite{J12} we know  that $w\in R(G)$ if and only if there exists a standard Levi
$M$ of $G$ such that $w=w_\ell^M$. 
Thus
\[
R(G)=\{w\in W \mid w=w_\ell^{M} \mbox{ for some standard Levi } M\subset G\}.
\]

Note that from their definitions, and since $w_\ell^2=e$, we have that $R(G)=w_\ell B(G)$ and $B(G)=w_\ell R(G)$, so $w\in B(G)$  iff  $\tilde{w}=w_\ell w\in R(G)$. Note that with our choice of representatives we find that if $w=w_\ell w^M_\ell\in B(G)$, then since $\ell(w_\ell)=\ell(w)+\ell(w^M_\ell)$ we have that if we concatinate minimal expressions for $w$ and $w^M_\ell$ we get a minimal expression for 
$w_\ell$ so that $\dw_\ell=\dw\dw^M_\ell$ or $\dw=\dw_\ell(\dw^M_\ell)^{-1}=\dw^G_M$. (See Steinberg \cite{St}, page 262.)

By \cite{J12}, for $\phi\in C_c^{\infty}(G)$ and $w\in
W$, and $ a\in A$, we have the orbital integral $I(\dw a, \phi)$ in \eqref{oi1} is
non-vanishing if and only if $w$ is  relevant and $a$ is relevant for $w$,    that is to
say, $\dw a$ supports an orbital integral if and only if $w=w^M_\ell \in R(G)$
and $a\in Z_M$.   

Now we take $f\in \cM(\pi)$ 
to be a matrix coefficient of $\pi$.   Then we may
similarly define
\[
I(\dw a, f)=\int_{U^{\dw{a}} \backslash U\times U} f(^{t}u_1 \dw a u_2)
\psi^{-1}(u_1)\psi^{-1}(u_2) du_1 du_2.
\]

If we take $\phi_f\in C^\infty_c(M)$ which projects to $f$ by \eqref{mcproj}
we have
\[
\begin{aligned}
I(\dw a, f)&=\int_{U^{\dw a} \backslash U\times U}
\left( \int_{Z} \phi_f(^{t}u_1 \dw za u_2) \omega_{\pi}^{-1}(z) dz \right)
\psi^{-1}(u_1)\psi^{-1}(u_2) du_1 du_2\\
&=\int_{Z} \left( \int_{U^{\dw a} \backslash U\times U} \phi_{f}(^{t}u_1 \dw za u_2)
\psi^{-1}(u_1)\psi^{-1}(u_2) du_1 du_2\right) \omega_{\pi}^{-1}(z) dz\\
&=\int_{Z} I( \dw za, \phi_{f}) \omega_{\pi}^{-1}(z) dz.
\end{aligned}
\]
Since
$\phi_{f}\in C_c^{\infty}(G)$, there is no convergence issue for
$I(\dw a, f)$, and hence we may switch the order of the integrations.

\begin{prop}\label{bfio}
If we take $f\in \cM(\pi)$ with $W^f(e)=1$, then for $w\in B(G)$, $w=w_\ell w^M_\ell\in R(G)$, and $\tilde f(g)=f(\dw_\ell g)$ we have
\[
B^G(\dw a,f)=I((\dw^M_\ell)^{-1}a,\tilde{f})
\]
for all $a\in A_{{w}}=Z_M$.
\end{prop}

\begin{proof} We begin with the expression for $B^G(\dw a,f)$ in \eqref{bf1}
\[
B^G(\dw a,f)=\int_{U^{-}_{ w}}
\int_{U} f(u_1\dw au_{2}) \psi^{-1}(u_1) du_1 \psi^{-1}(u_2) du_2.
\]
If we make the change of variables $u_1=\dw_\ell{^tu}\dw_\ell^{-1}$, 
then by the compatibility of our character and choice of Weyl group representatives $\psi(u_1)=\psi(u)$, and we can rewrite our integral as
\[
B^G(\dw a, f)=\int_{U \times U_{w}^-}\tilde{f} ({^tu}(\dw_\ell^M)^{-1}au_2)\psi^{-1}(u)\psi^{-1}(u_2)\ du du_2.
\]
Let $\tilde{w}=(\dw_\ell^M)^{-1}$. 

The orbital integral is given by
\[
I(\tilde{w}a,\tilde{f})=\int_{U^{\tilde{w}a}\backslash U\times U}\tilde{f} ({^tu}\tilde{w}au_2)\psi^{-1}(u)\psi^{-1}(u_2)\  dudu_2 .
\]
Since $a\in A_{w}=Z_M$, $U^{\tilde{w}a}=U^{\tilde{w}}=\{(u,u_2)\in U\times U \mid {^tu}\tilde{w}u_2=\tilde{w}\}$. Using $\tilde{w}=\dw_\ell^{-1}\dw$ we see $(u,u_2)\in U^{\tilde{w}}$ iff $u_1\dw u_2=\dw$ iff $u_1=\dw u_2^{-1}\dw^{-1}$. Since $u_1$ and $u_2\in U$, this is possible iff $u_2\in U_{w}^+$ and then  $(u_1,u_2)=(\dw u_2^{-1}\dw^{-1},u_2)$ or equivalently $(u,u_2)= (\tilde{w}{^tu_2^{-1}}\tilde{w}^{-1}, u_2)$. Therefore we have $U_{w}^+\simeq U^{\tilde{w}}$ through the map
$u_2\mapsto (\tilde{w}{^tu_2^{-1}}\tilde{w}^{-1}, u_2)$.  Hence $U^{\tilde{w}}\backslash U\times U\simeq U\times (U_{w}^+\backslash U)\simeq U\times U_{w}^-$.

Thus the integrals for $B^G(\dw a,f)$ and $I(\tilde{w}a, \tilde{f})$ are the same for all $a\in A_{\tilde{w}}$.
\end{proof}

If we combined Proposition \ref{bfio} with the analysis of Jacquet and Ye \cite{JY96, J12} we could develop a theory of Shalika germs for our Bessel integrals. If these Bessel integrals appeared in our integral representation for the local coefficients, the desired supercuspidal stability would quickly follow from the germ expansion. Unfortunately, the functions in the integral representations are only partial Bessel integrals. In this section we will adapt the techniques of Jacquet and Ye \cite{JY96}, and particularly Jacquet's  paper \cite{J12}, to our partial Bessel integrals. This will not yield a full theory of Shalika germs for the partial Bessel integrals, but it will allow us to establish  ``uniform smoothness'' results which will be sufficient for  stability.

\subsection{Preliminaries}\label{prelim2} Before we begin our adaptation of \cite{J12} we need a few more preliminaries. We retain the notation for $G=GL_n(F)$ as in Section \ref{prelim}. 

\subsubsection{Basic Weyl elements} For the convenience of the reader, we collect here the  various Weyl elements  that will play a role in what follows. Let $L\subset M\subset G$ be standard Levi subgroups of $G$. We let $w_\ell=w^G_\ell$ be the longest Weyl group element of $G$ and $w_\ell^M$ the longest Weyl group element of $M$. Their representatives in $G$, set in Section \ref{prelim}, are $\dw_\ell$ and $\dw_\ell^M$.  The elements of $B(G)$ are thus of the form $w=w_\ell w^M_\ell=w^G_M$ with representative $\dw^G_M=\dw_\ell(\dw_\ell^M)^{-1}$. We similarly set $w^M_L=w^M_\ell w^L_\ell\in B(M)$ with representative $\dw^M_L=\dw^M_\ell(\dw^L_\ell)^{-1}$. For convenience we also set $\tw_M=(\dw^M_\ell)^{-1}$.

\subsubsection{The basic open sets} For each $w\in W$ we let $C(w)=UwAU$ be the associated Bruhat cell, so $w\leq w'$ iff $C(w)\subset \overline{C(w')}$ defines the Bruhat order. For $w\in B(G)$ we will let $C_r(w)=UwA_wU\subset C(w)$ denote the relevant part of the Bruhat cell.
We define
  \[
  \Omega_w=\coprod_{w\leq w'}C(w').
  \]
  Since the individual Bruhat cells are invariant under the two sided action of $U\times U$, so is $\Omega_w$.
  
 The following is a simple consequence of the topology of Bruhat cells.
  
  \begin{lemma}  $\Omega_w$ is an open subset of $G$ and $C(w)$ is closed in $\Omega_w$.
  \end{lemma}

\subsubsection{Bessel distance}  If $w, w'\in B(G)$ with $w>w'$ we set (following Jacquet)
 \[
 d_B(w, w')=max\{ m\mid \text{ there exist } w_i\in B(G) \text{ with }w=w_m>w_{m-1}>\cdots >w_0=w'\}
 \]
 This counts the number of Weyl elements (or Bruhat cells) that support Bessel functions between $w$ and $w'$ (or $C(w)$ and $C(w'))$. $d_B(w,w')=1$ if $w$ and $w'$ support Bessel functions but no Weyl elements in between do. [The cases where we have proved stability in the past have involved Bessel functions for $w$ such that $d_B(w,e)=1$. Now we are dealing with $w_\ell$ which is as far away from $e$ as possible.]

\subsubsection{Twisted centralizers} We next collect some useful facts about twisted centralizers. We let $w\in B(G)$, so that  $w=w_\ell w^M_\ell=w^M_G$ for some Levi subgroup $M=M_w\subset G$.

 \begin{lemma} \label{S.3}   Let $w\in B(G)$. Then $U_{\dot{w}}\subset  U_w^+$.
 \end{lemma}
  
  \begin{pf}  Since $w=w_\ell w^M_\ell$ we have  $\dw=\dw_\ell(\dw^M_\ell)^{-1}=\dw^G_M$. Then $\dw_\ell^{-1}\dw=(\dw^M_\ell)^{-1}=\tilde{w}_M$. So $u\in U_\dw$ iff ${^tu}\tilde{w}_Mu=\tilde{w}_M$ iff $\tilde{w}_Mu\tilde{w}_M^{-1}={^tu^{-1}}$. Therefore $u\in U_{w_\ell^M}^-=U_M=U_w^+$.
  \end{pf}
  
 \begin{lemma}\label{S.4} Let $w\in B(G)$ and $a\in A$. Then $U_{\dot{w}a}\subset U_w^+$.
 \end{lemma} 
  
  \begin{pf} If $u\in U_{\dot{w}a}$ then ${^tu}\tilde{w}_Mau=\tilde{w}_Ma$ or ${^tu}=\tilde{w}_Mau^{-1}a^{-1}\tilde{w}_M^{-1}$. Now $u\in U$ implies $u^{-1}\in U$ and $au^{-1}a^{-1}\in U$.  Since ${^tu}\in U^-$ this gives $au^{-1}a^{-1}\in U_{\dw_\ell^M}^-=U_M$. Thus $u^{-1}\in a^{-1}U_Ma=U_M$ and so $u\in U_M=U_w^+$.\end{pf}
  
\begin{lemma}\label{S.5} Let $w\in B(G)$ and $a\in A_w$. Then $U_{\dot{w}a}=U_{\dot{w}}$.
\end{lemma}  
  
  \begin{pf} We have seen that  $U_{\dot{w}}\subset U_M$. Since $A_w=Z_M$, for $u\in U_{\dot{w}}$ we have $aua^{-1}=u$ for $a\in A_w$. So for $u\in U_{\dot{w}}$ and $a\in A_w$ 
  \[
  {^tu}\tilde{w}_Mau={^tu}\tilde{w}_Maua^{-1}a={^tu}\tilde{w}_Mua=\tilde{w}_Ma
  \]
  so $U_{\dot{w}}\subset U_{\dot{w}a}$. 
  
  Similarly,  since $U_{\dot{w}a}\subset U_M$, if $u\in U_{\dot{w}a}$ then ${^tu}\tilde{w}_Mau=\tilde{w}_Ma$ so that ${^tu}\tilde{w}_Maua^{-1}=\tilde{w}_M$. But $u\in U_M$ and $a\in Z_M$ implies $aua^{-1}=u$. So ${^tu}\tilde{w}_Mu=\tilde{w}_M$ and $u\in U_{\dot{w}}$.
  \end{pf}

Note that if the blocks in $M$ are of size $3$ or larger then $U_{\dw^G_M}$ will be strictly smaller than $U_{w^G_M}^+$.

\subsubsection{$B(M)$} Let  $M\subset G$ be a Levi subgroup. Then there is a partition $(n_1,\dots, n_t)$ of $n$ so that $M\simeq GL_{n_1}(F)\times\cdots\times GL_{n_t}(F)$ embeds in $G$ as a  block diagonal subgroup.

We define $B(M)$ in the same way we defined $B(G)$, that is
\[
w\in  B(M) \subset W_M \text{ iff for all } \alpha\in \Delta_M, w\alpha>0 \text{ implies }w\alpha \in \Delta_M.
\]
Then $B(M)\simeq B(GL_{n_1})\times\cdots\times B(GL_{n_t})$ in the block diagonal representation.
Once again we have that $w\in B(M)$ iff there exists a levi subgroup $L=L_w\subset M$ such that 
$w=w^M_\ell w_\ell^L$
where $w_\ell^L$ is the long Weyl element of the Weyl group of $L$. 

Given $w\in B(M)$ we set 
\[
A_w=\{ a\in A_M \mid \alpha(a)=1 \text{ for all simple } \alpha \in \Delta_M \text { such that } w\alpha>0 \}.
\]
If $w=w^M_\ell w^L_\ell$ then $A_w=Z_L$ is the center of $L$.

Let $R(G)$ be the relevant Weyl elements of Jacquet. Then 
$w\in R(G)$ iff  $w=w_\ell^M$
for some Levi subgroup $M\subset G$. Similarly, 
$w\in R(M)$ iff $w=w_\ell^L$
for some Levi subgroup $L\subset M$.  Since $M$ is a Levi of $G$, $L$ will also be a Levi subgroup of $G$, so
$R(M)\subset R(G)$.

We have 
$B(G)=w_\ell R(G)$, or $ R(G)=w_\ell^{-1}  B(G)$,
 and similarly
$B(M)=w_\ell^M R(M)$, or  $R(M)=(w_\ell^M)^{-1}B(M)$,
so 
$w_\ell (w_\ell^M)^{-1}B(M)=w_\ell R(M)\subset w_\ell R(G)=B(G)$.

Given $L\subset M \subset G$ let $w^G_L=w_\ell w^L_\ell\in B(G)$ and $w^M_L=w^M_\ell w^L_\ell\in B(M)$ be the associated Weyl group elements that support Bessel functions on $G$ and $M$, respectively. Note that we have $w^G_L=w^G_M w^M_L$ and  
\[
A_{w^G_L}=A_{w^M_L}=Z_L.
\]

The results of the previous sections transfer mutatis mutandis to $M$.

\subsubsection{Parametrizing tori} For each $i=0,\dots, n-1$, set
\[
H_i=\left\{h_i(t)=\bpm I_i \\ & tI_{n-i}\epm\big| t\in F^\times \right\}\simeq F^\times.
\]
Then $H_0=Z$ is the center of $G$ and $\prod_0^{n-1} H_i=A$ gives another splitting of the maximal (diagonal) torus of $G$.  We set $A'=\prod_{i=1}^{n-1} H_i$ so that $A=H_0 A'=ZA'$.

If $\alpha_i$ is the $i^{th}$ simple root of $A$ in $G$, $i=1,\dots, n-1$, then
\[
\alpha_i(h_j(t))=\begin{cases} t^{-1} & i=j \\ 1 & i\neq j\end{cases}
\]
for all $j=0,\dots, n-1$. It will be convenient to index the $H_i$ for $i=1,\dots,n-1$ by the corresponding simple root as well, so $H_{\alpha_i}=H_i$.

Now let $M\subset G$ be a Levi subgroup. So $M\simeq GL_{n_1}\times \cdots\times GL_{n_t}$ for some partition $(n_1,\dots,n_t)$ of $n$, realized as block diagonal matrices in $G$
\[
M=\bpm GL_{n_1}(F)\\ &\ddots \\ & & GL_{n_t}(F)\epm.
\]
 Let $\Delta_M$ denote the simple roots occurring in $M$. Then the center of $M$ is parametrized as 
\[
Z_M=H_0 \prod_{\alpha_i\notin \Delta_M} H_{\alpha_i}
\]
and we set $Z'_M=\prod_{\alpha_i\notin \Delta_M} H_{\alpha_i}$. Let us set $T'_M$ to be a complement in $A$ to $Z_M$ given by
\[
T'_M=\prod_{\alpha_i\in \Delta_M} H_{\alpha_i}
\]
so $A=T'_MZ_M=T'_MZ'_MZ$.

\subsubsection{Transverse tori}  For $M$ any Levi subgroup of $G$ we let $M^d$ be the derived group of $M$, so $M^d\simeq SL_{n_1}\times\cdots\times SL_{n_t}$. Let $w',w\in B(G)$ with associated Levi subgroups $M_{w'}=M'$ and $M_w=M$. Suppose  $w'\leq w$.  Then $M_w\subset M_{w'}$ and  $A_{w'}\subset A_w$.   Let $A_w^{w'}=A_w\cap M_{w'}^d=Z_M\cap (M')^d\subset A_w$. Note that $A_w^w=Z_M\cap M^d$ is finite, consisting of appropriate roots of unity on the blocks of $M$. Similarly  $A_w^{w'}\cap A_{w'}=A_{w'}^{w'}$ is finite and the subgroup $A_w^{w'}A_{w'}\subset A_w$ is open and of finite index.  This decomposition essentially decomposes the relevant torus $A_w$ for $w$ into the relevant torus for the smaller cell $A_{w'}$ and a transverse torus $A_w^{w'}$. In the germ analysis of \cite{J12}, the germ functions live along the transverse tori $A_w^{w'}$.
 
 Note that our notion of transverse tori  is independent of $G$  in the sense that if $w,w'\in B(M)$ for some Levi subgroup $M$ of $G$ with $w=w_\ell^M w_\ell^L$ and $w'=w_\ell^M w_\ell^{L'}$ and $w'\leq w$, so that  $L\subset L' \subset M\subset G$, then $A^{w'}_{w}=A_{w}\cap {L'}^d=Z_L\cap (L')^d$, which is the same as if we took $w=w_\ell w^L_\ell$ and $w'=w_\ell w^{L'}_\ell$ above.

\subsection{Basic properties of  partial Bessel integrals}

  Now let $g=\dot{w}a$ with $w\in B(G)$ and $a\in A_w$. Let $M=M_w$ be the Levi subgroup of $G$ such tht $w=w_\ell w^M_\ell$. Then we have $U_{\dot{w}a}=U_{\dot{w}}\subset U_w^+=U_{M}$. Write $U=U_w^+U_w^-$ which is the same as $U=U_MN_M$. Since $U_{\dot{w}}\subset U_w^+$ we have $U_{\dot{w}}\bs U=(U_{\dot{w}}\bs U_w^+)U_w^-$. Note that $U_w^-=N_M$ is normal in $U$.
  
  In the integral for $B^G_\vphi(\dw a,f)$ in \eqref{pbi}, write $u=u^+u^-$. Then
  \[
  B^G_\vphi(\dot{w}a,f)= \int_{U_{\dot{w}}\bs U_w^+}\int_{U_w^-}\int_U f(x\dot{w}au^+u^-)\vphi({^tu^-}\ {^t{u^+}}\tw_Ma'u^+u^-)\psi^{-1}(x)\psi^{-1}(u^+u^-)\ dxdu^-du^+.
  \]
 Since $u^+\in U_w^+=U_M$ and $a\in A_w=Z_M$ we have $a'u^+=u^+a'$. We can then conjugate past $\dot{w}$ in the argument of $f$   to obtain
\[
\begin{aligned}
B^G_\vphi(\dw a,f)= \int_{U_\dw\bs U_w^+}\int_{U_w^-}\int_U  &f(x(\dw u^+\dw^{-1})\dw au^-) \\
&\times \vphi({^tu^-}\ {^tu^+}\tw_M a'u^+u^-)\psi^{-1}(x)\psi^{-1}(u^+u^-)\ dxdu^-du^+.\\
\end{aligned}
  \]
  Now do the change of variables $x\mapsto x(\dw u^+\dw^{-1})^{-1}$. Then $\psi^{-1}(x)$ becomes $\psi^{-1}(x)\psi(\dw u^+\dw^{-1})$. Since $w\in B(G)$ then by Proposition \ref{pcompat} we have $\psi(\dw u^+\dw^{-1})=\psi(u^+)$. Then we are left with
   \[
B^G_\vphi(\dw a,f)
= \int_{U_\dw\bs U_w^+}\int_{U_w^-}\int_U f(x\dw au^-)\vphi({^tu^-}\ {^tu^+}\tw_M au^+u^-)\psi^{-1}(x)\psi^{-1}(u^-)\ dxdu^-du^+.
  \]  
  
  We can state this as the following lemma. 
  
 \begin{lemma}\label {S.6} Let $w\in B(G)$ with $w=w_\ell w^M_\ell$ and $a\in A_w$. Then we can write
  \[
  B^G_\vphi(\dw a,f)=\int_{U_\dw\bs U_w^+}\left[\int_{U_w^-}\int_U f(x\dw au^-)\vphi({^tu^-}\ {^tu^+}\tw_M a'u^+u^-)\psi^{-1}(x)\psi^{-1}(u^-)\ dxdu^-\right]\ du^+.
  \]
  \end{lemma}
  \medskip

  Suppose now that $f\in C_c^\infty(\Omega_w;\omega_\pi)$. Since $C(w)$ is closed in $\Omega_w$,  the support of $f$ intersected with $C(w)$ will be compact mod $Z$.  There will be open compact subgroups $U_1\subset U$ and $U_2\subset U_w^-$ such that the support of $(x,u^-)\mapsto f(x\dw au^-)$ lies in $U_1\times U_2$ independent of $a\in A_w$. Take $\vphi=\vphi_N$ with $N$ large enough depending on $f$ such that for all $g\in G$, $\vphi(^tu_2gu_2)=\vphi(g)$ for $u_2\in U_2$ as in Lemma \ref{phiinv}. Then 
$B^G_\vphi(\dw a,f)$ is really an integral over $U_1\times U_2$ and for $u^-\in U_2$ we have 
\[
\vphi({^tu^-}( {^tu^+}\tw_M a'u^+)u^-)=\vphi({^tu^+}\tw_M au^+)
\]
is independent of $x$ and $u^-$. Then this can come out of the inner two integrals. If we now let
\[
\tilde{\vphi}^G_M( a')=\int_{U_\dw\bs U_w^+}\vphi({^tu^+}\tw_M a'u^+)\ du^+
\]
 we obtain the following lemma.

\begin{lemma}\label{S.7}   Let $w=w_\ell w^M_\ell \in B(G)$ and $f\in C_c^\infty(\Omega_w;\omega_\pi)$. Then for a suitable $\vphi$, depending on $f$ as above, we have
\[
B^G_\vphi(\dw a,f)=\tilde{\vphi}^G_M( a')\int_{U\times U_w^-}f(x\dw au^-)\psi^{-1}(x)\psi^{-1}(u^-)\ dxdu^-=\tilde{\vphi}^G_M(a')B^G(\dw a,f)
\]
for $a\in A_w=Z_M$.
\end{lemma}

We next investigate the coefficient of proportionality $\tilde{\vphi}^G_M(a')$ that occurs. 

\begin{lemma}\label{Q.8}  Let $w=w_\ell w^M_\ell\in B(G)$. Then for $a\in Z_M$, $\tilde{\vphi}^G_M (a')=0$ iff $\vphi(\tw_M a')=0$.
\end{lemma} 

\begin{proof} By definition
\[
\tilde{\vphi}^G_M( a')=\int_{U_\dw\bs U_w^+}\vphi({^tu^+}\tw_M a'u^+)\ du^+.
\]
  If we let $O_M=\{m\in M\mid {^tm}\tw_M m=\tw_M\}$ then  we have $U_w^+=U_M$ and, by the proof of Lemma \ref{S.3}, $U_\dw=O_M\cap U_M$. 
  
Consider the map $T:(O_M\cap U_M)\backslash U_M \rightarrow Mat_n(F)$ given by $u\mapsto {^tu}\tw_M u$. This is a polynomial map and hence continuous. If we then consider $\vphi=\vphi_N$ so that $\vphi_N$ is the characteristic function of $X(N)$, then $X(N)$ is open and 
\[
\widetilde{\vphi}^G_M(e)=Vol_{(O_M\cap U_M)\backslash U_M}(T^{-1}(X(N)))\neq 0.
\] 

The effect of multiplying by $a'\in A'_w=Z'_M$, where we write $a'=\diag (I_{n_1}, a_2I_{n_2},\dots, a_t I_{n_t})$, is to scale the entries of $T(u)$ by an appropriate $a_i$. If we let $\vphi_{a'}(x)=\vphi(a'x)$ for $x\in Mat_n(F)$ then $\vphi_{a'}$ is the characteristic function of 
\[
X_{a'}(N)=\{x\in Mat_n(F)\mid a'x\in X(N)\}.
\]
Since $X_{a'}(N)$ is still open, we have $T^{-1}(X_{a'}(N))$ is open in $(O_M\cap U_M)\backslash U_M$. Finally,   $T(u)$  is of block form 
\[
T(u)=\bpm T_1(u_1)\\ &\ddots \\& & T_t(u_t)\epm \quad \text{with}\quad T_i(u_i)=\bpm 0 & & 1\\& \sddots & *\\ \pm 1 & * & *\epm
\]
so that
\[
a'T(u)=\bpm  a_1T_1(u_1)\\ &\ddots \\& & a_tT_t(u_t)\epm \quad \text{with}\quad a_iT_i(u_i)=\bpm 0 & & a_i\\& \sddots & *\\ \pm a_i & * & *\epm
\]
and $a_1=1$.
So we see that $\vphi_N(a'T(u))\neq 0$ places conditions on $a'$ coming from the bounds on the $|a_i|$ from the diagonal entries of the $a_iT_i(u_i)$.
In fact we see that this condition is equivalent to $\vphi_N(\tw_M a')\neq 0$.  Combined this gives the computation
\[
\widetilde{\vphi}^G_M( a')=\vphi(\tw_M a')Vol_{(O_M\cap U_M)\backslash U_M}(T^{-1}(X_{a'}(N))
\]
so that 
\[
\widetilde{\vphi}^G_M( a')\neq 0 \iff \vphi(\tw_M a')\neq 0.
\]
\end{proof}

This now gives the following result that will be important for what follows.

\begin{lemma}\label{S.8} Let $w=w_\ell w^M_\ell\in B(G)$.  Given $f\in C_c^\infty(\Omega_w;\omega_\pi)$ and  $\vphi=\vphi_N$ so that Lemma \ref{S.7} holds, then (enlarging $N$ if necessary) 
\[
B^G_\vphi(\dw a,f)\neq 0 \quad\iff\quad  B^G(\dw a,f)\neq 0. 
\]
\end{lemma}

\begin{proof} We first choose $\vphi=\vphi_N$ so that Lemma \ref{S.7} holds, i.e., 
\[
B^G_\vphi(\dw a,f)=\tilde{\vphi}^G_M(a')B^G(\dw a,f) 
\]
for $a\in A_w$. Then by Lemma \ref{Q.8} we have 
$\widetilde{\vphi}^G_M( a')\neq 0$  iff $ \vphi(\tw_M  a')\neq 0$.

Since $f\in C_c^\infty(\Omega_w;\omega_\pi)$ and 
\[
C(w)=U AU_w^-=U\dw ZA'U_w^-\simeq U\times Z \times A' \times U_w^-
\]
 is closed in $\Omega_w$, we see that there are compact sets $U_1\subset U$, $U_2\subset U_w^-$ and $K'\subset A'$ such that $f(x\dw au)\neq 0$ implies $x\in U_1$, $u\in U_2$ and $a=za'$ with $z\in Z$ and $a'\in K'$. 
 
By the proof of Lemma  \ref{Q.8} we know
$\vphi(\tw_M a') \neq 0$ if $a'$ satisfies a system of inequalities of the form $|a_i'|\leq q^{N_i}$ depending on  $\vphi$.
Since $a'\in K'$, the absolute values entries  $|a'_i|$ are bounded above and below,  and for $N$ sufficiently  large we will have $|a'_i|<q^{N_i}$. Then we will have $\vphi(\tw_M a')\neq 0$ for all $a'\in K'$ and thus $\tilde{\vphi}^G_M(a')\neq 0$. Thus $B^G(\dw a,f)\neq 0$ implies 
\[
B^G_\vphi(\dw a,f)=\tilde{\vphi}^G_M( a')B^G(\dw a,f)\neq 0.
\]
Since the other implication is elementary, we are done.
\end{proof}

\subsection{Partial Bessel Integrals for $M$} 

In what follows we will also need partial Bessel integrals on Levi subgroups $M\subset G$.  We let $C_c^\infty(M;\omega_\pi)$ be the smooth functions of compact support on $M$ which satisfy $h(zm)=\omega_\pi(z)h(m)$ for $z\in Z=Z_G$; note that this is a transformation under the center of $G$, not $M$.  For $m\in M$ the twisted centralizer in $U_M$ is 
\[
U_{M,m}=\{ u\in U_M\mid {^tu} \tw_M mu=\tw_M m\}. 
\]
 The partial Bessel integral on $M$ is then
\[
B_\vphi^M(m,h)=\int_{U_{M,m}\backslash U_M} \int_{U_M} h(xmu)\vphi({^tu}\tilde{w}_Mm'u)\psi^{-1}(xu)\ dxdu
\]
where $m'$ is  obtained from $m$ by ``stripping off the center $Z$ of $G$'' as in Section \ref{Bfbi},

We wish to compare these integrals  with the Bessel integrals on $G$.  In what follows we let $w'=w_\ell w^M_\ell=w^G_M\in B(G)$.   We begin with the following Lemma. 

\begin{lemma}\label{GM1}   $f\in C_c^\infty(\Omega_{w'}; \omega_\pi)$. Set
\[
h(m)=h_f(m)=\int_{U_{w'}^-}\int_{U_{(w')^{-1}}^-} f(x^-\dw' mu^-) \ dx^- du^-. 
\]
Then $h\in C_c^\infty(M;\omega_\pi)$ and every such $h$ can be obtained this way.
\end{lemma}

\begin{proof} We have the decomposition $
\Omega_{w'}=U_{(w')^{-1}}^- \times \dw'M\times U_{w'}^-$. 
  Since $f$ is compactly supported on $\Omega_{w'}$ mod center, then there are compact subsets $U_1\subset  U_{(w')^{-1}}^-$, $U_2\subset U_{w'}^-$ such that $f(x^-\dw' mu^-)\neq 0$ implies $x^-\in U_1$ and $u^-\in U_2$. Thus the integral converges. Since $f$ is compactly supported mod $Z$ then there is also a compact set $K\subset M$ such that $f(x^-\dw mu^-)\neq 0$ implies $m\in ZK$. Hence $h(m)\neq 0$ only if $m\in ZK$, i.e, $h$ is compactly supported mod $Z$. The transformation under $Z$ is preserved. Hence $h\in  C_c^\infty(M;\omega_\pi)$.
  The surjectivity follows as in Jacquet \cite{J12}. 
 \end{proof}

The main result of this section is the following relation between the Bessel integrals for $f$ and $h$ which are related in this way.

\begin{prop}\label{GM2}  Let $f\in C_c^\infty(\Omega_{w'};\omega_\pi)$ and let $h=h_f\in C_c^\infty(M;\omega_\pi)$. Then for all $\vphi=\vphi_N$ with $N$ sufficiently large and for every Levi $L\subset M\subset G$  we have
\[
B^G_\vphi(\dw^G_L a,f)=B^M_\vphi(\dw^M_L a, h)
\]
for all $a\in Z_L$.
\end{prop}

We first need to compare the twisted centralizers that appear in the two Bessel integrals.

\begin{lemma} Suppose we have a chain of Levi subgroups $L\subset M\subset G$ with associated Weyl elements $w^G_L\in B(G)$ and $w^M_L\in B(M)$. Then the twisted centralizers agree, i.e. $U_{M,\dw^M_L a}=U_{\dw^G_L a}$ for all $a\in Z_L$.
\end{lemma}

\begin{proof} 
From Lemma \ref{S.5} we know that for $a\in A_{w^G_L}$, $U_{\dw^G_L a}=U_{\dw^G_L}\subset U_{w^G_L}^+=U_L$.  So
\[
U_{\dw^G_L a}=\{u\in U_L \mid  {^tu}\tw_Lu=\tw_L\}.
\]
Since $U_L\subset U_M$, the same calculation will give
\[
U_{M,\dw^M_L a}=\{u\in U_L \mid  {^tu}\tw_Lu=\tw_L\}.
\]
Hence $U_{M,\dw^M_L a}=U_{\dw^G_L a}$ for all $a\in Z_L$.
\end{proof}

\begin{proof} [Proof of the Proposition] By definition, for $a\in Z_L$,
\[
B^G_{\vphi}(\dw^G_La,f)=\int_{U_{\dw^G_L}\backslash U}\int_U f(x\dw^G_L au)\varphi({^tu}\tw_L a'u)\psi^{-1}(xu)\ dxdu
\]
and
\[
B^M_{\vphi}(\dw^M_L a,h)=\int_{U_{\dw^M_L}\backslash U_M}\int_{U_M} h(x'\dw^M_L au')\varphi({^tu'}\tw_L a'u')\psi^{-1}(x'u')\ dx'du'.
\]
Note that we have $U_{\dw^G_L}=U_{\dw^M_L}\subset U_L\subset U_M\subset U$.  

In $B^G_\vphi(\dw^G_L a,f)$, let us decompose the $dx$ integration as $ x=x^-x^+\in U=U_{(w')^{-1}}^-U_{(w')^{-1}}^+$
where $w'=w^G_M$
and the $du$ integration as $u=u^+u^-\in U=U_{w'}^+U_{w'}^-$.
Recall that $U_{w'}^+=U_M$ and $U_{w'}^-=N_M$.
Further write $\dw^G_L=\dw^G_M\dw^M_L=\dw' \dw^M_L$. Then
\[
f(x\dw^G_L a u)=f(x^-x^+\dw'\dw^M_L au^+u^-)=f(x^-\dw'(x'\dw^M_L au')u^-)
\]
with $x'\in U'=U_{M}$ and $u'=u^+\in U'=U_{M}$.

 Decomposing $U_{\dw^G_L}\backslash U=U_{\dw^M_L}\backslash U$ as $(U_{\dw^M_L}\backslash U_M)N_M=(U_{\dw^M_L}\backslash U_M)U_{w'}^-$, we can now write
\[
\begin{aligned}
B^G_{\vphi}&(\dw^G_L  a,f)=\\
&\int_{U_{\dw^M_L}\backslash U_M\times U_M}\left[ \int_{U_{(w')^{-1}}^-\times U_{w'}^-} f(x^-\dw'(x'\dw^M_L au')u^-)\vphi({^tu^-}({^tu'}\tw_L a'u')u^-)\psi^{-1}(x^-u^-)\ dx^-du^-\right]\\
&\quad\quad\quad \psi^{-1}(x'u')\ dx^-du^-.
\end{aligned}
\]

As we noted above, we have the decomposition $\Omega_{w'}=U_{(w')^{-1}}^- \times \dw'M\times U_{w'}^-$.   Since $f$ is compactly supported on $\Omega_{w'}$ mod center, then there are compact subsets $U_1\subset  U_{(w')^{-1}}^-$, $U_2\subset U_{w'}^-$ such that $f(x^-\dw' mu^-)\neq 0$ implies $x^-\in U_1$ and $u^-\in U_2$.
 If we then increase $N$ so that  $\vphi=\vphi_N$ is invariant under sufficiently large  open compact subgroups of  $U_{w'}^-$, then we will have
\[
\vphi({^tu^-}({^tu'}^{M}\tw_L  a'u')u^-)=\vphi({^tu'}\tw_L a'u').
\]
Then
\[
\begin{aligned}
B^G_{\vphi}&(\dw_\ell  a,f)=\\
&\int_{U_{\dw^M_L}\backslash U'\times U'}\left[ \int_{U_{(w')^{-1}}^-\times U_{w'}^-} f(x^-\dw'(x'\dw^M_Lau')u^-)\vphi({^tu^-}({^tu'}\tw_L a'u')u^-)\psi^{-1}(x^-u^-)\ dx^-du^-\right]\\
&\quad\quad\quad \psi^{-1}(x'u')\ dx'du'\\
&
=\int_{U_{\dw^M_L}\backslash U'\times U'}\left[ \int_{U_{(w')^{-1}}^-\times U_{w'}^-} f(x^-w'(x'\dw^M_L au')u^-)\psi^{-1}(x^-u^-) dx^-du^-\right]\vphi({^tu'}\tw_L a'u')\\&\quad\quad\quad \psi^{-1}(x'u')\ dx'du'\\
&=\int_{U_{\dw^M_L}\backslash U'\times U'}  h(x'\dw^M_L au')\vphi({^tu'}\tw_L a'u')\psi^{-1}(x'u')\ dx'du'\\
&=B^{M}_{\vphi}(\dw_\ell^{M}a,h').
\end{aligned}
\]
which indeed establishes the desired equality.
\end{proof}

\subsection{Removing non-relevant cells}

 We return to consideration of Bessel integrals on $G$.  Whether analyzing the asymptotics of Bessel functions to establish the stability of $\gamma$-factors or analyzing Shalika germs for orbital integrals, one proceeds Bruhat cell by Bruhat cell. We expect non-zero contributions only from the relevant parts of the Bruhat cells that support Bessel functions. Other cells should contribute nothing. We refer to this as ``removing non-relevant cells''. In this section we present a number of lemmas analyzing the contributions of non-relevant cells. 
 
\subsubsection{Basic Lemma}

We begin with our basic lemma:

\medskip

\begin{lemma}[Basic Lemma]\label{BL} Let $f\in C_c^\infty(G;\omega_\pi)$. Let $U_1$ and $U_2$ be compact open subsets of $U$.  Set
\[
f'(g)=\frac{1}{Vol(U_1\times U_2)}\int_{U_1\times U_2}f(u_1gu_2)\psi^{-1}(u_1)\psi^{-1}(u_2)\ du_1du_2.
\]
Then for for appropriate $\vphi$, depending on $U_2$, we have
\[
B^G_\vphi(g,f)=B^G_\vphi(g,f')
\]
for all $g\in G$
\end{lemma}

\medskip

\begin{pf} By definition
\[
\begin{aligned}
B^G_\vphi(g,f')&=\int_{U_g\backslash U}\int_U f'(xgu)\vphi({^tu}\dw_\ell^{-1}g'u)\psi^{-1}(x)\psi^{-1}(u)\ dxdu\\
&= \frac{1}{Vol(U_1\times U_2)}\int_{U_g\backslash U}\int_U\left[\int_{U_1\times U_2}f(u_1xguu_2)\psi^{-1}(u_1)\psi^{-1}(u_2)\ du_1du_2\right]\\
&\quad\quad \times \vphi({^tu}\dw_\ell^{-1}g'u)\psi^{-1}(x)\psi^{-1}(u)\ dxdu\\
\end{aligned}
\]
We interchange integrations,  justifiable since $U_1$ and $U_2$ are compact, to have
\[
\begin{aligned}
B^G_\vphi(g,f')&= \frac{1}{Vol(U_1\times U_2)}\int_{U_1\times U_2}\left[\int_{U_g\backslash U}\int_U f(u_1xguu_2)
\vphi({^tu}\dw_\ell^{-1}g'u)\psi^{-1}(x)\psi^{-1}(u)\ dxdu\right]\\
&\quad\quad \times \psi^{-1}(u_1)\psi^{-1}(u_2)\ du_1du_2.\\
\end{aligned}
\]
Now make the change of variables $x\mapsto u_1^{-1}x$ and $u\mapsto uu_2^{-1}$. Note that this last change of variables just permutes the cosets of the domain of integration $U_g\backslash U$.  Then $\psi^{-1}(x)$ becomes $\psi(u_1)\psi^{-1}(x)$ and $\psi^{-1}(u)$ becomes $\psi^{-1}(u)\psi(u_2)$. The characters on the $u_i$ then cancel and we are left with
\[
\begin{aligned}
B^G_\vphi(g,f')&= \frac{1}{Vol(U_1\times U_2)}\int_{U_1\times U_2}\left[\int_{U_g\backslash U}\int_U f(xgu)
\vphi({^tu}_2^{-1}{^tu}\dw_\ell^{-1}g'uu_2^{-1})\psi^{-1}(x)\psi^{-1}(u)\ dxdu\right]\\
&\quad\quad \times \ du_1du_2.\\
\end{aligned}
\]
Since $U_2$ is compact, by increasing the support of  $\varphi$ if necessary, we can assume that $\varphi({^tu}_2\dw_\ell^{-1}g'u_2)=\varphi(\dw_\ell^{-1}g')$ for all $u_2\in U_2$ and $g\in G$ by Lemma \ref{phiinv}.  In this case, the integrand is independent of $u_1$ and $u_2$, so that
\[
\begin{aligned}
B^G_\vphi(g,f')&= \frac{1}{Vol(U_1\times U_2)}\int_{U_1\times U_2}\left[\int_{U_g\backslash U}\int_U f(xgu)
\vphi({^tu}\dw_\ell^{-1}g'u)\psi^{-1}(x)\psi^{-1}(u)\ dxdu\right] \ du_1du_2\\
&=\int_{U_g\backslash U}\int_U f(xgu) \vphi({^tu}\dw_\ell^{-1}g'u)\psi^{-1}(x)\psi^{-1}(u)\ dxdu\\
&=B^G_\varphi(g,f).
\end{aligned}
\]
\end{pf}

\noindent{\bf Remark}. Since our sets $\Omega_w$, for $w\in B(G)$, are open, we have $C_c^\infty(\Omega_w;\omega_\pi)\subset C_c^\infty(G;\omega_\pi)$, so the Basic Lemma \ref{BL} holds in this context as well.

\subsubsection{Relevant torus to full torus}

 We fix an element $w=w_\ell w^M_\ell\in B(G)$, so a Weyl element that supports a Bessel function. We have 
\[
\Omega_w=\coprod_{w\leq w'} C(w')
\]
an open set in $G$. Then $C(w)=UwAU$ is closed in $\Omega_w$. Since any two choices of representatives of $w$ differ by an element of $A$, $C(w)$ is independent of the choice of representative. Let $C_r(\dw)=U\dw A_wU$ be the relevant part of the cell $C(w)$. Since two choices of representatives for $w$ may not differ by an element of $A_w$, this now depends on a choice of representative. $C_r(\dw)$ is closed in $C(w)$, being defined by the closed conditions of $\alpha_i(a)=1$ for certain simple roots $\alpha_i$, and hence in $\Omega_w$. Let $\Omega'_\dw=\Omega_w-C_r(\dw)$, the complement in $\Omega_w$ of the relevant part of the cell.  The following is the analogue of Jacquet's Lemma 2.2 in \cite{J12} for our partial Bessel integrals.

\begin{lemma} \label{G.2}  Let $f\in C_c^\infty(\Omega_w;\omega_\pi)$. Suppose $B^G_\vphi(\dw a,f)=0$ for all $a\in A_w$. Then  there exists $f_0\in C_c^\infty(\Omega'_\dw;\omega_\pi)$ such that, for all sufficiently large  $\vphi$ depending only on $f$, we have $B^G_\vphi(g,f)=B^G_\vphi(g,f_0)$ for all $g\in G$.
\end{lemma}

\begin{proof} 
We have
\[
0\longrightarrow C_c^\infty(\Omega'_\dw;\omega_\pi)\longrightarrow C_c^\infty(\Omega_w;\omega_\pi)\longrightarrow C_c^\infty(C_r(\dw);\omega_\pi)\longrightarrow 0.
\]

The Bessel integral is given by
\[
B^G_\vphi(\dw a,f)=\int_{U_{\dw a}\bs U}\int_U f(x\dw au)\vphi({^tu}\tw_M a'u)\psi^{-1}(x)\psi^{-1}(u)\ dxdu.
\]
Let us write $C(w)=UwAU$ as $C(w)=U\dw ZA'U_w^-$, where $U_w^-=\{u\in U\mid wuw^{-1}\in U^-\}$. In this decomposition we have uniqueness of expression, and in fact $U\times Z\times A'\times U_w^-\rightarrow C(w)$ is a homeomorphism.  Since $f\in C_c^\infty(\Omega_w;\omega_\pi)$ and $C(w)$ is closed in $\Omega_w$, there are compact subgroups $U_1\subset U$ and $U^-_2\subset U_w^-$ such that for every  $a\in A$ the map $(x,u)\mapsto f(x\dw au)$ is supported in $U_1\times U^-_2$. 

Now let us apply Lemma  \ref{S.7}. Then we have, for $a\in A_w$, 
\[
B^G_\vphi(\dw a,f)=\tilde{\vphi}^G_M(a')\int_{U\times U_w^-}f(x\dw au^-)\psi^{-1}(x)\psi^{-1}(u^-)\ dxdu^-.
\]
 Then from our hypotheses and Lemma \ref{S.8}, we can conclude that
\[
\int_{U_1\times U^-_2}f(x\dw au^-)\psi^{-1}(x)\psi^{-1}(u^-)\ dxdu^-=\int_{U\times U_w^-}f(x\dw au^-)\psi^{-1}(x)\psi^{-1}(u^-)\ dxdu^-=0.
\]

Let  $U_2^+\subset U_w^+$ be an open compact subgroup such that $\dw U_2^+\dw^{-1}\subset U_1$. Let $U_2=U_2^+U_2^-\subset U$. 
Let
\[
f_0(g)=\frac{1}{Vol(U_1\times U_2)}\int_{U_1\times U_2} f(u_1gu_2)\psi^{-1}(u_1)\psi^{-1}(u_2)\ du_1du_2.
\]
Then $f_0\in C_c^\infty(\Omega_w;\omega_\pi)$. Suppose $g=\dw a$ with $a\in A_w$. Then setting $u_2=u_2^+u_2^-$,  conjugating $u_2^+$ past $\dw a$, and doing a change of variable in $u_1$ we obtain
\[
\begin{aligned}
f_0(\dw a)&=\frac{1}{Vol(U_1\times U_2)}\int_{U^+_2} \int_{U_2^-}\int_{U_1}f(u_1\dw au^+_2u_2^-)\psi^{-1}(u_1)\psi^{-1}(u_2^+u_2^-)\ du_1du^-_2du_2^+\\
&=\frac{1}{Vol(U_1\times U_2)}\int_{U^+_2}\int_{U_2^-}\int_{U_1}f(u_1\dw au_2^-)\psi^{-1}(u_1)\psi(\dw u^+_2\dw^{-1})\psi^{-1}(u_2^+u_2^-)\ du_1du^-_2du_2^+.\\
\end{aligned}
\]
Since $w$ supports a Bessel function, then as in  Proposition \ref{pcompat}, $\psi(\dw u_2^+\dw^{-1})=\psi(u_2^+)$. Hence the characters cancel and the integrand is independent of $u_2^+$. Thus we have
\[
f_0(\dw a)=\frac{Vol(U_2^+)}{Vol(U_1\times U_2)}\int_{U_2^-}\int_{U_1}f(u_1\dw au_2^-)\psi^{-1}(u_1)\psi^{-1}(u_2^-)\ du_1du^-_2=0.
\]

As in Jacquet \cite{J12}, one can extend this to show that $f_0$ vanishes on all of $C_r(\dw)$, 
so in fact $f_0\in C_c^\infty(\Omega'_\dw$) by the above exact sequence. 
The method Jacquet uses is to assume that say $u'\in U$. Then if $f_0(u'\dw a)\neq 0$ then for some $u_1\in U_1$ and $u_2^-\in U_w^-$ we must have $f(u_1u'\dw au_u^-)\neq 0$ in the integrand of $f_0$. But from our assumption on the support of $f$ this implies that $u_1u'\in U_1$ and since $U_1$ was a subgroup of $U$, $u'\in U_1$. Hence we can perform a change of variables to obtain $f_0(u'\dw a)=\psi(u')f(a\dw)=0$, contradiction. Hence $f_0(u'\dw a)=0$. One does the same argument for any ${u^-}'\in U_w^-$ on the right.   Hence $f_0(u'\dw a{u^-}')=0$ for all $u\in U$, ${u^-}'\in U_w^-$ and $a\in A_w$.

Finally we can apply the Basic Lemma \ref{BL} to conclude  $B^G_\vphi(g,f)=B^G_\vphi(g,f_0)$. 
\end{proof}
  
  We now want to extend from $C_r(\dw)$ to $C(w)$, that is, from the relevant torus to the full torus. This is the analogue of Jacquet's Lemma 2.3 in \cite{J12}. We now consider $C(w)$ as a closed set in $\Omega_w$ and take  $\Omega_w^\circ=\Omega_w-C(w)$.

\begin{lemma}\label{G.3}  Let $f\in C_c^\infty(\Omega_w;\omega_\pi)$. Suppose $B^G_\vphi(\dw a,f)=0$ for all $a\in A_w$. Then there exists $f_0\in C_c^\infty(\Omega^\circ_w;\omega_\pi)$ such that, for all sufficiently large $\vphi$ depending only on $f$, we have $B^G_\vphi(g,f)=B^G_\vphi(g,f_0)$ for all $g\in \Omega_w$.
\end{lemma}

  \begin{pf} By Lemma \ref{G.2} we can assume that $f$ vanishes on $C_r(\dw)$, i.e., $f\in C_c^\infty(\Omega'_\dw;\omega_\pi)$. 
  
 Since $f$ is compactly supported (mod $Z$) on $\Omega_w$ and $C(w)$ is closed in $\Omega_w$, $f$ will be compactly supported on $C(w)$ mod $Z$.
   Since we can write $C(w)=U\dw ZA'U_w^-$ with uniqueness of expression, then   there are compact subgroups $U_1\subset U$, $U_2^-\subset U_w^-$ and $K'\subset A'$ such that if $f(u\dw za_1u^-)=\omega_\pi(z)f(u\dw a_1u^-)\neq 0$ then $u\in U_1$, $u^-\in U_2^-$ and $a'\in K'$. Moreover, since we can (and have) assume $f$ vanishes on $C_r(\dw)$, then there is $c>0$ such that for all $a'\in K'$ there exists at least one simple root $\alpha$ in $M$ such that $|\alpha(a')-1|\geq c$  (since $A_w=Z_M$,  the center of $M$).

Take $U_2^+\subset U_w^+$ a large enough compact open subgroup such that  the character
 \[
 u_2^+\mapsto \psi(\dw a'u_2^+(a')^{-1}\dw^{-1}(u_2^+)^{-1})
 \]
 is non-trivial on $U_2^+$ for all $a'\in K'$, so that
 \[
 \int_{U_2^+} \psi(\dw a'u_2(a')^{-1}\dw^{-1}u_2^{-1})\ du_2^+=0.
 \]
 Note that since $w$ supports a Bessel function, then every simple root space in $U_w^+$ remains simple upon conjugation by $\dw$, and since $a'\in K'$ there is a simple root of $M$ such that $|\alpha(a') -1|\geq c$. Hence the character above is non-trivial for large enough $U_2^+$.  
   Enlarge $U_2^-$ if necessary so that it is normalized by $U_2^+$ and take $U_2=U_2^+U_2^-$
 
 Let us take $U_1$ to be decomposable, that is, of the form $U_1=U_1^-U_1^+$ with $U_1^-$ a compact subgroup of $U_w^-$ and $U_1^+$ a compact subgroup of $U_w^+$ such that $U_1^+$ normalizes $U_1^-$. Enlarge  $U_1^+$ to be a large enough compact open subgroup that   $(a')^{-1}U_2^+a'\subset U_1^+$ for all $a'\in K'$. Then  enlarge $U_1^-$ if necessary to be normalized by $U_1^+$ so that $U_1=U_1^-U_1^+$ is an enlarged compact open subgroup of $U$.

 Define $f_1$ on $\Omega_w$ by
 \[
 f_1(g)=\int_{U_2}\int_{U_1}f(u_1gu_2)\psi^{-1}(u_1u_2)\ du_1du_2.
 \]
 We now claim that 
 \[
f_1(\dw a)= \int_{U_2}\int_{U_1}f(u_1\dw au_2)\psi^{-1}(u_1u_2)\ du_1du_2=0
 \]
 for all $a\in A$. If we first write $a=za'$ with $z\in Z$ and $a'\in A'$ we have $f_1(\dw a)=f_1(\dw za')=\omega_\pi(z)f_1(\dw a')$. So it is enough to consider $f_1(\dw a')$.  Decompose $U_2=U_2^+U_2^-$ as above. Then
 \[
 \begin{aligned}
f_1(\dw a')&= \int_{U_2^+}\int_{U_2^-}\int_{U_1}f(u_1\dw a'u_2^+u_2^-)\psi^{-1}(u_1u^+_2u_2^-)\ du_1du_2^-du^+\\
&=\int_{U_2^+}\int_{U_2^-}\int_{U_1}f(u_1(\dw a'u_2^+(a')^{-1}\dw^{-1})\dw a'u_2^-)\psi^{-1}(u_1u^+_2u_2^-)\ du_1du_2^-du^+.\\
\end{aligned}
\]
If $f(u_1(\dw a'u_2^+(a')^{-1}\dw^{-1})\dw a'u_2^-)\neq 0$, then $u_1(\dw a'u_2^+(a')^{-1}\dw^{-1})\in U_1$ and $a'\in K'$. Once $a'\in K'$ we have $\dw a'u_2^+(a')^{-1}\dw^{-1}\in U_1^+\subset U_1$. So we can change variables to obtain
\[
\begin{aligned}
f_1(\dw a')&=\int_{U_2^+}\int_{U_2^-}\int_{U_1}f(u_1\dw a'u_2^-)\psi^{-1}(u_1u^+_2u_2^-)\psi(\dw a'u_2^+(a')^{-1}\dw^{-1})\ du_1du_2^-du^+\\
&=\int_{U_2^+}\psi(\dw a'u_2^+(a')^{-1}\dw^{-1})\psi^{-1}(u_2^+)\ du_2^+\ \int_{U_2^-}\int_{U_1}f(u_1\dw a'u_2^-)\psi^{-1}(u_1u_2^-)\ du_1du_2^-.\\
\end{aligned}
\]
Since  $a'$ must be in  $K'$, as we have seen, then the first integral vanishes by our choice of $U_2^+$.
 
 We now need to extend this vanishing to all of $C(w)$. As above we can reduce to $a'\in A'$. 
   Suppose $f_1(u\dw a')\neq 0$ for some $u\in U$ and $a'\in A'$.  Then as above
 \[
 \begin{aligned}
 f_1(u\dw a')&=\int_{U_2^+}\int_{U_2^-}\int_{U_1}f(u_1u\dw a'u_2^+u_2^-)\psi^{-1}(u_1u^+_2u_2^-)\ du_1du_2^-du^+\\
&=\int_{U_2^+}\int_{U_2^-}\int_{U_1}f(u_1u(\dw a'u_2^+(a')^{-1}\dw^{-1})\dw a'u_2^-)\psi^{-1}(u_1u^+_2u_2^-)\ du_1du_2^-du^+.
\end{aligned}
\]
The integrand is identically zero unless $a'\in K'$ by our choice of $K'$. Then $\dw a'u_2^+(a')^{-1}\dw^{-1}\in U_1^+\subset U_1$. If $f_1(u\dw a')\neq 0$ then there must be a choice of $u_1\in U_1$ and $u_2\in U_2$ so that 
\[
f(u_1u\dw a'u_2)= f(u_1u\dw a'u_2^+u_2^-)=f(u_1u(\dw a'u_2^+(a')^{-1}\dw^{-1})\dw a'u_2^-)\neq 0.
\]
This implies that $a'\in K'$ and $u_1u(\dw a'u_2^+(a')^{-1}\dw^{-1})\in U_1$. But as $U_1$ is a subgroup and $u_1$ and  now $\dw a'u_2^+(a')^{-1}\dw^{-1}\in U_1$ then $u\in U_1$.  Then returning to the basic definition of $f_1$ we see
$f_1(u\dw a')=\psi(u)f_1(\dw a')=0$, 
a contradiction. Hence $f_1(u\dw a')=0$ for all $u\in U$ and $a'\in A'$.

Next consider $u'\in U_w^-$ and suppose there are $u\in U$ and $a'\in A'$ so that $f_1(u\dw a'u')\neq 0$.  Since $U_w^+$ normalizes $U_w^-$, we can write $f_1(u\dw a'u')$ as
\[
\begin{aligned}
f_1(u\dw a'u')&=\int_{U_2^+}\int_{U_2^-}\int_{U_1}f(u_1u\dw a'u'u_2^+u_2^-)\psi^{-1}(u_1u^+_2u_2^-)\ du_1du_2^-du^+\\
&=\int_{U_2^+}\int_{U_2^-}\int_{U_1}f(u_1u(\dw a'u_2^+(a')^{-1}\dw^{-1})\dw a'((u_2^+)^{-1}u'u_2^+)u_2^-)\psi^{-1}(u_1u^+_2u_2^-)\ du_1du_2^-du^+.
\end{aligned}
\]
For the integrand to be non-vanishing we must have $a'\in K'$ and $((u_2^+)^{-1}u'u_2^+)u_2^-\in U_2^-$. Since $U_2^-$ is a subgroup, this implies that $(u_2^+)^{-1}u'u_2^+\in U_2^-$, and since this is normalized by $U_2^+$, this forces $u'\in U_2^+$. Regardless, once $(u_2^+)^{-1}u'u_2^+\in U_2^-$, we can do a change of variables in the $U_2^-$ integral to obtain
\[
\begin{aligned}
f_1(u\dw a'u')&=\int_{U_2^+}\int_{U_2^-}\int_{U_1}f(u_1u(\dw a'u_2^+(a')^{-1}\dw^{-1})\dw a'u_2^-)\psi^{-1}(u_1u^+_2u_2^-)\psi((u_2^+)^{-1}u'u_2^+)\ du_1du_2^-du^+.\\
&=\psi(u')f_1(u\dw a)=0
\end{aligned}
\]
This is a contradiction. Hence $f_1(u\dw a'u')=0$ for all $u\in U$, $a'\in A'$, and $u'\in U_w^-$, that is, $f_1(g)=0$ for all $g\in C(w)$.

Now we let
\[
f_0(g)=\frac{1}{Vol(U_1\times U_2)}f_1(g)=\frac{1}{Vol(U_1\times U_2)}\int_{U_2}\int_{U_1}f(u_1gu_2)\psi^{-1}(u_1u_2)\ du_1du_2
\]
for $g\in \Omega_w$.
Then $f_0\in C_c^\infty(\Omega_w;\omega_\pi)$ and $f_0$ vanishes on $C(w)$, which is closed in $\Omega_w$. Hence by the usual exact sequence, $f_0\in C_c^\infty(\Omega_w^\circ;\omega_\pi)$. By our Basic Lemma \ref{BL}, 
$B^G_\vphi(g,f)=B^G_\vphi(g,f_0)$ for all $g\in \Omega_w$. Hence we are done.
  \end{pf}
  
  \subsubsection{Removing non-relevant cells}
  
  This is  the analogue of Jacquet's Lemma 2.4 in \cite{J12} for our partial Bessel integrals.

  \begin{lemma}\label{G.4} Let $w=w_\ell w^M_\ell\in B(G)$. Let $\Omega_{w,0}$ and $\Omega_{w,1}$ be $U\times U$ and $A$-invariant open sets of $\Omega_w$ such that  $\Omega_{w,0}\subset \Omega_{w,1}$ and $\Omega_{w,1}-\Omega_{w,0}$ is a union of Bruhat cells $C(w')$ such that $w'$ does not support a Bessel function, i.e., $w'\notin B(G)$. Then for any $f_1\in C_c^\infty(\Omega_{w,1};\omega_\pi)$ there exists $f_0\in  C_c^\infty(\Omega_{w,0};\omega_\pi)$ such that, for all sufficiently large $\vphi$ depending only on $f_1$, we have $B^G_\vphi(g,f_0)=B^G_\vphi(g,f_1)$ for all $g\in G$.
\end{lemma}

  \begin{pf} Since $\Omega_w$ is a finite union of Bruhat cells, we can find an increasing union
  \[
  \Omega_{w,0}=\Omega_{w_0}'\subset \Omega_{w,1}'\subset\cdots\subset\Omega_{w,r}'=\Omega_{w,1}
  \]
  such that each $\Omega_{w,i+1}'-\Omega_{w,i}'$ is a single Bruhat cell $C(w'_i)$ with $w_i'\notin B(G)$ and $C(w'_i)$ closed in $\Omega_{w,i+1}'$. So, by induction, we can reduce to proving the assertion for a single pair, that is, we can assume 
  \[
  \Omega_{w,1}-\Omega_{w,0}=C(w')
  \]
  with $w'\notin B(G)$ and $C(w')$ closed in $\Omega_{w,1}$.
  
   Since $C(w')$ is closed in $\Omega_{w,1}$ and $f_1$ has compact support mod $Z$ there, then $f_1$ has compact support mod $Z$ on $C(w')$. Since $U\times U_{w'}^-\times Z\times A'$ is homeomorphic to $C(w')$, there is a compact subset $K'\subset A'$ and compact open subgroups $U_1\subset U$ and $U_2^-\subset U_{w'}^-$ such that $f_1(u\dw'u^-za')\neq 0$ implies $u\in U_1$, $u^-\in U_{w'}^-$ and $a'\in K'$.  Let $\wt{U}_{\dw'a}$ be the stabilizer of $\dw'a$ in $U\times U$, that is, $\wt{U}_{\dw'a}=\{(u_1,u_2)\in U\times U\mid u_1\dw'au_2=\dw'a\}$.  Note that $U$ acts on $U\times U$ on the left by left multiplication in the first factor and on the right by right multiplication in the second factor. 
   
   Suppose $f_1(u_1\dw'u_2a')\neq0$ with $u_i\in U$ and $a'\in A'$.  Write $u_2=u_2^+u_2^-$ with $u_2^\pm\in U_{w'}^\pm$. Then $f_1(u_1\dw'u_2a')=f_1(u_1(\dw'u_2^+\dw'^{-1})\dw'u_2^-a')$.   If this is non-zero, then 
   \[
   u_1(\dw'u_2^+\dw'^{-1})\in U_1,\quad u_2^-\in U_2^-, \quad\text{and}\quad a'\in K'.
   \]
  If we write $u_1=u_1'(\dw'u_2^+\dw'^{-1})^{-1}$, then  $ u_1(\dw'u_2^+\dw'^{-1})=u_1'\in U_1$. Note that the pair
  $((\dw'u_2^+\dw'^{-1})^{-1},u_2^+)\in \wt{U}_{\dw'}$. So we have that $f_1(u_1\dw'u_2a')\neq 0$ implies that 
  \[
  (u_1,u_2)\in U_1\cdot\wt{U}_{\dw'}\cdot U_2^-\quad\text{and}\quad a'\in K'.
  \]
  Let us thicken $U_2^-$ to a subgroup of $U$ as follows. Let $U_2^+=\{ u^+\in U_{w'}^+\mid w'u^+w'^{-1}\in U_1\}=U_{w'}^+\cap w'^{-1}U_1w'$. This is a compact open subgroup of $U_{w'}^+$. Enlarge $U_2^-$ if necessary so that $U_2^+$ normalizes $U_2^-$ and let $U_2=U_2^+U_2^-$, a compact open subgroup of $U$. We still have that
 \[
 f_1(u_1\dw'u_2a')\neq 0\quad\text{ implies} \quad 
  (u_1,u_2)\in U_1\cdot\wt{U}_{\dw'}\cdot U_2\quad\text{and}\quad a'\in K'. 
 \]
 Then if we consider $f_1(u_1\dw'a'u_2)=f_1(u_1\dw'(a'u_2(a')^{-1})a')$, then we see that this being non-zero implies
 \[
 (u_1,(a'u_2(a')^{-1}))\in U_1\cdot \wt{U}_{\dw'}\cdot U_2\quad\text{and}\quad a'\in K'.
 \]
 Now, $(u_1,u_2)\in \wt{U}_{\dw'a'}$ iff $(u_1,(a'u_2(a')^{-1}))\in \wt{U}_{\dw'}$. So we have
  \[
 f_1(u_1\dw'a'u_2)\neq 0\quad\text{implies}\quad (u_1,u_2)\in U_1\cdot\wt{U}_{\dw'a'}\cdot (a')^{-1}U_2a'\quad\text{and}\quad a'\in K'.
\]
Since $K'$ is compact, we can enlarge $U_2$ to a compact open and decomposable $U_2'={U_2'}^+{U_2'}^-$ such that for all $a'\in K'$ we have $(a')^{-1}U_2a'\subset U_2'$. Then we have
  \[
 f_1(u_1\dw'a'u_2)\neq 0\quad\text{implies}\quad (u_1,u_2)\in U_1\cdot\wt{U}_{\dw'a'}\cdot U_2'\quad\text{and}\quad a'\in K'
\]
with both $U_1$ and $U_2'$ compact and open in $U$.

We now come to the construction of $f_0$. Since $w'$ does not support a Bessel function, there is a positive root $\alpha$  such that $w'\alpha$ is positive but not simple. If $U_\alpha$ is the unipotent root subgroup associated to $\alpha$, then necessarily $U_\alpha\subset U_{w'}^+$ and $\psi$ is non-trivial on $U_\alpha$ but not on $w'U_\alpha{w'}^{-1}$. Enlarge $U_2'$ if necessary, still keeping it decomposable,  so that $\psi$ is non-trivial on $U_\alpha\cap U_2'=U_\alpha\cap {U_2'}^-$. Then enlarge $U_1$ to $U_1'$ if necessary so that if for some $a'\in K'$ we have $(u_1,u_2)\in\wt{U}_{\dw'a'}$ and $u_2\in U_2'$ then $u_1\in U_1'$. 

Now consider $f_0'$ defined on $\Omega_{w',1}$ by
\[
f_0' (g)=\int_{U_2'}\int_{U_1'}f_1(u_1gu_2)\psi^{-1}(u_1u_2)\ du_1du_2.
\]
Consider $f'_0(\dw'a')$ for $a'\in A'$. We have
\[
f_0'(\dw'a')=\int_{U_2'}\int_{U_1'}f_1(u_1\dw'a'u_2)\psi^{-1}(u_1)\psi^{-1}(u_2)\ du_1du_2.
\]
Note that if $a'\notin K'$, the integrand vanishes identically. So we may assume $a'\in K'$. Now let $u_2'\in U_\alpha\cap U_2'$ such that $\psi(u_2')\neq 0$.  Then
\[
\begin{aligned}
\psi(u_2')f_0'(\dw'a')&=\psi(u_2')\int_{U_2'}\int_{U_1'}f_1(u_1\dw'a'u_2)\psi^{-1}(u_1)\psi^{-1}(u_2)\ du_1du_2\\
&=\int_{U_2'}\int_{U_1'}f_1(u_1\dw'a'u_2)\psi^{-1}(u_1)\psi^{-1}((u_2')^{-1}u_2)\ du_1du_2\\
&=\int_{U_2'}\int_{U_1'}f_1(u_1\dw'a'u_2'u_2)\psi^{-1}(u_1)\psi^{-1}(u_2)\ du_1du_2\\
&=\int_{U_2'}\int_{U_1'}f_1(u_1(\dw'a'u_2'(a')^{-1}\dw'^{-1})\dw'a'u_2)\psi^{-1}(u_1)\psi^{-1}(u_2)\ du_1du_2.\\
\end{aligned}
\]
Since $u_2'\in U_2'$, we have that $\dw'a'u_2'(a')^{-1}\dw'^{-1}\in U_1'$. So we can again change variables and obtain
\[
\begin{aligned}
\psi(u_2')f_0'(\dw' a')&=\int_{U_2'}\int_{U_1'}f_1(u_1(\dw' a'u_2'(a')^{-1}\dw'^{-1})\dw'a'u_2)\psi^{-1}(u_1)\psi^{-1}(u_2)\ du_1du_2\\
&=\int_{U_2'}\int_{U_1'}f_1(u_1\dw'a'u_2)\psi^{-1}(u_1)\psi(\dw'a'u_2'(a')^{-1}\dw'^{-1})\psi^{-1}(u_2)\ du_1du_2\\
&=\psi(\dw'a'u_2'(a')^{-1}\dw'^{-1})\int_{U_2'}\int_{U_1'}f_1(u_1\dw'a'u_2)\psi^{-1}(u_1)\psi^{-1}(u_2)\ du_1du_2.\\
\end{aligned}
\]
By our choice of $\alpha$, we know that since $u_2'\in U_\alpha$ that $\dw'a'u_2'(a')^{-1}\dw'^{-1}\in w'U_\alpha {w'}^{-1}$ which is a positive but non-simple root subgroup. Thus $\psi(\dw'a'u_2'(a')^{-1}\dw'^{-1})=1$. Thus
\[
\begin{aligned}
\psi(u_2')f_0'(\dw'a')&=\int_{U_2'}\int_{U_1'}f_1(u_1\dw'a'u_2)\psi^{-1}(u_1)\psi^{-1}(u_2)\ du_1du_2.\\
&=f_0'(\dw'a_1).
\end{aligned}
\]
Since $\psi(u_2')\neq 1$ we conclude that $f_0'(\dw'a')=0$, as desired.

Now consider $f_0'$ on the full cell $C(w')$. Let $g=u_1'\dw'za'u_2'$ with $u_1'\in U$, $u_2'\in U_{w'}^-$ and $a'\in A'$. Suppose $f_0'(g)=\omega_\pi(z)f_0'(u_1'\dw'a'u_2')\neq 0$. Then, from the definition of $f_0'$ there exist $u_1\in U_1'$ and $u_2\in U_2'$ such that $f_0(u_1u_1'\dw'a'u_2'u_2)\neq 0$. From our previous analysis, this implies that
\[
a'\in K'\quad\text{and}\quad (u_1u_1',u_2'u_2)\in U_1'\cdot \wt{U}_{\dw'a'}\cdot U_2'.
\]
Since $U_1'$ and $U_2'$ are groups, this last condition is equivalent to 
\[
(u_1',u_2')\in U_1'\cdot\wt{U}_{\dw'a'}\cdot U_2'.
\]
So we can write $(u_1',u_2')=(u_1^0 v_1,v_2u_2^0)$ where $(v_1,v_2)\in \wt{U}_{\dw'a'}$ and $u_i^0\in U_i'$. Thus we have
\[
f_0'(u_1'\dw'a'u_2')=f_0'(u_1^0v_1\dw'a'v_2u_2^0)=f_0'(u_1^0\dw'a'u_2^0).
\]
Since now $u_i^0\in U_i'$ we can perform a change of variables and arrive at
\[
f_0'(u_1'\dw'a'u_2')=\psi(u_1^0u_2^0)f_0'(\dw'a').
\]
But as we showed above, $f_0'(\dw'a')=0$ for all $a'\in A'$. This is a contradiction. Thus $f'_0(g)=0$ for all $g\in C(w')$.

Now let
\[
f_0(g)=Vol(U_1'\times U_2')^{-1}f_0'(g)
\]
for $g\in G$ (actually $g\in \Omega_{w,1}\subset\Omega_{w}\subset G$ since $f_1$ is only supported in $\Omega_{w,1}$).
By the above analysis, we have $f_0\in C_c^\infty(\Omega_{w,1};\omega_\pi)$ and $f_0$ vanishes on $C(w')$. Since $C(w')$ is closed in $\Omega_{w,1}$ and $\Omega_{w,1}-C(w')=\Omega_{w,0}$ we have that $f_0\in C_c^\infty(\Omega_{w,0};\omega_\pi)$.

By our Basic Lemma \ref{BL}, for appropriate choice of $\vphi$ depending only on $f_1$ through $U_2'$ , we have
$B^G_\vphi(g,f_0)=B^G_\vphi(g,f_1)$ for all $g\in G$. This finishes Lemma \ref{G.4}.
  \end{pf}

\subsection{Supecuspidal stability  for $GL_n$, I: Small cell analysis}

As noted above, to analyze the asymptotics of Bessel integrals or orbital integrals, one analyses the contributions from each relevant cell inductively, beginning with the smallest cell. In this section we analyze the small cell contributions to these asymptotics both on $G$ and, for induction purposes, on $M$. Since our functions $f$ always transform by a character of the center  $Z$ of $G$, the small cells for $G$ and $M$ will behave a bit differently and require separate analysis.

We begin with $f\in \mathcal M(\pi)\subset C_c^\infty(G;\omega_\pi)$ such that $W^f(e)=1$.  

\subsubsection{The contribution of the small cell of $G$}  Considering $e$ was a Weyl group element, we have 
$M_e=G$, $ A_e=Z_G=Z$, and  $U_e^+=U$.
Also
$\Omega_e=\coprod_{e\leq w'} C(w')=G$.
If we view $e$ as a Weyl group element then our choice of representative is simply $\dot{e}=I_n$. Since our $f$ transforms under $Z=A_e$ by the central character of $\pi$, the small cell is special and must be dealt with separately first.

\begin{prop}\label{prop1}  Let $f_0\in C_c^\infty(G\omega_\pi)$ with $W^{f_0}(e)=1$.  
For each $f\in C_c^\infty(G;\omega_\pi)$  with $W^f(e)=1$
and for each $w'\in B(G)$ with $d_B(e,w')=1$ there exists a function $f_{w'}\in C_c^\infty(\Omega_{w'};\omega_\pi)$ such that for any $w\in B(G)$  we have
\[
B^G_\vphi(\dw a,f)=\sum_{w'} B^G_\vphi(\dw a,f_{w'})+\sum_{a=bc}\omega_\pi(c)B^G_{\vphi}(\dw b,f_0)
\]
for all $a\in A_{w}$. Here $a=bc$ runs over the possible decompositions of $a$ with $b\in A^e_w$ and $c\in A_e=Z$.
\end{prop}

\begin{proof}
By definition, for $a\in A_e=Z$, we have $a'=e$ and
\[
\begin{aligned}
B^G_\vphi(\dot{e}a,f)=B^G_\vphi(a,f)
&=\omega_\pi(a)\int_{U_e\backslash U}\int_U f(xu)\vphi({^tu}\tw_Gu)\psi^{-1}(x)\psi^{-1}(u)\ dxdu\\
&=\omega_\pi(a)\int_{U_e\backslash U}\vphi({^tu}\tw_Gu)\ du \int_U f(x)\psi^{-1}(x)\ dx\\
&=\omega_\pi(a)\tilde{\vphi}^G_G(e)W^f(e)=\omega_\pi(a)\tilde{\vphi}^G_G(e).
\end{aligned}
\]

 We have defined 
\[
B^G(\dw a,f)=\int_{U_w^-}\int_U f(x\dw au^-)\psi^{-1}(x)\psi^{-1}(u^-)\ dxdu^-
\]
for $a\in A_w$ and $f\in C_c^\infty(\Omega_w;\omega_\pi)$. This is a pure Bessel integral, not a partial one.
For $w=e$ we have $U_e^-=\{e\}$ and the integral becomes 
\[
B^G(\dot{e}a,f)=\int_U f(xa)\psi^{-1}(x)\ dx=\omega_\pi(a)W^f(e)=\omega_\pi(a).
\]

We have fixed an auxiliary function $f_0\in C_c^\infty(G;\omega_\pi)$ with $W^{f_0}(e)=1$. From $f_0$ we will construct a second auxiliary  function $f_1\in C_c^\infty(G;\omega_\pi)$ having the same Bessel integrals as $f$ on the small cell following Jacquet \cite{J12}.
To define $f_1$ we need to decompose $G$ into $G^d$ and $A_e$. By definition $G^d=SL_n(F)$.
By an elementary observation, given $g\in GL_n(F)$ there are only finitely many decompositions $g=g_1c$ with $g_1\in G^d$ and $c\in Z$ and they differ by elements in $A_e^e$.  We observe that $A_e^e$ is finite, and in fact consists of the diagonal matrices whose entries are $n^{th}$-- roots of unity. There will be no such decompositions if $\det(g)\notin F^{\times,n}$.  We now set
\[
f_1(g)=\sum_{g=g_1c} f_0(g_1)B^G(\dot{e}c,f)=\sum_{g=g_1c} f_0(g_1)\omega_\pi(c).
\]
Note that  if $\det(g)$ is not a $n^{th}$-power in $F^\times$ then $f_1(g)=0$. Then $f_1\in C_c^\infty(G;\omega_\pi)$, the subgroup of elements whose determinant is a $n^{th}$-power being open in $G$. 

\begin{lemma} $B^G_\vphi(\dot{e}a,f_1)=B^G_\vphi(\dot{e}a,f)$ for all $a\in A_e=Z$.
\end{lemma}

\begin{proof} Since $f_1\in C_c^\infty(G;\omega_\pi)$ then we know that as above
\[
B^G_\vphi(\dot{e}a,f_1)=\tilde{\vphi}^G_G(e)\omega_\pi(a)W^{f_1}(e).
\]
But we have
\[
W^{f_1}(e)=\int_U f_1(x)\psi^{-1}(x)\ dx.
\]
To compute this we must decompose $x\in U$ into $x=g_1c$ with $g_1\in G^d$ and $c\in Z$. But now we have a unique such decomposition, namely $g_1=x$ and $c=e$. So  $f_1(x)=f_0(x)$
and since $W^{f_0}(e)=1$, we have
$B^G_\vphi(\dot{e}a,f_1)=\tilde{\vphi}^G_e(e)\omega_\pi(a)=B^G_\vphi(\dot{e}a,f)$
 for all $a\in Z$.
 \end{proof}

We now have 
$B^G_\vphi(\dot{e}a, f-f_1)=0$
for all $a\in A_e$. We can begin to apply our non-relevant cell  lemmas.  
We have $C_r(\dot{e})=ZU\subset C(e)=AU$ and  $\Omega^\circ_e= \Omega_e-C(e)=G-AU$. Then by Lemma \ref{G.3} there is $f'_2\in C_c^\infty(\Omega_e^\circ;\omega_\pi)$ such that
$B^G_\vphi(g,f-f_1)=B^G_\vphi(g,f'_2)$
for all $g\in G$.

We next want to apply Lemma \ref{G.4} to move up to the next cells that support Bessel functions. To that end, in the notation of Lemma \ref{G.4}, we let
\[
\Omega_1=\bigcup_{{w\in B(G)}\atop{w\neq e}} \Omega_w =\bigcup_{{w'\in B(G)}\atop{d_B(w',e)=1}}\Omega_{w'}\quad\text{and}\quad \Omega_0=\Omega_e^\circ=G-C(e).
\]
Then by Lemma \ref{G.4} there exists $f_2\in C_c^\infty (\Omega_1;\omega_\pi)$ such that for an appropriate $\vphi$ we have
\[
B^G_\vphi(g,f_2)=B^G_\vphi(g,f'_2)=B^G_\vphi(g,f-f_1)
\]
for all $g\in G$.  This can also then be written as
\[
B^G_\vphi(g,f)=B^G_\vphi(g,f_1)+B^G_\vphi(g,f_2)
\]
for all $g\in G$. 
 Following Jacquet \cite{J12}, we then use a partition of unity argument to write
 \[
 f_2=\sum f_{w'}\quad\text{with}\quad f_{w'}\in C_c^\infty(\Omega_{w'};\omega_\pi).
 \]
Therefore, for any $w\in B(G)$ we will have
\[
B^G_\vphi(\dw a,f)=B^G_\vphi(\dw a,f_1)+\sum_{d_B(w',e)=1} B^G_\vphi(\dw a,f_{w'})
\]
for $a\in A_w$. 

So we are left with analyzing $B^G_\vphi(\dw a,f_1)$.
Since $w\in B(G)$ we can write $w=w_\ell w^L_\ell$ for some Levi subgroup $L\subset G$. Then $w=w^G_L$ and $a\in A_w=Z_L$. 

By definition
\[
B^G_\vphi(\dw a,f_1)=\int_{U_{\dw a}\backslash U}\int_U f_1(x\dw au)\vphi({^tu}\tw_L a'u)\psi^{-1}(x)\psi^{-1}(u)\ dxdu.
\]
Since $a\in A_w$ we have $U_{\dw a}=U_\dw$, so
\[
B^G_\vphi(\dw a,f_1)=\int_{U_{\dw}\backslash U}\int_U f_1(x\dw au)\vphi({^tu}\tw_L a'u)\psi^{-1}(x)\psi^{-1}(u)\ dxdu
\]
and by Lemma \ref{S.6} we can write this as
  \[
  B^G_\vphi(\dw a;f_1)=\int_{U_\dw\bs U_w^+}\left[\int_{U_w^-}\int_U f_1(x\dw au^-)\vphi({^tu^-}\ {^tu^+}\tw_L a'u^+u^-)\psi^{-1}(x)\psi^{-1}(u^-)\ dxdu^-\right]\ du^+.
  \]

To insert our definition of $f_1$ we must decompose
\[
x\dw au^-=g_1c
\]
with $g_1\in G^d$ and $c\in A_e=Z$.  Following Jacquet, write this as 
\[
g_1=x\dw au^-c^{-1}=x\dw ac^{-1}u^-
\]
using that $c\in Z$. Since $g_1\in G^d$ we have 
\[
1=\det(g_1)=\det(x\dw ac^{-1}u^-)=\det(ac^{-1}).
\]
By definition, this says $b=ac^{-1}\in A^e_w=SL_n(F)\cap Z_L$. So $g_1=x\dw bu$ with $b\in A^e_w$ such that $a=bc$. Thus
\[
f_1(x\dw au^-)=\sum_{a=bc}f_0(x\dw bu^-)\omega_\pi(c).
\]

For the following computation, note that since $c\in Z$, if we decompose $A_w=ZA'_w$ so that the first entry of the elements of $A'_w$ is $1$ then $a'=(bc)'=b'$.

If we insert this expression for $f_1$ into our formula for $B^G_\vphi(\dw a,f_1)$ we obtain
\[
\begin{aligned}
B^G_\vphi &(\dw a, f_1)=\int_{U_{\dw}\backslash U}\int_U f_1(x\dw au)\vphi({^tu}\tw_L a'u)\psi^{-1}(x)\psi^{-1}(u)\ dxdu\\
&=\int_{U_\dw\backslash U_w^+}\left[\int_{U\times U_w^-} f_1(x\dw au^-)\vphi({^tu^-}\ {^tu^+}\tw_L a'u^+u^-)\psi^{-1}(x)\psi^{-1}(u^-)\ dxdu^-\right]du^+\\
&=\int_{U_\dw\backslash U_w^+}\left[\int_{U\times U_w^-}\sum_{a=bc}f_0(x\dw bu^-)\omega_\pi(c)   \vphi({^tu^-}\ {^tu^+}\tw_L a'u^+u^-)\psi^{-1}(x)\psi^{-1}(u^-)\ dxdu^-\right]du^+\\
&=\sum_{a=bc} \omega_\pi(c) \int_{U_\dw\backslash U_w^+}\left[\int_{U\times U_w^-}f_0(x\dw bu^-)   \vphi({^tu^-}\ {^tu^+}\tw_L b'u^+u^-)\psi^{-1}(x)\psi^{-1}(u^-)\ dxdu^-\right]du^+\\
&=\sum_{a=bc}\omega_\pi(c)\int_{U_{\dw}\backslash U}\int_U f_0(x\dw bu)\vphi({^tu}\tw_L b'u)\psi^{-1}(x)\psi^{-1}(u)\ dxdu\\
&=\sum_{a=bc}\omega_\pi(c)B^G_{\vphi}(\dw b,f_0).
\end{aligned}
\]
 This completes the proof of Proposition \ref{prop1}. 
\end{proof}

The last term in the expression in Proposition \ref{prop1} depends only on $\pi$ through $\omega_\pi$ and otherwise depends only on the auxiliary function $f_0$. So, in the proof of stability, these terms will always be the same for $\pi_1$ and $\pi_2$ as long as $\omega_{\pi_1}=\omega_{\pi_2}$ and we take the same auxiliary function $f_0$.

\subsubsection{The contribution of the small cell of $M$}

Take $M$ the Levi of a proper parabolic subgroup $P_M=MN_M$ of $G=GL_n$ corresponding to the partition $(n_1,\dots, n_t)$ of $n$. So $M\simeq GL_{n_1}\times\cdots\times GL_{n_t}$ and can be viewed  as block diagonal matrices 
\[
M=\left\{m=\bpm m_1\\&\ddots\\ & &m_t\epm\right\}\simeq GL_{n_1}\times \cdots\times GL_{n_t}.
\]
Let $w'=w_\ell w^M_\ell$ be the corresponding element of $B(G)$.
Then 
\[
A_{w'}=Z_M=\left\{\bpm a_1I_{n_1}\\ & \ddots \\ & & a_tI_{n_t} \\ \epm\right\}\simeq GL_1^t.
\]
 We have $U_M=U_{w'}^+$.

We begin with $h\in C_c^\infty(M;\omega_\pi)$ where the $\omega_\pi$ is a character of the center $Z$ of $G$, not that of $M$.   
 
We begin with the small cell of $M$.  We have $M^d= SL_{n_1}(F)\times\cdots\times SL_{n_t}(F)$ and 
$A_{w'}^{w'}= Z_M\cap M^d$. $A_{w'}^{w'}$ is finite and consists of $n_i^{th}$ roots of unity in the $i^{th}$ block of $M$.

We take $h_0\in C_c^\infty(M;\omega_\pi)$ such that $B^M(\dot{e}_M,h_0)=B^M(\dot{e}, h_0)=\frac{1}{\kappa_M}$, where $\kappa_M=|Z\cap A^{w'}_{w'}|$, and $B^M(b,h_0)=0$ for $b\in A^{w'}_{w'}$ and  $b\notin Z\cap A^{w'}_{w'}$. Note that
\[
B^M(\dot{e},h_0)=\int_{U_{M,e}^-}\int_{U_M} h_0(xu^-)\psi^{-1}(x)\psi^{-1}(u^-)\ dxdu^-=\int_{U_M} h_0(x)\psi^{-1}(x)\ dx=W^{h_0}(e)
\]
and in fact, for any $b\in Z\cap A_{w'}^{w'}$
\[
B^M(b,h_0)=W^{h_0}(b)=\omega_\pi(b) W^{h_0}(e).
\]

To define $h_1$ we (partially) decompose $M$ into $M^d$ and $A_{w'}=Z_M$. $M^d\cap Z_M=A^{w'}_{w'}$ is finite.  We define
\[
h_1(m)=\sum_{m=m'c} h_0(m')B^M(c,h)
\]
with $m'\in M^d$ and $c\in Z_M=A_{w'}$. 
Note that if $\det(m_i)$ is not a $n_i^{th}$ power on each block, then $h_1(m)=0$. 

\begin{prop}\label{prop2} $B^M_\vphi(a,h_1)=B^M_\vphi(a,h)$ for all $a\in Z_M=A_{w'}$.
\end{prop}

\begin{proof} Note that $U_{\dw'}=U_{M, \dot{e}_M}\subset U_M.$ Now we have, for $a\in Z_M$,
\[
B^M_\vphi(\dot{e}_M a,h_1)=B^M_\vphi(a, h_1)=\int_{U_{M,\dot{e}_M}\backslash U_M} \int_{U_M} h_1(xau)\vphi({^tu}\tw_M a'u)\psi^{-1}(xu)\ dxdu.
\]
Since the small cell is closed in $M$, the $M$-version of Lemma \ref{S.7} applies and we can write
\[
B^M_\vphi(\dot{e}_M a,h_1)=\tilde{\vphi}^M_M( a')\int_{U_M} h_1(xa)\psi^{-1}(x)\ dx=\tilde{\vphi}^M_M(a')B^M(a,h_1)
\]
for all sufficiently large $\vphi$ depending on $h_1$. 

We now substitute the definition of $h_1$.
We must decompose $xa=m'c$ with $m'\in M^d$ and $c\in Z_M$. Since $U_M\subset M^d$ if we decompose $a=bc$ with $b\in M^d\cap Z_M=A^{w'}_{w'}$ and $c\in Z_M=A_{w'}$ then there are only a finitely many possible $b$. We can write $xa=xbc$ with $xb\in M_e$. We must also decompose $a'=(bc)'=b'c'$ where $b'$ and $c'$ are the components of $b$ and $c$ in $A'_{w'}=Z'_M$, that is, having first coordinate $1$. We then arrive at
\[
\begin{aligned}
B^M_\vphi(\dot{e}_M a,h_1)& =\tilde{\vphi}^M_M( a')\int_{U_M} \sum_{a=bc} h_0(xb) B^M(c,h)\psi^{-1}(x)\ dx\\
&=\sum_{a=bc}B^M(c,h)\left[ \tilde{\vphi}^M_M( b'c')\int_{U_M} h_0(xb) \psi^{-1}(x)\ dx\right]\\
&=\sum_{a=bc}\tilde{\vphi}^M_M(b'c')B^M(c,h)B^M(b,h_0).
\end{aligned}
\]
Now, by construction, $B^M(b,h_0)=0$ unless $b\in Z\cap A^{w'}_{w'}$ and in this case $B^M(b,h_0)=\omega_\pi(b)B^M(e,h_0)=\frac{\omega_\pi(b)}{\kappa_M}$.  Then the above becomes
\[
\begin{aligned}
B^M_\vphi(\dot{e}_M a,h_1)&=\frac{1}{\kappa_M}\sum_{{a=bc}\atop{b\in Z\cap A^{w'}_{w'}} }\tilde{\vphi}^M_M(b'c')B^M(c,h)\omega_\pi(b)\\
&=\frac{1}{\kappa_M}\sum_{{a=bc}\atop{b\in Z\cap A^{w'}_{w'}} }\tilde{\vphi}^M_M(b'c')B^M(bc,h)\\
&=\frac{1}{\kappa_M}\sum_{{a=bc}\atop{b\in Z\cap A^{w'}_{w'}} }B^M_\vphi(a,h)
\end{aligned}
\]
where we now need to take $\vphi$ such that the analogue of Lemma \ref{S.7} holds for $h$ and $h_1$.
Since in the decomposition $a=bc$ both $a$ and $c$ are in $Z_M=A_{w'}$, for any choice of $b\in Z\cap A^{w'}_{w'}$ there is a decomposition. So there are $\kappa_M$ terms in the sum. Thus we conclude
\[
B^M_\vphi(a, h_1)=B^M_\vphi(a, h)
\]
as desired.
\end{proof}

We next want to understand  $B^M_\vphi(\dw_\ell^M a,h_1)$ for $a\in A_{w_\ell^M}=A_{w_\ell}=A$ as this is what will occur in expression for the local coefficient. We are looking for uniform smoothness in certain directions.

First, since $a\in A=A_{w_\ell^M}$ we have $U_{\dw^M_\ell a}=U_{\dw_\ell^M}\subset U_{w_\ell^M}^+=\{ e_M\}$ since we are working on $M$ and so we compute $U_{w_\ell^M}^+$ in $U_M$. Therefore
\[
B^M_\vphi(\dw_\ell^M a,h_1)=\int_{U_M\times U_M} h_1(x\dw^M_\ell a u)\vphi({^tu} a'u)\psi^{-1}(xu)\ dxdu.
\]
If we decompose $a$ as $a=za'$ in accord with the decomposition $A=ZA'$, so the first coordinate of $a'$ is $1$, then
 \[
 \begin{aligned}
B^M_\vphi(\dw_\ell^M a,h_1)&=\int_{U_M\times U_M} h_1(x\dw^M_\ell za' u)\vphi({^tu}  a'u)\psi^{-1}(xu)\ dxdu
=\omega_\pi(z)B^M_\vphi(\dw^M_\ell a',h_1)
\end{aligned}
\] 

We next insert the definition of $h_1$
\[
h_1(m)=\sum_{m=m'c} h_0(m')B^M(c,h)
\]
with $m'\in M^d$ and $c\in Z_M$. Thus in our integral we must write
\[
x\dw^M_\ell a'u=m_1c \quad \text{or}\quad x \dw^M_\ell a'c^{-1}u\in M^d.
\]
For $M=GL_{n_1}\times\cdots\times GL_{n_t}$ we have $M^d=SL_{n_1}\times\cdots\times SL_{n_t}$ and $x, u, \dw^M_\ell\in M^d$, thus it is enough to decompose $a'=bc$ with $b\in A\cap M^d$ and $c\in Z_M$. The intersection is finite, so we have at most a finite number of such decompositions (and for some $a$ there may be no such decomposition). We therefore have
\[
h_1(x\dw^M_\ell a'u)=\sum_{a'=bc}h_0(x \dw^M_\ell bu)B^M(c, h).
\]
If we now decompose $b=z_bb'$, with $z_b\in Z$ and $b'\in A'$, and $c=z_cc'$, with $z_c\in Z$ and $c'\in Z'_M$, then $a'=z_bz_cb'c'$ implies $a'=b'c'$ and $z_bz_c=1$. Since $h, h_0\in C_c^\infty(M;\omega_\pi)$ we have 
\[
h_0(x \dw^M_\ell bu)B^M(c, h)=\omega_\pi(z_b)\omega_\pi(z_c)h_0(x \dw^M_\ell b'u)B^M(c', h)=h_0(x \dw^M_\ell b'u)B^M(c', h).
\]

Therefore we have
\[
\begin{aligned}
B^M_{\vphi}(\dw^M_\ell a',h_1)&=\int_{U\times U}\sum_{a'=bc} h_0(x\dw^M_\ell b'u)B^M(c',h)\vphi({^tu} b'c'u)\psi^{-1}(xu)\ dxdu\\
&=\sum_{a'=bc}B^M(c',h)\int_{U\times U}h_0(x\dw^M_\ell b'u)\vphi({^tu} b'uc')\psi^{-1}(xu)\ dxdu\\
&=\sum_{a'=bc}B^M(c',h)B^M_{\vphi_{c'}}(\dw^M_\ell b',h_0)
\end{aligned}
\]
where, as before, $\vphi_{c'}(m)=\vphi(mc')$ for $c'\in Z'_M$.

\begin{prop} \label{prop3}  For $a\in A^{w'}_{w_\ell}A_{w'}\subset A_{w_\ell}=A$ let $a=bc$ be a fixed decomposition. All such decompositions are of the form $a=(b\zeta^{-1})(\zeta c)$ with $\zeta\in A^{w'}_{w'}$, which is a finite set of matrices with appropriate roots of unity on the diagonal. Further
write $c=c'z$ with $c'\in Z'_M=A'_{w'}$ and $z\in Z$. Then $B^M_\vphi(\dw^M_\ell a, h_1)=\omega_\pi(z)B^M_\vphi(\dw^M_\ell bc', h_1)$ is uniformly smooth as a function of $c'$.
\end{prop}
\begin{pf} We have the expression
\[
B^M_{\vphi}(\dw^M_\ell a,h_1)=\sum_{a=bc}B^M(c,h)B^M_{\vphi_c}(\dw^M_\ell b,h_0)
\]
where we sum over all decompositions of $a=bc$. If we fix one such decomposition and replace the sum over the decompositions as a sum over $\zeta\in A^{w'}_{w'}$ this becomes
\[
B^M_{\vphi}(\dw^M_\ell a,h_1)=\sum_{\zeta}B^M(\zeta c,h)B^M_{\vphi_{\zeta c}}(\dw^M_\ell b\zeta^{-1},h_0).
\]
As $\vphi$ is a characteristic function depending on the absolute value of the entries and then entries of $\zeta$ are roots of unity, $\vphi_{\zeta c}=\vphi_c$. So this is
\[
B^M_{\vphi}(\dw^M_\ell a,h_1)=\sum_{\zeta}B^M(\zeta c,h)B^M_{\vphi_{ c}}(\dw^M_\ell b\zeta^{-1},h_0).
\]

Now
\[
B^M(\zeta c,h)=\int_U h(x\zeta c)\psi^{-1}(x) \ dx=\omega_\pi(\zeta_1z)\int_U h(x\zeta'c')\psi^{-1}(x) \ dx
\]
where $\zeta=diag(\zeta_1I_{n_1},\dots, \zeta_tI_{n_t})$ and $\zeta'=diag(I_{n_1}, \zeta_2\zeta_1^{-1}I_{n_2},\dots,\zeta_t\zeta_1^{-1}I_{n_t})$.
Since we have $h\in C_c^\infty(M;\omega_\pi)$ and the small cell $C^M(e_M)=AU=ZA'U$ is closed in $M$ we see that there are compact subsets $U_1\subset U$ and $K'\subset A'$ such that $h(x\zeta'c')\neq 0$ implies $x\in U_1$ and $\zeta'c'\in K'$. In fact, since $Z'_M\subset A'$ is closed, there is a further subset $K''\subset Z'_M$ such that $h(x\zeta'c')\neq 0$ implies $x\in U_1$ and $\zeta'c'\in K''$ or $c'\in (\zeta')^{-1}K''$. Therefore, writing $a=bc=bc'z$ as above, we see
\[
B^M_{\vphi}(\dw^M_\ell a,h_1)=\omega_\pi(z)\sum_{\zeta}B^M(\zeta c',h)B^M_{\vphi_{c}}(\dw^M_\ell b\zeta^{-1},h_0).
\]
is zero unless $c'\in \bigcup_{\zeta'}(\zeta')^{-1}K''$. Thus we have compact support on $Z'_M$ depending only on $h$ through $K''$ and $T_{M^1}\cap Z_M$, hence independent of $a$ or $b$. 

The dependence on $c'$ is in the argument of $h$ and in the scaling of $\vphi$. As $h$ is smooth and its support is compact in $c'$ there will be a uniform open compact subgroup $C_1\subset Z'_M$ such that $h(x\zeta c' c_1)=h(x\zeta c')$ for $c_1\in C_1$ and all $x\in U_1$ and $c'\in Z'_M$. Shrinking if necessary, we can take $C_1\subset Z'_M(\fo)$, that is, so that the entries of $c_1$ are all units. But then in the scaling $\vphi_{cc_1}$ this will not effect the absolute values of the entries and hence the value of $\vphi_c$. 

So there exists a compact open subgroup $C_1\subset Z'_M$ such that 
\[
B^M_\vphi(\dw^M_\ell bzc'c_1,h_1)=B^M_\vphi(\dw^M_\ell bzc',h_1)
\]
for all $a=bc\in A_{M^d}$ and $c_1\in C_1$. Since $B^M(\dw^M_\ell a, h_1)$  vanishes off of $A_{M^1}$, this holds for all $a\in A$.
\end{pf}

\subsubsection{Lifting to $G$}

Let $M\subset G$ be the (proper) Levi subgroup from the previous section. Let $h_1\in C_c^\infty(M;\omega_\pi)$ be as above with 
 $w'=w_\ell w^M_\ell\in B(G)$.
 
 Let $f_{w'}\in C_c^\infty(\Omega_{w'};\omega_\pi)$, and $h=h_{f_w'}\in C_c^\infty(M;\omega_\pi)$ as in Lemma \ref{GM1}. 
 Construct $h_1$ as in the previous section  such that $B^M_\vphi(a,h_1)=B^M_\vphi(a,h)$ for all $a\in Z_M=A_{w'}$.
 As in Lemma \ref{GM1} choose $f_1\in C_c^\infty(\Omega_{w'};\omega_\pi)$ such that 
\[
\int_{U_{(w')^{-1}}^-\times U_{w'}^-}f_1(x^-\dw'mu^-)\psi^{-1}(x^-u^-)\ dx^-du^- = h_1(m).
\]
with $h_1\in C_c^\infty(M;\omega_\pi)$. This is possible by the surjectivity of the map $f\mapsto h_f$ of Lemma \ref{GM1}. 
Then from Proposition \ref{GM2} we know that for all  $L$ with $A\subset L\subset M$ we have
\[
B^G_\vphi(\dw^G_L a,f_1)= B^M_\vphi(\dw^M_L a, h_1)
\]
for $a\in A_{w^G_L}=A_{w^M_L}=Z_L$.

If we apply this with  $L=M$  we have
\[
B^G_\vphi(\dw' a, f_1)=B^M_\vphi(\dot{e}_M a, h_1)=B^M_\vphi(a,h)=B^G_\vphi(\dw' a,f_{w'}).
\]
Therefore $f_{w'}-f_1\in C_c^\infty(\Omega_{w'};\omega_\pi)$ such that $B^G_\vphi(\dw' a,f_{w'}-f'_1)=0$ for all $a\in A_{w'}$. We can now apply 
Lemma \ref{G.2}, Lemma \ref{G.3} and finally Lemma \ref{G.4}, plus a partition of unity argument,  to find a family $\{f_{w''}\}$ parametrized by $w''\in B(G)$ such that $w''>w'$ and $d_B(w'',w')=1$ so that $f_{w''}\in C_c^\infty(\Omega_{w''}; \omega_\pi)$ and 
that for any $w\in B(G)$ 
we have
\[
B^G_\vphi(\dw a, f_{w'})=B^G_\vphi(\dw a, f_1) +\sum_{w''} B^G_\vphi(\dw a,f_{w''})
\]
for all $a\in A_w$. 
Now for each $f_{w''}$ we have $w''=w^G_{M''}$ and we will be able to perform an  induction.

If we apply the above equality with $L=A$ corresponding to the big cells, we have
\[
B^G_\vphi(\dw_\ell a,f_1)=B^M_\vphi(\dw^M_\ell a, h_1)
\] 
for all $a\in A$.  By Proposition \ref{prop3}, if we decompose $A^{w'}_{w_\ell}A_{w'}$ (with finite intersection), as  $a=bc$,  
then
 \[
B^G_\vphi(\dw_\ell a, f_1)=B^M_{\vphi}(\dw^M_\ell a,h_1)=B^M_{\vphi}(\dw^M_\ell bc, h_1)=\omega_\pi(z)B^M_\vphi(\dw^M_\ell bc',h_1)
\]
is compactly supported  in $c'\in Z'_M$ and thus  $ B^G_{\vphi}(\dw_\ell bc,f_1)$ is uniformly smooth as a function of $c'\in Z'_M=A_{w'}$. This gives us the ``uniform smoothness'' on $A'_{w"}$ that we will need for stability.

\subsubsection{Summary} For use below, let us state the results in this subsection formally.

\begin{prop}\label{prop4} Let $w'=w_\ell w^M_\ell \in B(G)$ and $f_{w'}\in C_c^\infty(\Omega_{w'};\omega_\pi)$. Then there exists $f_{1,w'}\in C_c^\infty(\Omega_{w'};\omega_\pi)$ such that 
\begin{enumerate}
\item[(i)] There exists a family $\{f_{w''}\}$ parametrized by $w''\in B(G)$ such that $w''>w'$ and $d_B(w'',w')=1$ so that $f_{w''}\in C_c^\infty(\Omega_{w''}; \omega_\pi)$ and 
for any $w\in B(G)$ 
we have
\[
B^G_\vphi(\dw a, f_{w'})=B^G_\vphi(\dw a, f_{1,w'}) +\sum_{w''} B^G_\vphi(\dw a,f_{w''})
\]
for all $a\in A_w$. 
\item[]
\item[(ii)] $B^G_\vphi(\dw_\ell a, f_{1,w'})=\omega_\pi(z)B^G_{\vphi}(\dw_\ell bc',f_{1,w'})$ is uniformly smooth as a function of $c'\in Z'_M=A'_{w'}$.
\end{enumerate}
\end{prop}

This is analogue of Jacquet's Proposition 2.1 of \cite{J12} for our partial Bessel integrals.  In place of our uniform smoothness statement for 
$B^G_\vphi(\dw_\ell a, f_{1,w'})$ he has the beginning of his germ expansion. The presence of our cutoff function $\vphi$ keeps us from decomposing this Bessel integral into two pieces, one along the cell and one transverse to the cell, as Jacquet does.

\subsection{Supercuspidal stability for $GL_n$, II: Uniform smoothness}

We can now establish the ``uniform smoothness'' result we will need for our proof of stability.

\subsubsection{Setting up the induction}
 We consider the Bruhat order on $W(G)$ restricted to $B(G)$. We recall that in the Bruhat order, $w'<w$ if $w\neq w'$ and $C(w')\subset \overline{C(w)}$, so the Bruhat cell for $w'$ is contained in the closure of the Bruhat cell for $w$.  As we have noted, $B(G)$ is in bijection with the set of standard Levi subgroups $\cL(G)=\{M \mid A\subset M\subset G\}$ since $w\in B(G)$ iff $w=w_\ell w^M_\ell$ with $M\in\cL(G)$. Then the Bruhat order on $B(G)$ corresponds to the reverse containment order on $\cL(G)$.

 We will induct on $d_B(w,e)$.
 
\subsubsection{The induction}  Let us make explicit the first two steps of the induction. The first step is essentially Proposition \ref{prop1}. Fix an auxiliary  $f_0\in C_c^\infty(G;\omega_\pi)$ with $W^{f_0}(e)=1$.
 
 Let $f\in \cM(\pi)\subset C_c^\infty(G,\omega_\pi)$ also with $W^f(e)=1$.
;
 \noindent{\bf Step 1.} {\it There exists $f_{1,e}\in C_c^\infty(G;\omega_\pi)$ and for each $w'\in B(G)$ with $d_B(w',e)=1$ there exist a function $f_{w'}\in C_c^\infty(\Omega_{w'};\omega_\pi)$ such that  for sufficiently large $\vphi$
 \begin{enumerate}
 \item[(i)]  
for any $w\in B(G)$ 
we have
\[
B^G_\vphi(\dw a, f)=B^G_\vphi(\dw a, f_{1,e}) +\sum_{d_B(w',e)=1} B^G_\vphi(\dw a,f_{w'})
\]
for all $a\in A_w$;
\item[] 
\item[(ii)] for each $w\in B(G)$, $B^G_\vphi(\dw a,f_{1,e})$ depends only upon the auxiliary function $f_0$ and $\omega_\pi$ for all $a\in A_w$.
\end{enumerate}}

  This is simply a restatement of Proposition \ref{prop1} in the previous section.
  
  For the second step, essentially the induction step, we apply Proposition  \ref{prop4} to each $f_{w'}$ above.
 
\noindent{\bf Step $\mathbf 2'$.}  {\it For each $f_{w'}$, there exists $f_{1,w'}\in C_c^\infty(\Omega_{w'};\omega_\pi)$ such that  for sufficiently large $\vphi$
\begin{enumerate}
\item[(i)] there exists a family $\{f_{w', w''}\}$ parametrized by $w''\in B(G)$ with $w''>w'$ and $d_B(w'',w')=1$ so that $f_{w',w''}\in C_c^\infty(\Omega_{w''}; \omega_\pi)$ and 
for any $w\in B(G)$ 
we have
\[
B^G_\vphi(\dw a, f_{w'})=B^G_\vphi(\dw a, f_{1,w'}) +\sum_{d_B(w'',w')=1} B^G_\vphi(\dw a,f_{w',w''})
\]
for all $a\in A_w$;
\item[]
\item[(ii)] $B^G_\vphi(\dw_\ell a, f_{1,w'})=\omega_\pi(z)B^G_{\vphi}(\dw_\ell bc',f_{1,w'})$ is uniformly smooth as a function of $c'\in Z'_M$.
\end{enumerate}} 

If we combine Step 1 and Step $2'$ we have that  for any $w\in B(G)$ 
we have
\[
B^G_\vphi(\dw a, f)=B^G_\vphi(\dw a,f_{1,e})+\sum_{d_B(w',e)=1} B^G_\vphi(\dw a, f_{1,w'}) +\sum_{d_B(w'',w')=d_B(w',e)=1} B^G_\vphi(\dw a,f_{w',w''})
\]
for all $a\in A_w$. 
We note that $d_B(w'',w')=1$ and $d_B(w',e)=1$ is equivalent to $d_B(w'',e)=2$. So if we set
\[
f_{w''}=\sum_{d_B(w'',w')=1} f_{w',w''}\in C_c^\infty(\Omega_{w''};\omega_\pi)
\]
then we can combine Step 1 and  Step $2'$ as follows.

\noindent{\bf Step 2.} {\it For $f\in \cM(\pi)$ with $W^f(e)=1$ there exists $f_{1,e}\in C_c^\infty(G;\omega_\pi)$ and for each $w'\in B(G)$ with $d_B(w',e)=1$ there exist $f_{w',1}\in C_c^\infty(\Omega_{w'};\omega_\pi)$ and for each $w''\in B(G)$ with $d_B(w'',e)=2$ an element $f_{w''}\in C_c^\infty(\Omega_{w''};\omega_\pi)$ such that for sufficiently large $\vphi$ 
\begin{enumerate}
\item[(i)]  
for any $w\in B(G)$ 
we have
\[
B^G_\vphi(\dw a, f)=B^G_\vphi(\dw a,f_{1,e})+\sum_{d_B(w',e)=1} B^G_\vphi(\dw a, f_{1,w'}) +\sum_{d_B(w'',e)=2} B^G_\vphi(\dw a,f_{w''})
\]
for all $a\in A_w$;
\item[]
\item[(ii)]  for each $w\in B(G)$, $B^G_\vphi(\dw a,f_{1,e})$ depends only upon the auxiliary function $f_0$ and $\omega_\pi$ for all $a\in A_w$.
\item[]
\item[(iii)] $B^G_\vphi(\dw_\ell a, f_{1,w'})=\omega_\pi(z)B^G_{\vphi}(\dw_\ell bc',f_{1,w'})$ is uniformly smooth as a function of $c'\in A'_{w'}$.
\end{enumerate}}

Inductively we can now show the following.

\noindent{\bf General Step.} {\it Let $f\in \cM(\pi)$ with $W^f(e)=1$. Let  $m$ be an integer with $1\leq m\leq d_B(w_\ell,e)+1$. Then 
\begin{enumerate}
\item[(a)] there exits $f_{1,e}\in C_c^\infty(G;\omega_\pi)$
\item[(b)] for each $w'\in B(G)$ with $1\leq d_B(w',e)<m$ there exist $f_{1,w'}\in C_c^\infty(\Omega_{w'};\omega_\pi)$
\item[(c)] for each $w''\in B(G)$ with $d_B(w'',e)=m$ there is an element $f_{w''}\in C_c^\infty(\Omega_{w''};\omega_\pi)$ 
\end{enumerate}
such that for appropriate $\vphi$ 
\begin{enumerate}
\item[(i)] for any $w\in B(G)$ 
we have
\[
B^G_\vphi(\dw a, f)=B^G_\vphi(\dw a,f_{1,e})+\sum_{1\leq d_B(w',e)<m} B^G_\vphi(\dw a, f_{1,w'}) +\sum_{d_B(w'',e)=m} B^G_\vphi(\dw a,f_{w''})
\]
for all $a\in A_w$;
\item[]
\item[(ii)]  for each $w\in B(G)$, $B^G_\vphi(\dw a,f_{1,e})$ depends only on the auxiliary function $f_0$ and $\omega_\pi$ for all $a\in A_w$;
\item[]
\item[(iii)] for each $w'\in B(G)$ with $1\leq d_B(w',e)<m$, $B^G_\vphi(\dw_\ell a, f_{1,w'})=\omega_\pi(z)B^G_{\vphi}(\dw_\ell bc',f_{1,w'})$ is uniformly smooth as a function of $c'\in A'_{w'}$.
\end{enumerate}}
  
 \begin{pf} The first step in the induction is Step 2 above. The induction step is done by applying Proposition \ref{prop4} to each $w''$ with $d_B(w'',e)=m$ to attain the analogue of Step $2'$ and then argue as above in passing from Step $2'$ to Step 2.
  \end{pf}
  
  This is our analogue of Jacquet's Proposition 3.1 of \cite{J12}.
  
\subsubsection{Uniform smoothness}  If we take the case of $m=d_B(w_\ell,e)+1$ we arrive at the proposition that we need to prove stability.  This is our analogue of the main theorem of Jacquet \cite{J12}, his germ expansion for his Kloosterman orbital integrals. Recall that we have fixed an auxiliary $f_0\in C_c^\infty(G;\omega_\pi)$ with $W^{f_0}(e)=1$. 
  
 \begin{prop}\label{unifsmooth}Let $f\in \cM(\pi)$ with $W^f(e)=1$. Then 
\begin{enumerate}
\item[(a)] there exits $f_{1,e}\in C_c^\infty(G;\omega_\pi)$
\item[]
\item[(b)] for each $w'\in B(G)$ with $1\leq d_B(w',e)$ there exist $f_{1,w'}\in C_c^\infty(\Omega_{w'};\omega_\pi)$
\end{enumerate}
such that for sufficiently large $\vphi$ 
\begin{enumerate}
\item[(i)]  we have
\[
B^G_\vphi(\dw_\ell a, f)=B^G_\vphi(\dw_\ell a,f_{1,e})+\sum_{1\leq d_B(w',e)} B^G_\vphi(\dw_\ell a, f_{1,w'}) 
\]
for all $a\in A$;
\item[]
\item[(ii)]  
 $B^G_\vphi(\dw_\ell a,f_{1,e})$ depends only upon the auxiliary function $f_0$ and $\omega_\pi$ for all $a\in A$;
\item[]
\item[(iii)] for each $w'\in B(G)$ with $1\leq d_B(w',e)$ we have $B^G_\vphi(\dw_\ell a, f_{1,w'})=\omega_\pi(z)B^G_{\vphi}(\dw_\ell bc',f_{1,w'})$ is uniformly smooth as a function of $c'\in A'_{w'}$.
\end{enumerate}
\end{prop}

We have stated this proposition for the long Weyl element $w_\ell$ since this is what we need for stability. Statements (i) -- (iii) hold for any $w\in B(G)$ and $a\in A_w$. In this case, if $w<w'$  the terms $B^G_\vphi(\dw a,f_{1,w'})=0$ since $f_{1,w'}\in C_c^\infty(\Omega_{w'};\omega_\pi)$ and hence vanishes on $C(w)$.
  
\subsection{Supercuspidal stability for $GL_n$, III: Stability}
Since the stability involves twisting by a highly ramified character, we need to know how the partial Bessel function varies under twisting. This is elementary and follows from the formula \eqref{bf1}.

\begin{lemma} \label{twistedbf} Let $\chi$ be a character of $F^\times$, viewed as a character of $G=GL_n(F)$ through composition with the determinant.
\begin{enumerate}
\item Let $w\in B(G)$ support a Bessel function. Then for all $g\in C(w)$ we have
\[
j_{\pi\otimes\chi,w}(g)=\chi(\det(g))j_{\pi,w}(g).
\]
\item For the partial Bessel function of Proposition \ref{integrep3}
\[
j_{\pi\otimes\chi,\dot{w}_\ell,\kappa}(\dot{w}_\ell a)=\chi(\det(a))j_{\pi,\dot{w}_\ell,\kappa}(\dot{w}_\ell a)
\]
for all $a\in A$.
\end{enumerate}
\end{lemma}

Now we prove  the
stability of local coefficients under twisting by a sufficiently
highly ramified character, and hence the stability of local
$\gamma$-factors in the supercuspidal case.

\begin{proof} We begin with $\pi_1$ and $\pi_2$ two supercuspidal representations of $G=GL_n(F)$ having the same central character $\omega=
\omega_{\pi_1}=\omega_{\pi_2}$.  Note that if $\chi$ is a character of $F^\times$, then the central character of $\pi_i\otimes\chi$ is $\omega\chi^n$.

We consider the difference $C_\psi(s,\pi_1\otimes\chi)^{-1}-C_\psi(s,\pi_2\otimes\chi)^{-1}$.   For $\chi$  sufficiently  ramified,  $\omega\chi^n$ will be  ramified; this is necessary for Proposition \ref{integrep3}. Then applying Proposition \ref{integrep3} to both local coefficients, we have a $\kappa_0$ so that for all  $\kappa\geq \kappa_0$ the representation \eqref{uir} holds for  all $\pi_1\otimes\chi$ and $\pi_2\otimes\chi$, and we can  expresses the difference as
\[
C_\psi(s,\pi_1\otimes\chi)^{-1}-C_\psi(s,\pi_2\otimes\chi)^{-1}=\gamma(\tfrac{ns}{2},\ \omega\chi^n,\ \psi)\omega(-1)\chi(-1)^n D_\chi(s),
\]
where
\[
\begin{aligned}
D_\chi(s)&=\int_{Z\backslash A} (j_{\pi_1\otimes\chi,\dot{w_\ell},\kappa}(\dot{w}_\ell a)-j_{\pi_2\otimes\chi,\dot{w}_\ell,\kappa}(\dot{w}_\ell a)) \omega\chi^n(a_1)^{-1}|a_1|^{-(n-1)(s-1)/2}\prod_{i=2}^n|a_i|^{(n+s+1-2i)/2}\ da\\
&=\int_{ A'} (j_{\pi_1\otimes\chi,\dot{w}_\ell ,\kappa}(\dot{w}_\ell a')-j_{\pi_2\otimes\chi,\dot{w}_\ell ,\kappa}(\dot{w}_\ell a')) \prod_{i=2}^n|a_i|^{(n+s+1-2i)/2}\ da'.
\end{aligned}
\]

Consider now the difference in Bessel functions. By Lemma \ref{twistedbf} we have
\[
j_{\pi_1\otimes\chi,\dot{w}_\ell,\kappa}(\dot{w}_\ell a')-j_{\pi_2\otimes\chi,\dot{w}_\ell,\kappa}(\dot{w}_\ell a')=\chi(\det(a'))(j_{\pi_1,\dot{w}_\ell,\kappa}(\dot{w}_\ell a')-j_{\pi_2,\dot{w}_\ell,\kappa}(\dot{w}_\ell a')).
\]
Now choose $f_1\in \cM(\pi_1)$ and $f_2\in \cM(\pi_2)$ such that $W^{f_1}(e)=W^{f_2}(e)=1$ and such that 
\[
j_{\pi_i,\dw_\ell,\kappa}(\dw_\ell a')=B^G_{\vphi_\kappa}(\dw_\ell a',f_i).
\]
 We may assume that $\kappa$ is large enough that Proposition \ref{unifsmooth} holds for both  $f_1$ and $f_2$ with the same auxiliary $f_0$, and drop $\kappa$  from the notation.
Then applying Proposition \ref{unifsmooth} we have
\[
\begin{aligned}
j_{\pi_1,\dot{w}_\ell,\kappa}(\dot{w}_\ell a')-j_{\pi_2,\dot{w}_\ell,\kappa}(\dot{w}_\ell a')&=B^G_\vphi(\dw_\ell a',f_1)-B^G_\vphi(\dw_\ell a', f_2)\\
&=(B^G_\vphi(\dw_\ell a',f_{1,1,e})-B^G_\vphi(\dw_\ell a',f_{2,1,e}))\\
&\quad\quad + \sum_{1\leq d_B(w',e)} (B^G_\vphi(\dw_\ell a', f_{1,1,w'}) -B^G_\vphi(\dw_\ell a', f_{2,1,w'}) )
\end{aligned}
\]

Since both $B^G_\vphi(\dw_\ell a',f_{1,1,e})$ and $B^G_\vphi(\dw_\ell a',f_{2,1,e})$ only depend on the common $f_0$ and $\omega_{\pi_1}=\omega_{\pi_2}$, these will cancel and we are left with
\[
j_{\pi_1,\dot{w}_\ell,\kappa}(\dot{w}_\ell a')-j_{\pi_2,\dot{w}_\ell,\kappa}(\dot{w}_\ell a')=\sum_{1\leq d_B(w',e)} (B^G_\vphi(\dw_\ell a', f_{1,1,w'}) -B^G_\vphi(\dw_\ell a', f_{2,1,w'}) ).
\]
Substituting this in the formula for $D_\chi(s)$ we find
\[
D_\chi(s)= \sum_{1\leq d_B(w',e)} \int_{A'}(B^G_\vphi(\dw_\ell a', f_{1,1,w'}) -B^G_\vphi(\dw_\ell a', f_{2,1,w'}) )\chi(\det(a'))\prod_{i=2}^n|a'_i|^{(n+s+1-2i)/2}\ da'.
\]
If we now utilize part (iii) of Proposition \ref{unifsmooth} we can write this as
\[
\begin{aligned}
D_\chi(s)= \sum_{1\leq d_B(w',e)} \int_{A^{w'}_{w_\ell}}&
\Big[\int_{A_{w'}'}(B^G_\vphi(\dw_\ell bc', f_{1,1,w'}) -B^G_\vphi(\dw_\ell bc', f_{2,1,w'}) )\\&\chi(\det(c'))\prod_{i=2}^n|c_i'|^{(n+s+1-2i)/2}\ dc'\Big] \chi(\det(b))\prod_{i=2}^n|b_i'|^{(n+s+1-2i)/2}\ db.
\end{aligned}
\]

 Again appealing to part (iii) of Proposition \ref{unifsmooth}, we have that the piece of the inner integrand
 \[
 (B^G_\vphi(\dw_\ell bc', f_{1,1,w'}) -B^G_\vphi(\dw_\ell bc', f_{2,1,w'}) )\prod_{i=2}^n|c_i'|^{(n+s+1-2i)/2}  
\]
is uniformly smooth as a function of $c'\in A'_{w'}$. Thus for $\chi$ sufficiently highly ramified we have
\[
\int_{A'_{w'} } (B^G_\vphi(\dw_\ell bc', f_{1,1,w'}) -B^G_\vphi(\dw_\ell bc', f_{2,1,w'}) )\prod_{i=2}^n|c_i'|^{(n+s+1-2i)/2} \chi(c')\ dc'=0.
\]

Taking $\chi$ sufficiently highly ramified that all inner integrals vanish, we may conclude that $D_\chi(s)=0$. Hence
\[
C_\psi(s,\pi_1\otimes\chi)=C_\psi(s,\pi_2\otimes\chi).
\]
which establishes Proposition \ref{scstab}.
\end{proof}

\vskip-.125truein
\end{document}